\documentclass[11pt,fullpage]{article}  
\usepackage[right = 2.5cm, left=2.5cm, top = 2.5cm, bottom =2.5cm]{geometry}
\usepackage{amsmath,amsfonts,amssymb}
\usepackage{amsthm}
\usepackage{verbatim}
\usepackage{alltt}  
\usepackage{epsfig}  
\usepackage{hyperref}
\usepackage[mathscr]{euscript}

\newcommand{\Real}{\mathbb R}

\newcommand{\Torus}{\mathbb T}

\newcommand{\norm}[1]{\left\lVert #1 \right\rVert}
\newcommand{\abs}[1]{\left\vert#1\right\vert}
\newcommand{\set}[1]{\left\{#1\right\}}

\newcommand{\grad}{\nabla}

\newcommand{\jap}[1]{\langle{#1}\rangle}

\numberwithin{equation}{section}

\begin{document}  

\newtheorem{defi}{Definition}[section]  
\newtheorem{rema}{Remark}[section]  
\newtheorem{remark}{Remark}[section]  
\newtheorem{prop}{Proposition}[section]  
\newtheorem{proposition}{Proposition}[section]  
\newtheorem{lem}{Lemma}[section]  
\newtheorem{lemma}{Lemma}[section]  
\newtheorem{theo}{Theorem}[section]  
\newtheorem{cor}{Corollary}[section]  
\newtheorem{conc}{Conclusion}[section]  

\title{On the stability threshold for the 3D Couette flow in Sobolev regularity} 

\author{J. Bedrossian, P. Germain, N. Masmoudi}

\maketitle

\begin{abstract} 
We study Sobolev regularity disturbances to the periodic, plane Couette flow in the 3D incompressible Navier-Stokes equations at high Reynolds number \textbf{Re}.  
Our goal is to estimate how the stability threshold scales in $\textbf{Re}$: the largest the initial perturbation can be while still resulting in a solution that does not 
transition away from Couette flow.  
In this work we prove that initial data which satisfies $\norm{u_{in}}_{H^\sigma} \leq \delta\textbf{Re}^{-3/2}$ for any $\sigma > 9/2$ and some $\delta = \delta(\sigma) > 0$ depending only on $\sigma$, 
is global in time, remains within $O(\textbf{Re}^{-1/2})$ of the Couette flow in $L^2$ for all time, and converges to the class of ``2.5 dimensional'' streamwise-independent solutions referred to as \emph{streaks} for times $t \gtrsim \textbf{Re}^{1/3}$. 
Numerical experiments performed by Reddy et. al. with ``rough'' initial data estimated a threshold of $\sim \textbf{Re}^{-31/20}$, which shows very close agreement with our estimate. 
\end{abstract} 

\setcounter{tocdepth}{1}
{\small\tableofcontents}

\section{Introduction}
\subsection{Presentation of the problem}

We consider the three-dimensional Navier-Stokes equation with inverse Reynolds number $\nu = \textbf{Re}^{-1}> 0$ 
\begin{equation*}
\partial_t v - \nu \Delta v + v \cdot \nabla v = - \nabla p \\
\end{equation*}
set on $\mathbb{T} \times \mathbb{R} \times \mathbb{T}$, in other words $v(t,x,y,z) \in \mathbb{R}^3$ and $p(t,x,y,z) \in \mathbb{R}$ are functions of $(t,x,y,z) \in \mathbb{R}_+ \times \mathbb{T} \times \mathbb{R} \times \mathbb{T}$ (the torus $\mathbb{T}$ is the periodized interval $[0,1]$).
The simplest non-trivial stationary solution is the \emph{Couette flow} $(y,0,0)^t$. 
Despite the apparent simplicity, understanding the stability of this flow at high Reynolds number ($\nu \rightarrow 0$) is of enduring interest as a canonical, but subtle, problem in hydrodynamic stability, and has been studied regularly throughout the history of fluid mechanics (along with several variants); see e.g. \cite{Kelvin87,Romanov73,Orszag80,Tillmark92,TTRD93,ReddySchmidEtAl98,Chapman02,Liefvendahl2002} for a small representative subset or the texts \cite{DrazinReid81,SchmidHenningson2001,Yaglom12} and the references therein.    

Denoting $u$ for the perturbation of the Couette flow (that is, we set $v = (y, 0, 0)^t + u$), then it satisfies
\begin{equation}
\label{NSC}
\left\{ \begin{array}{l}
\partial_t u - \nu \Delta u + y \partial_x u +  \begin{pmatrix} u^2 \\ 0 \\ 0 \end{pmatrix} - \nabla \Delta^{-1} 2\partial_x u^2  = - u\cdot \nabla u + \nabla \Delta^{-1} (\partial_i u^j \partial_j u^i)\\ 
u(t=0) = u_{in}.
\end{array} \right.
\end{equation}
In this work, we want to answer the following question in the inviscid limit $\nu \rightarrow 0$: \textit{\textbf{Given $\sigma$, what is the smallest $\gamma>0$ such that: if the initial perturbation is such that $ \| u_{in} \|_{H^\sigma} = \epsilon < \nu^\gamma$, then $u$ remains close to the Couette flow (in a suitable sense) and converges back to the Couette flow as $t \to \infty$?}} 
Hence, the goal is not just to prove that the 3D Couette flow is nonlinearly stable in a suitable sense (this is straightforward for \eqref{NSC}) but to estimate the \emph{stability threshold} -- the size of the \emph{largest} ball around zero in $H^\sigma$ such that all solutions remain close to Couette. It is also of interest to determine the dynamics of solutions near the threshold \cite{SchmidHenningson2001}. 

\subsection{Background and previous work} 

Understanding the stability and instability of laminar shear flows at high Reynolds number has been a classical question in applied fluid mechanics since the early experiments of Reynolds \cite{Reynolds83} (see e.g. the texts \cite{DrazinReid81,SchmidHenningson2001,Yaglom12}). 
In 3D hydrodynamics, one of the most ubiquitous phenomena is that of \emph{subcritical transition}: when a laminar flow becomes unstable and transitions to turbulence in experiments or computer simulations at sufficiently high Reynolds number despite perhaps being spectrally stable.  
In fact, the flows in question can be nonlinearly asymptotically stable at all Reynolds number, despite being unstable for all practical purposes \cite{Romanov73} (see also \cite{KreissEtAl94,Liefvendahl2002}). 
It was suggested by Lord Kelvin \cite{Kelvin87} that indeed the flow may be stable, but the stability threshold is decreasing as $\nu \rightarrow 0$, resulting in transition at a finite Reynolds number in any real system. Hence, the goal is, \emph{given a norm} $\norm{\cdot}_X$, to determine a $\gamma = \gamma(X)$ such that 
\begin{align*}
\norm{u_{in}}_{X} & \lesssim \nu^\gamma \quad \Rightarrow \quad \textup{stability} \\ 
\norm{u_{in}}_{X} & \gg \nu^\gamma \quad \Rightarrow \quad \textup{possible instability}. 
\end{align*} 
Of course we do not know a priori that the stability threshold is a power law. In the applied literature, $\gamma$ is often referred to as the \emph{transition threshold}. 
The $\gamma$ is expected to depend non-trivially on the norm $X$ (as observed in, for example, the numerical experiments of \cite{ReddySchmidEtAl98}).

Many works in applied mathematics and physics have been devoted to estimating $\gamma$; see e.g. \cite{BGM15I,SchmidHenningson2001,Yaglom12} and the references therein. 
The linearized problem is non-normal and permits several kinds of transient growth mechanisms: (A) a transient un-mixing effect known as the Orr mechanism, noticed by Orr in 1907 in the context of 2D Couette flow \cite{Orr07}, (B) the 3D lift-up effect, which rearranges mean streamwise momentum to deform the shear flow away from Couette, noticed first by Ellingsen and Palm \cite{EllingsenPalm75} (see also \cite{landahl80}), (C) the transient growth of higher derivatives due to mixing, and (D) a transient vorticity stretching. 
Trefethen et. al. \cite{TTRD93} considered the implications that non-normal effects could have in the weakly nonlinear regime, in particular, forwarding the idea that the nonlinearity could repeatedly re-excite the transient growth, producing a ``nonlinear bootstrap'' scenario.
The authors of \cite{TTRD93} conjecture that $\gamma > 1$ for \eqref{NSC};  a number of works have taken these, and related, ideas further to make conjectures generally giving $1 \leq \gamma \leq 7/4$ (see e.g. \cite{Gebhardt1994,BDT95,Waleffe95,BT97,LHRS94,Chapman02}). 
Unfortunately, many of these authors do not carefully consider how the regularity of the initial data may affect the answer, despite the fact that the strength of the transient growth mechanisms is deeply tied to the regularity since the Couette flow can move information from small scales to large scales (see \S\ref{sec:lin} or \cite{BM13,BGM15I} -- in fact, the sensitivity was noted by Reynolds \cite{Reynolds83}).      
However, a few take the regularity into account, in particular Reddy et. al. \cite{ReddySchmidEtAl98}, where numerical experiments estimated $\gamma \approx 5/4$ for smooth initial data and $\gamma \approx 31/20$ for ``noisy'' data. More recent numerical experiments have since suggested $\gamma \approx 1$ for smooth data \cite{DuguetEtAl10}. 

In this paper we consider Sobolev regularity data and prove that if the initial perturbation satisfies $\norm{u_{in}}_{H^\sigma} \leq \delta\nu^{3/2}$ for $\sigma > 9/2$ and $\delta$ depending only on $\sigma$, then the solution stays within $O(\nu^{1/2})$ of the Couette flow, is attracted back to the class of $x$-independent solutions (referred to here as \emph{streaks}) for $t \gtrsim \nu^{-1/3}$, and finally converges back to equilibrium as $t \rightarrow \infty$.  
Note that this result is very closely matched by the numerical estimate $\gamma \approx 31/20$ of \cite{ReddySchmidEtAl98}; see Remark \ref{rmk:sharp} below for more discussions on regularity and the over-estimations in numerical experiments.  
The main result is stated in Theorem \ref{mainthm} below, the main bootstrap argument is set up in \S\ref{sec:Outline}, and the requisite estimates constitute the remainder of the paper.

The main stabilizing effect is the \emph{mixing-enhanced dissipation} wherein the mixing due to the Couette flow results in anomalously fast dissipation time-scales (first derived by Lord Kelvin \cite{Kelvin87}); see \S\ref{sec:lin} for more discussion or previous works such as \cite{RhinesYoung83,DubrulleNazarenko94,LatiniBernoff01,BernoffLingevitch94,Bajer2001,BeckWayne11,ConstantinEtAl08,BMV14,BCZ15} (\cite{DubrulleNazarenko94} are the first to the authors' knowledge to observe that this is important for understanding \eqref{NSC}). \emph{Inviscid damping}, first derived by Orr \cite{Orr07} in 2D and later noticed to be a hydrodynamic analogue of Landau damping (see e.g. \cite{BoydSanderson,Ryutov99,MouhotVillani11}), also plays a role in suppressing certain nonlinear effects. 

Nonlinear stability of the Couette flow in Sobolev topology has been considered previously in the case of the bounded, infinite channel, that is, $y \in [-1,1]$ and $x \in \Real$ (which can of course lead to further complications, due to the presence of boundary layers), first by Romanov \cite{Romanov73}, with later improvements by \cite{KreissEtAl94} and~\cite{Liefvendahl2002}. This last paper seems to give the best mathematically rigorous result to date for this geometry, namely $\gamma \leq 4$.
In \cite{BGM15I,BGM15II}, we study the stability threshold in Gevrey-$\alpha$ for $\alpha \in (1,2)$ for \eqref{NSC} (Gevrey class was first introduced in \cite{Gevrey18}). 
Roughly speaking, in \cite{BGM15I} we prove that $\gamma = 1$ in these topologies (consistent with the numerical results of \cite{DuguetEtAl10}) and in \cite{BGM15II} we study the dynamics of solutions which are as large as $\nu^{2/3-\delta}$.
Note that the numerical over-estimation of \cite{ReddySchmidEtAl98}, $5/4$ vs $1$, is more pronounced in Gevrey than in Sobolev; see Remark \ref{rmk:sharp}.

All previous work in fluid mechanics and kinetic theory that depend on mixing as the stabilizing mechanism in models with strong nonlinear resonances are in infinite regularity (indeed, the resonances in \eqref{NSC} are far more problematic than those in 2D Navier-Stokes/Euler \cite{BM13} or Vlasov-Poisson \cite{MouhotVillani11}). 
In this work we are looking for the boundary (in terms of $\gamma$) between when finite regularity results are possible and when infinite regularity seems to be required; see \S\ref{sec:disc} for a more in-depth discussion of the relationship between this work and previous related infinite regularity results in \cite{MouhotVillani11,BMM13,BM13,BMV14,Young14,BGM15I,BGM15II}. 
We remark that there exists some finite regularity results in certain kinetic theory models \cite{FaouRousset14,FernandezEtAl2014,Dietert2014}, however, this is possible only because the nonlinearities being studied satisfy stringent non-resonance conditions.

\subsection{Streak solutions}
The first basic property to notice about \eqref{NSC} is that it admits a wide class of so-called ``2.5 dimensional'' solutions, which are often referred to as \emph{streaks}, due to the streak-like appearance of the relatively fast fluid in experiments and computations \cite{TTRD93,SchmidHenningson2001,Trefethen2005,bottin98}.
We will see that all solutions below the threshold converge to these streak solutions for $t \gtrsim \nu^{-1/3}$ and hence these solutions describe the fully 3D nonlinear dynamics for long times.  

\begin{proposition}[Streak solutions] \label{prop:Streak} 
Let $\nu \in [0,\infty)$, $u_{in} \in H^{5/2+}$ be divergence free and independent of $x$, that is, $u_{in}(x,y,z) = u_{in}(y,z)$, and denote by $u(t)$ the corresponding unique strong solution to \eqref{NSC} with initial data $u_{in}$. 
Then $u(t)$ is global in time and for all $T > 0$, $u(t) \in L^\infty( (0,T);H^{5/2+}(\Real^3))$. 
Moreover, the pair $(u^2(t),u^3(t))$ solves the 2D Navier-Stokes/Euler equations in $(y,z) \in \Real \times \Torus$:
\begin{subequations} \label{def:2DNSEStreak} 
\begin{align} 
\partial_t u^i + (u^2,u^3)\cdot \grad u^i &  = -\partial_i p + \nu \Delta u^i, \quad\quad i \in \set{2,3} \\ 
\partial_y u^2 + \partial_z u^3 & = 0,  
\end{align} 
\end{subequations} 
and $u^1$ solves the (linear) forced advection-diffusion equation
\begin{align} 
\partial_t u^1 + (u^2,u^3)\cdot \grad u^1 = -u^2 + \nu \Delta u^1. \label{eq:u1streak}
\end{align}  
\end{proposition}

\subsection{Linear effects} \label{sec:lin}

Four linear effects will play a key role in the analysis to come: lift up, inviscid damping, enhanced dissipation, and vortex stretching. 
We present quickly the linearized problem, and how these four effects arise.

\subsubsection{The linearized problem} 
The linearized problem reads
\begin{align*}
\left\{ \begin{array}{l}
\partial_t u - \nu \Delta u + y \partial_x u +  \begin{pmatrix} u^2 \\ 0 \\ 0 \end{pmatrix} - \nabla \Delta^{-1} 2 \partial_x u^2  =0\\
u(t=0) = u_{in}.
\end{array} \right.
\end{align*} 
Switch to the independent variables $(\overline{x},y,z) = (x-ty,y,z)$ by setting $\overline{u}(t,\overline{x},y,z) = u(t,x,y,z)$; it solves
\begin{align} \label{eq:linbaru} 
\left\{ \begin{array}{l}
\partial_t  \overline u - \nu \Delta_L  \overline u +  \begin{pmatrix}  \overline u^2 \\ 0 \\ 0 \end{pmatrix} - \nabla_L \Delta_L^{-1} 2 \partial_{ \overline x}  \overline u^2  =0\\
 \overline u(t=0) =  \overline u_{in}.
\end{array} \right.
\end{align} 
where $\nabla_L = (\partial_{\overline{x}},\partial_y - t \partial_{\overline{x}},\partial_z)$, and $\Delta_L = \nabla_L \cdot \grad_L$.

\subsubsection{Lift up} Consider first the projection on zero frequencies of the above equation (for a function $f(t,x,y,z)$, we denote $f_0 (y,z)= \int f(x,y,z) \,dx$). 
Note that $\bar{u}_0 = u_0$ and hence it reads
\begin{equation*}
\left\{ \begin{array}{l}
 \partial_t u_0 - \nu \Delta u_0 = - \begin{pmatrix}  u^2_0 \\ 0 \\ 0 \end{pmatrix} \\
 u_0(t=0) =  (u_{in})_0.
\end{array} \right.
\end{equation*}
The solution of this linear problem is given by
$$
u = \left( \begin{array}{c} e^{\nu t \Delta} \left[ (u^1_{in})_0 - t (u^2_{in})_0 \right] \\ e^{\nu t \Delta}(u^2_{in})_0 \\ e^{\nu t \Delta}(u^3_{in})_0 \end{array} \right)
$$
The linear growth predicted by this formula for times $t \lesssim \frac{1}{\nu}$ is known as the \textit{lift up effect}, and was first noticed by Ellingsen and Palm \cite{EllingsenPalm75} (see also \cite{landahl80}). 
This non-normal transient growth turns out to be a primary source of instability in \eqref{NSC} for small data -- note also that this effect is not present in 2D due to the vanishing of $u_0^2$ by incompressibility in that case.   
For smooth data of size $\epsilon$, we can expect at best the bounds, 
$$
\| u_0^1 \|^2_{L^\infty H^s} + \nu \| u_0^1 \|^2_{L^2 H^s} \lesssim \left( \frac{\epsilon}{\nu}\right)^2.
$$

\subsubsection{Inviscid damping} 
Turning now to non zero frequencies in $\bar{x}$, denoted for a function $f(\bar{x},y,z)$ by $f_{\neq} = f - f_0$, observe that the linearized problem satisfied by $\overline q^2_{\neq} = \Delta_L \overline u^2_{\neq}$ reads
\begin{equation} \label{eq:q2lin}
\left\{ \begin{array}{l}
\partial_t  \overline q^2 - \nu \Delta_L  \overline q^2 = 0 \\
\overline q^2 (t=0) = \overline q^2_{in}.
\end{array} \right.
\end{equation}
For smooth data of size $\epsilon$, this gives a global bound on $q^2$ of order $\epsilon$. 
This unknown was first introduced by Kelvin \cite{Kelvin87}, and is often used when studying the stability of parallel shear flows (see e.g. \cite{Chapman02,SchmidHenningson2001}).   
The velocity field can be recovered by the formula $ \bar{u}^2_{\neq} = \Delta_L^{-1} \overline q^2_{\neq}$, or, in Fourier  (denoting $k$, $\eta$, $l$ for the dual variables of $\bar{x}$, $y$, $z$ respectively)
\begin{align}
\widehat{ \overline u^2_{\neq}} = \frac{1}{k^2 + (\eta - kt)^2 + l^2} \widehat{ \overline q^2_{\neq}}.\label{eq:u2lin}
\end{align} 
Due to the bound $ \frac{1}{k^2 + (\eta - kt)^2 + l^2} \lesssim \frac{\jap{\eta}^2}{\jap{kt}^2}$, this leads to a decay estimate of the type
\begin{align} 
\| \overline u^2_{\neq} \|_{H^s} \lesssim \frac{1}{t^2}\norm{\overline{q}^2_{\neq}}_{H^{s+2}}. \label{ineq:LinID}
\end{align} 
This decay mechanism is known as \textit{inviscid damping}; indeed, notice that the decay rate is independent of $\nu$ and is true also for the linearized 3D Euler equations. For the nonlinear problem, we will mostly depend on $L^2_t H^s_x$ estimates, in which case we can expect estimates such as 
\begin{align} 
\| t \overline u^2_{\neq} \|_{L^2 H^s} \lesssim \epsilon
\end{align} 
($\epsilon$ standing for the size of the data). 
The regularity loss in \eqref{ineq:LinID} is required to control the transient growth in \eqref{eq:u2lin} for $\eta k > 0$; modes that are tilted against the shear and are subsequently un-mixed to large scales before being mixed to small scales. This non-normal effect was pointed out by Orr \cite{Orr07} in 2D, however, it will remain important in 3D.  
Orr referred to the time $t= \eta/k$ as the \emph{critical time}, a terminology we also use below.

\subsubsection{Enhanced dissipation}
In order to understand enhanced dissipation better, consider the model scalar problem, such as that solved by $\bar{q}^2$ above in \eqref{eq:q2lin}, 
\begin{equation*}
\left\{ \begin{array}{l}
\partial_t  w_{\neq} - \nu \Delta_L  w_{\neq} = 0 \\
w_{\neq} (t=0) =  (w_{in})_{\neq}.
\end{array} \right.
\end{equation*}
Taking the Fourier transform, the problem can be recast as
\begin{equation*}
\left\{ \begin{array}{l}
\partial_t  \widehat{w_{\neq}} - \nu (k^2 + (\eta- kt)^2 + l^2) \widehat{w_{\neq}} = 0 \\
\widehat{w_{\neq}} (t=0) =\widehat{(w_{in})_{\neq}}.
\end{array} \right.
\end{equation*}
Thus $ \widehat{w_{\neq}}(t,k,\eta,l) = e^{-\nu \int_0^t  (k^2 + (\eta- k\tau)^2 + l^2) d\tau}\widehat{ (w_{in})_{\neq}}$. Due to the inequality $\int_0^t  (k^2 + (\eta- k\tau)^2 + l^2) d\tau \gtrsim t^3$, for the linear problem we get the decay
$$
\| w_{\neq} \|_{H^s} \lesssim \epsilon e^{-c\nu t^3}.
$$
This decay is much faster than the standard viscous dissipation, indeed, the characteristic time scale for dissipation in non-zero-in-$x$ modes is order $\sim \nu^{-1/3}$ instead of $\nu^{-1}$. 
We refer to this phenomenon as \textit{enhanced dissipation}; as mentioned above, it has been studied in several contexts previously, see e.g. \cite{RhinesYoung83,DubrulleNazarenko94,LatiniBernoff01,ConstantinEtAl08,BeckWayne11,BMV14,BCZ15}.   
In this work, we will use $L^2$ time-integrated estimates of the type 
$$
\| w_{\neq} \|_{L^2 H^s} \lesssim \frac{\epsilon}{\nu^{1/6}} \quad \mbox{and} \quad \| tw_{\neq} \|_{L^2 H^s} \lesssim \frac{\epsilon}{\sqrt \nu}.
$$

\subsubsection{Vorticity stretching and kinetic energy cascade} 
The control of $\bar{q}^2$ provides the rapid decay of $\bar{u}^2$ via inviscid damping, which can then be integrated to understand the evolution of $\bar{u}^1$ and $\bar{u}^3$ in 
 \eqref{eq:linbaru}. In particular, we see that for times $1 \ll t \ll \nu^{-1/3}$, $\bar{u}^{1,3}_{\neq}$ are essentially time-independent, and hence over these times $u^{1,3}_{\neq}$ are being mixed like a passive scalar by the Couette flow. Hence, over these time scales we see a forward cascade of \emph{kinetic energy} (this persists on the nonlinear level as well \cite{BGM15I}).   
Due to the negative order of the Biot-Savart law, it is easy to see that a forward cascade of kinetic energy is only possible if there is an accompanying vorticity stretching; this can also be confirmed by studying \eqref{eq:linbaru} in vorticity form.  

Finally, we summarize the linear behavior here. 
 \begin{proposition}[Linearized Navier-Stokes] \label{prop:LinNSE} 
Let $u_{in}$ be a divergence free, smooth vector field. The solution to the linearized Navier-Stokes $u(t)$ with initial data $u_{in}$ satisfies the following for some $c \in (0,1/3)$
\begin{subequations} 
\begin{align} 
\norm{\bar{u}^{2}_{\neq}(t)}_{H^\sigma}  & \lesssim \jap{t}^{-2} e^{-c\nu t^3} \norm{u^2_{in}}_{H^{\sigma+2}} \label{ineq:U2LinearID_vs} \\ 
\norm{\bar{u}^{1,3}_{\neq}(t)}_{H^\sigma} & \lesssim e^{-c\nu t^3}  \norm{u_{in}}_{H^{\sigma+7}},   \label{ineq:U1LinearID_vs}
\end{align} 
\end{subequations} 
and the formulas
\begin{subequations} \label{def:NSEstreak}
\begin{align}
u^1_0(t,y,z) & = e^{\nu t \Delta}\left(u^1_{in \; 0} -  t u_{in \; 0}^2\right) \label{eq:liftup_vs} \\ 
u^2_0(t,y,z) & = e^{\nu t \Delta} u^2_{in \; 0} \\
u^3_0(t,y,z) & = e^{\nu t \Delta} u^3_{in \; 0}. 
\end{align}
\end{subequations}
\end{proposition}

\subsection{Statement of results}
We now state our main results. 
\begin{theo} \label{mainthm}
For all $\sigma > 9/2$, there exists $\delta = \delta(\sigma)$ such that: if $\nu \in (0,1)$ and $u_{in}$ is divergence free with   
\begin{align} 
\epsilon = \| u_{in} \|_{H^{\sigma}} < \delta \nu^{3/2}, \label{ineq:smallness}
\end{align}
then the resulting strong solution to \eqref{NSC} is global in time and there exists a function $\psi(t,y,z)$ satisfying 
$$
\| \psi \|_{L^\infty H^{\sigma}}^2 + \nu \| \nabla \psi \|_{L^2 H^{\sigma}}^2 \lesssim \frac{\epsilon^2}{\nu^2},
$$
such that, denoting by $U^i$, $i\in \set{1,2,3}$ the velocity field $u^i$ in the new coordinates
$$
U^i (t,x-ty-t\psi(t,y,z),y+\psi(t,y,z),z) = u^i(t,x,y,z),
$$
the solution $u(t)$ to ~\eqref{NSC} with initial data $u_{in}$ is global in time and satisfies the following estimates: 
\begin{subequations}
\begin{align}
& \norm{u_0^1}_{L^\infty H^{\sigma}} +   \sqrt{\nu} \norm{\grad u^{1}_0}_{L^2 H^{\sigma}} \lesssim \frac{\epsilon}{\nu} \\ 
& \norm{u_0^{2,3}}_{L^\infty H^\sigma} + \sqrt{\nu} \norm{\grad u^{2,3}_0}_{L^2 H^{\sigma}} \lesssim \epsilon & \\
&  \norm{U^{2}_{\neq}}_{L^\infty H^{\sigma-2}} + \norm{\grad_L U^{2}_{\neq}}_{L^2 H^{\sigma-3}} + \norm{t U^{2}_{\neq}}_{L^2 H^{\sigma-4}} \lesssim \epsilon \label{ineq:U2ID} \\ 
& \norm{U^1_{\neq}}_{L^\infty H^{\sigma-3}} + \sqrt{\nu}\norm{t U^{1}_{\neq}}_{L^2 H^{\sigma-4}}  \lesssim \epsilon \label{ineq:U1ED} \\ 
& \norm{U^3_{\neq}}_{L^\infty H^{\sigma-2}} + \sqrt{\nu}\norm{t U^{3}_{\neq}}_{L^2 H^{\sigma-3}}  \lesssim \epsilon \label{ineq:U3ED} 
\end{align}
\end{subequations} 
\end{theo}

\begin{rema} 
The latter terms in \eqref{ineq:U1ED} and \eqref{ineq:U3ED} emphasize the effect of enhanced dissipation, discussed above in \S\ref{sec:lin}. In particular, the scaling of the $L^2H^{\sigma-4}$ norm of $tU^i_{\neq}$ is far better at small $\nu$ than what would be true of the heat equation.  
The second two estimates in \eqref{ineq:U2ID} emphasize the effect of inviscid damping: notice indeed that the decay does not depend on $\nu$. 
\end{rema}

\begin{rema} \label{rmk:sharp}
 How optimal are the assumptions of the theorem?
\begin{itemize} 
\item As mentioned previously, numerics in ~\cite{ReddySchmidEtAl98} estimated a threshold for ``noisy data" at $\epsilon \sim \nu^{31/20}$; Theorem \ref{mainthm} shows that the stability threshold is slightly better. In light of the numerical evidence, it is reasonable to conjecture that Theorem \ref{mainthm} is sharp in terms of $\gamma$ over some range of Sobolev spaces.  
 
\item By parabolic smoothing, it should be possible to slightly weaken \eqref{ineq:smallness} to something like: $u_{in} = u_S +u_R$ with $\norm{u_S}_{H^{9/2+}} + C\nu^{\frac{9}{4} - \frac{\alpha}{2}}\norm{u_R}_{H^{\alpha}} < \delta \nu^{3/2}$ for a universal $C$ at least over some range of $\alpha \in (5/2,9/2)$.
 This is a local-in-time effect which is totally independent of Theorem \ref{mainthm} (though it may be a non-trivial refinement of the local theory for \eqref{NSC}). 
This is qualitatively consistent with the numerical over-estimation observed in \cite{ReddySchmidEtAl98} and others: numerical algorithms will inevitably introduce noise at the smallest scales of the simulation and hence possibly over-estimate $\gamma$ -- indeed, more recent computations carried out in \cite{DuguetEtAl10} are closer to the $\gamma \approx 1$ in the case of smooth data.   
This also suggests that the Sobolev regularity $\gamma$ is more robust to low-regularity noise than the infinite regularity $\gamma$ (which requires exponentially small noise \cite{BGM15I}), which is consistent with the mentioned numerical observations.   
\end{itemize}
\end{rema}

\subsection{Brief discussion of the results and new ideas} \label{sec:disc}
Our work shows that it may now be feasible to build a mathematical theory of subcritical instabilities in fluid mechanics and possibly also in related fields, such as magneto-hydrodynamics. This seems especially possible in finite regularity, as the methods here are significantly more tractable than those in infinite regularity \cite{BGM15I,BGM15II}.  
Indeed, in the proof of Theorem \ref{mainthm}, we need to use methods which differ significantly from those used in the infinite regularity works \cite{MouhotVillani11,BMM13,BM13,BMV14,Young14,BGM15I,BGM15II}. 
In all of these previous works, the infinite regularity class is used to absorb the potential frequency cascade due to weakly nonlinear effects in a process related to classical Cauchy-Kovalevskaya-type arguments in e.g. \cite{Nirenberg72,Nishida77,FoiasTemam89,LevermoreOliver97} (see \S\ref{sec:FM} for more precise discussions) or, in the case of \cite{MouhotVillani11}, via a Nash-Moser-type iteration. 
Here this is clearly not an option, and hence we need to rule out any such cascade with the \emph{least possible} amount of dissipation; something which will require a different kind of understanding of the weakly nonlinear effects in the pressure and a more precise understanding of the interplay between the enhanced dissipation and vortex stretching. 
The starting point for this is the linear analysis of \S\ref{sec:lin}, and based on this, Fourier multipliers which precisely encode the interplay between the dissipation and possible growth are designed. These multipliers are then used to make energy estimates which lose the minimal amount of information from the linear terms; see \S\ref{sec:FM} for specifics and context with existing ideas in e.g. \cite{FoiasTemam89,Alinhac01,BM13} and others (in particular, we need multipliers which more precisely capture the effect of dissipation than in \cite{BGM15I,BGM15II}). 

Once we have understood and quantified the linear terms, one needs to understand how this linear behavior interacts with the nonlinearity.  
For this, of critical importance in the proof is the precise structure of the nonlinearity, which contains a number of null structures. 
Similar to null forms for quasilinear wave equations, introduced in \cite{Klainerman}, the null structures encountered in the present paper cancel possible interactions between large modes or derivatives of the solution. The simplest is that the nonlinearity in \eqref{NSC} does not allow $u^1_0$ to directly interact with itself in a nonlinear way (this is essentially how Proposition \ref{prop:Streak} works) -- however a similar structure also limits the way $u^1_{\neq}$ and $u_0^1$ interact. 
Another slightly more subtle structure is that, since the nonlinearity is comprised of forms of the type $u^j \partial_j u^i$, the large growth of $y$ derivatives is 
crucially counter-balanced by the inviscid damping of $u^2$ in nonlinear terms. Indeed, this is why quantifying the inviscid damping of $u^2$ is important for the proof to work. Similarly, the $u^1 \partial_x$ and $u^3 \partial_z$ structure pairs less problematic derivatives with the more problematic $u^{1,3}$. 
Since the inviscid damping is important, a key physical mechanism to understand is how the streak and the kinetic energy cascade interact nonlinearly in the $y$ derivative of the pressure, that is, the nonlinear term: $-\partial_y \Delta^{-1}\left(\partial_zu_0^1 \partial_x u^3_{\neq}\right)$. 
Controlling this term is one of the main challenges, which is done in \S\ref{sec:NLPQ2}, and in it, \emph{all} of the linear effects outlined in \S\ref{sec:lin} are playing a role (which is why it is very important that these are treated precisely). 
See \S\ref{sec:Outline} below for more details on the proof and techniques. 

\section{Preliminaries and outline of the proof} \label{sec:Outline}
\subsection{Notations}

\subsubsection{Miscellaneous} Given two quantities $A$ and $B$, we denote $A \lesssim B$ if there exists a constant $C$ such that $A \leq CB$. This constant might depend on $\sigma$, but not on $\delta$, $\nu$, $C_0$ or $C_1$ (the two latter quantities remain to be defined). 
Finally, we write $\langle x \rangle = \sqrt{1+ x^2}$.

\subsubsection{Fourier Analysis}
The Fourier transform of a function $f(X,Y,Z)$, denoted $\widehat{f}(k,\eta,\ell)$ or $\mathcal{F} f$, is such that
\begin{align*}
& f(X,Y,Z) = \sum_k \int_{\eta \in \mathbb{R}} \sum_\ell \widehat{f}(k,\eta,\ell) e^{2\pi i(kX + \eta Y + \ell Z)} \,d\eta\\
& \widehat{f}(k,\eta,\ell) =  \int_{X \in \mathbb{T}} \int_{Y \in \mathbb{R}} \int_{Z \in \mathbb{T}} f(X,Y,Z) e^{-i 2\pi (kX + \eta Y + \ell Z)} \,dX\,dY\,dZ.
\end{align*}
The Fourier multiplier with symbol $m(k,\eta,\ell)$ is such that
$$
m(D) f = \mathcal{F}^{-1} m(k,\eta,\ell) \mathcal{F} f.
$$
The projections on the zero frequency in $X$ of a function $f(X,Y,Z)$ is denoted by
$$
P_0 f = f_0 = \int f(X,Y,Z)\,dX
$$
while
$$
P_{\neq} f = f_{\neq} = f - P_0 f.
$$

\subsubsection{Functional spaces}
The Sobolev space $H^N$ is given by the norm
$$
\| f \|_{H^N} = \| \langle D \rangle^N f \|_{L^2}.
$$
Recall that, for $s>\frac{3}{2}$, $H^s$ is an algebra: $\|fg\|_{H^s} \lesssim \|f\|_{H^s} \|g\|_{H^s}$.

We will sometimes use the notation $H^{s+}$ for $H^{s+\kappa}$, where $\kappa$ can be taken arbitrarily small, with (implicit) constants depending on $\kappa$.

For a function of space and time $f = f(t,x)$, and times $a < b$, the Banach space $L^p(a,b;H^N)$ is given by the norm
$$
\| f \|_{L^p(a,b;H^N)} = \left\| \left\| f \right\|_{H^N} \right\|_{L^p(a,b)}.
$$
For simplicity of notation we usually simply write $\norm{f}_{L^p H^N}$ as the time-interval of integration in this work will be the same basically everywhere. 

\subsubsection{Littlewood-Paley decomposition and paraproduct} \label{sec:LPP} 
Start with a smooth, non-negative function $\theta$ supported in the annulus $B(0,5) \setminus B(0,1)$ of $\mathbb{R}^3$, and such that $\sum_{j=-\infty}^{+\infty} \theta \left(\frac{\xi}{2^j} \right) = 1$ for $\xi \neq 0$, and define the Fourier multipliers
$$
P_j = \theta \left(\frac{D}{2^j} \right) \qquad P_{\leq J} =  \sum_{j=-\infty}^J  \theta \left(\frac{D}{2^j} \right) \qquad P_{>J} = 1 - P_{\leq J}.
$$
These Fourier multipliers enable us to split the product into two pieces such that each corresponds to the interaction of high frequencies of one function with low frequencies of the other:
$$
fg = f_{Hi} g_{Lo} + f_{Lo} g_{Hi}
$$
with
\begin{align*}
f_{Hi} g_{Lo} = \sum_j P_j f P_{\leq j} g \qquad f_{Lo} g_{Hi} = \sum_j P_{\leq j-1} f P_j g
\end{align*}
(the lack of symmetry in this formula is irrelevant). We record the estimate
\begin{align} 
\| f_{Hi}g_{Lo}\|_{H^s} \lesssim \|f\|_{H^s} \| g\|_{H^\sigma} \qquad \mbox{for $s>0$, $\sigma>\frac{3}{2}$.} \label{LPprod}
\end{align}
We further note that if $g$ depends only on two variables, say $y$ and $z$, then we have 
\begin{align} 
\| f_{Hi}g_{Lo}\|_{H^s} \lesssim \|f\|_{H^s} \| g\|_{H^\sigma} \qquad \mbox{for $s>0$, $\sigma>1$.} \label{LPprod2p5D}
\end{align}

\subsection{Re-formulation of the equations} \label{COV}
First, we re-formulate the equations to make them more amenable to long-time, nonlinear analysis.  

\subsubsection{Change of dependent variables}
In order to understand the linearized equation in \S\ref{sec:lin}, it is important to use the unknown $q^2 = \Delta u^2$. 
 In linear or formal weakly nonlinear analyses (see e.g. \cite{Chapman02,SchmidHenningson2001} and the references therein) it is natural to couple $q^2$ with the vertical component of the vorticity, however, we will also need to change independent variables to adapt to the mixing caused by $u_0^1$, which makes this approach very problematic. 
Therefore, it is more convenient to work with the set of unknowns $q^i = \Delta u^i$ (as observed in \cite{BGM15I}).
These unknowns satisfy the system 
\begin{equation} \label{eq:qi}  
\left\{ \begin{array}{l}
\partial_t q^1 +y \partial_x q^1 - \nu \Delta q^1+  2\partial_{xy} u^1 + q^2 - 2\partial_{xx} u^2 = - u \cdot \nabla q^1 - q^j \partial_j u^1 -  2 \partial_{i} u^j \partial_{ij}u^1 + \partial_x \left(\partial_i u^j \partial_j u^i\right) \\
\partial_t q^2 +y \partial_x q^2 - \nu \Delta q^2 = - u \cdot \nabla q^2 - q^j \partial_j u^2 - 2  \partial_{i} u^j \partial_{ij}u^2 + \partial_y \left(\partial_i u^j \partial_j u^i\right)\\
\partial_t q^3 +y \partial_x q^3 - \nu \Delta q^3 +  2\partial_{xy} u^3 - 2\partial_{xz} u^2  = - u \cdot \nabla q^3  - q^j \partial_j u^3 - 2 \partial_{i} u^j \partial_{ij}u^3 + \partial_z \left(\partial_i u^j \partial_j u^i\right)\\
q(t=0) = q_{in}.
\end{array} \right.
\end{equation}

\subsubsection{Change of independent variables} \label{sec:coordchnge}
The $x$-component of the streak, $u^1_0$, is expected to be as large as $O(\epsilon \nu^{-1})$ (again from \S\ref{sec:lin}), which is far too large to be balanced directly by the dissipation (it is not hard to check this would require $\epsilon \ll \nu^{2}$). 
Hence, we remove the fast mixing action due to the streak itself, an approach also used in \cite{BGM15I} for the same reason.  
There is essentially no choice in the change of coordinates we can employ -- it is dictated uniquely by the desired properties and the structure of the equation. 
Hence, define the coordinate transform as in \cite{BGM15I},  
\begin{subequations} \label{def:cords}
\begin{align} 
X & = x - t y - t \psi(t,y,z) \\ 
Y & = y + \psi(t,y,z) \\ 
Z & = z,
\end{align}
\end{subequations} 
where $\psi$ is chosen to satisfy the PDE 
\begin{subequations}
\begin{align}
\frac{d}{dt}(t\psi) + u_0 \cdot \grad \left(t\psi\right) & = u_0^1 - t u_0^2 + \nu t \Delta \psi  \label{def:psi} \\ 
\lim_{t \rightarrow 0} t \psi(t) & = 0. 
\end{align}
\end{subequations}
The mild coordinate singularity at $t = 0$ will be irrelevant, as this coordinate transform will only be applied for $t \geq 1$ (see \S\ref{sec:coordstuff} for more details). 
To distinguish between old and new coordinates, we capitalize $u^i$ and $q^i$, while $\psi$ itself becomes $C$:
\begin{align*} 
U^i(t,X,Y,Z) = u^i(t,x,y,z), \quad Q^i(t,X,Y,Z) = q^i(t,x,y,z), \quad C(t,Y,Z) = \psi(t,y,z), 
\end{align*}
where we are using the shorthand $X = X(t,x,y,z)$, $Y = Y(t,x,y,z)$, and $Z = Z(t,x,y,z)$. 
Similarly, we denote
\begin{align*}
\psi_y(t,Y,Z) = \partial_y\psi(t,y,z), \quad \psi_z(t,Y,Z) = \partial_z\psi(t,y,z). 
\end{align*}
In the new coordinates, differential operators are modified as follows: if $f(t,x,y,z) = F(t,X,Y,Z)$,
\begin{align*} 
\nabla f(t,x,y,z) = \left( \begin{array}{c} \partial_x f \\ \partial_y f \\ \partial_z f \end{array} \right) = \left( \begin{array}{c} \partial_X F \\ (1 + \psi_y)(\partial_Y - t \partial_X) F \\ ( \partial_Z + \psi_z (\partial_Y - t \partial_X ))  F \end{array} \right) = \left( \begin{array}{c} \partial_X^t F \\ \partial_Y^t F \\ \partial_Z^t F \end{array} \right) = \nabla_t F(t,X,Y,Z).
\end{align*} 
It will be useful to isolate the ``linear part'' of $\nabla_t$ (that is, the contribution associated with the linearized problem), which we denote $\nabla_L$:
\begin{align*} 
\nabla_L = \left( \begin{array}{c} \partial_X \\ \partial_Y - t \partial_X \\ \partial_Z \end{array}  \right) = \left(  \begin{array}{l}  \partial_X \\ \partial_Y^L \\ \partial_Z \end{array} \right).
\end{align*}
Using the notation
\begin{subequations} 
\begin{align} 
\Delta_L & = \nabla_L \cdot \grad_L = \partial_X^2 + (\partial_Y^L)^2 + \partial_Z^2 \\ 
G & = (1+\psi_y)^2+\psi_z^2-1, \label{def:G}
\end{align} 
\end{subequations} 
the Laplacian transforms as
\begin{align*}
\Delta f = \Delta_t F & = \left( (\partial_{X})^2 + (\partial_Y^t)^2 + (\partial_Z^t)^2 \right) F =  \Delta_L F + G \partial_{YY}^L F + 2 \psi_z \partial_{ZY}^L F + \Delta_t C \partial_Y^L F. 
\end{align*}
We will also need the modified Laplacian
\begin{align*} 
\widetilde{\Delta}_t F = \Delta_t F - \Delta_t C \partial_Y^L F = \Delta_L F + G \partial_{YY}^L F + 2 \psi_z \partial_{YZ}^L F.
\end{align*}
Notice that $\psi_y$, $\psi_z$, and $C$ are related as follows:
\begin{subequations} \label{eq:coefGeom}
\begin{align} 
\psi_y & =  \frac{\partial_Y C}{1 - \partial_YC} = \partial_Y C\sum_{j = 0}^\infty (\partial_YC)^j \\ 
\psi_z & = \frac{\partial_Z C}{1 - \partial_YC} = \partial_Z C\sum_{j = 0}^\infty (\partial_YC)^j. 
\end{align}
\end{subequations} 
Define $g$ and $\widetilde{U}_0$ (which will be the $X$-independent part of the velocity in the new coordinates) by
\begin{align*} 
g = \frac{1}{t}(U_0^1 - C), \quad \widetilde{U}_0 = \left( \begin{array}{c} 0 \\ g \\ U_0^3 \end{array} \right).
\end{align*}
Computing from \eqref{def:psi}, \eqref{eq:qi}, and \eqref{NSC} gives the following system (see \cite{BGM15I} for more details), 
\begin{align} \label{def:MainSys}
\left\{
\begin{array}{l}
Q^1_t - \nu \widetilde{\Delta}_t Q^1 + Q^2 + 2\partial_{XY}^t U^1 - 2\partial_{XX} U^2 \\
 \qquad \qquad \qquad \qquad = - \widetilde U_0 \cdot \nabla Q^1 -  U_{\neq} \cdot \nabla_t Q^1 - Q^j \partial_j^t U^1 - 2\partial_i^t U^j \partial_{ij}^t U^1 + \partial_X(\partial_i^t U^j \partial_j^t U^i)  \\  
Q^2_t - \nu \widetilde{\Delta}_t Q^2  = - \widetilde U_0 \cdot \nabla Q^2 - U_{\neq} \cdot \nabla_t Q^2 - Q^j \partial_j^t U^2 - 2\partial_i^t U^j \partial_{ij}^t U^2 + \partial_Y^t(\partial_i^t U^j \partial_j^t U^i)  \\ 
Q^3_t - \nu \widetilde{\Delta}_t Q^3 +2\partial_{XY}^t U^3 - 2\partial_{XZ}^t U^2 \\
\qquad \qquad \qquad \qquad = - \widetilde U_0 \cdot \nabla Q^3 - U_{\neq} \cdot \nabla_t Q^3  - Q^j \partial_j^t U^3 - 2\partial_i^t U^j \partial_{ij}^t U^3 + \partial_Z^t(\partial_i^t U^j \partial_j^t U^i), \\  
\end{array}
\right.
\end{align} 
coupled with 
\begin{align} \label{eq:Cg} 
\left\{
\begin{array}{l}
\partial_t C + \widetilde{U}_0 \cdot \nabla C = g - U_0^2 + \nu \widetilde{\Delta}_t C \\ 
\partial_t g + \widetilde{U}_0 \cdot \nabla g = - \frac{2}{t}g - \frac{1}{t} (U_{\neq} \cdot \nabla_t U^1_{\neq})_0 + \nu \widetilde{\Delta}_t g. 
\end{array}
\right.
\end{align} 
Although most work is done directly on the system \eqref{def:MainSys}, \eqref{eq:Cg}, 
for certain steps it will be useful to use the momentum form of the equations
\begin{equation}
\partial_t U - \nu \widetilde{\Delta}_t U + \widetilde{U}_0 \cdot \nabla U + U_{\neq} \cdot \nabla_t U =  \begin{pmatrix} - U^2 \\ 0 \\ 0 \end{pmatrix} + \nabla_t {\Delta}_t^{-1} 2 \partial_X U^2 + \nabla_t \Delta_t^{-1} (\partial^t_i U^j \partial_j^t U^i). \label{eq:momentum}
\end{equation}

\subsubsection{Shorthands} \label{sec:short} 

It will be quite convenient to use shorthands for the various terms appearing in the above equations, and to be able to distinguish whether interacting modes have zero or non-zero $X$ frequency. Let us start with linear terms, appearing in the equations for $Q^k$, $k=1,3$:
\begin{align*}
& LU = Q^2 & \mbox{(lift up term)}\\
& LS = 2\partial_{XY}^t U^k & \mbox{(linear stretching term)}\\
& LP = - 2\partial_{Xk} U^2& \mbox{(linear pressure term)} .
\end{align*}
Next, consider the nonlinear terms in \eqref{def:MainSys}. In the following, $i,j$ run in $\{ 1,2,3 \}$, while $\epsilon_1$ and $\epsilon_2$ may be $0$ or $\neq$):
\begin{align*}
& T_{0,\epsilon_1} = \widetilde U_0 \cdot \nabla Q^k_{\epsilon1}  & \mbox{(transport term)} \\
& T_{\neq, \epsilon_1} = U_{\neq} \cdot \nabla_t Q^k_{\epsilon1} & \mbox{(transport term)} \\
& NLS1(j,\epsilon_1,\epsilon_2) = Q^j_{\epsilon_1} \partial_j^t U^k_{\epsilon_2}  & \mbox{(nonlinear stretching term)}\\
& NLS2(i,j,\epsilon_1,\epsilon_2) = 2 \partial_i^t U^j_{\epsilon_1} \partial_{ij}^t U^k_{\epsilon_2}  & \mbox{(nonlinear stretching term)}\\
& NLP(i,j,\epsilon_1,\epsilon_2) = \partial^t_k(\partial_j^t U^i_{\epsilon_1} \partial_i^t U^j_{\epsilon_2})  & \mbox{(nonlinear pressure term)}.
\end{align*}
We will often abuse notation and, for instance, denote indifferently $NLS1$ for this term, and its contribution to an energy estimate.
The origin of these terms is more or less clear except perhaps the ``stretching'' terminology, which is due to the similarity these terms have with the vortex stretching term in the 3D Navier-Stokes equations in vorticity form. 

\subsection{The Fourier multipliers} \label{sec:FM}

At this point, our work here will depart from the infinite regularity case \cite{BGM15I}.   

\bigskip

\noindent \underline{The multiplier $m$: stretching versus dissipation.} Our focus here is the following linear equation, 
$$
\partial_t f + 2 \partial_{XY}^L \Delta_L^{-1} f - \nu \Delta_L f = 0,
$$
which occurs as some of the main linear terms governing $Q^1$ and $Q^3$ in \eqref{def:MainSys}.  
This equation can be seen as a competition between the linear stretching term $2\partial_{XY}^L \Delta_L^{-1}f$ and the dissipation term $\nu\Delta_L f$. 
Taking the Fourier transform, it becomes
$$
\partial_t \widehat{f} + 2 \frac{k(\eta - kt)}{k^2 + (\eta - kt)^2 + \ell^2} \widehat{f} + \nu \left(k^2 + (\eta - kt)^2 + \ell^2 \right) \widehat{f} = 0.
$$
If $k \neq 0$, the factor $2 \frac{k(\eta - kt)}{k^2 + (\eta - kt)^2 + \ell^2}$ is positive for $t < \frac{\eta}{k}$, in which case it amounts to damping on $\widehat{f}$; and negative for $t> \frac{\eta}{k}$, in which case it corresponds to an amplification of $\widehat{f}$. As for the factor $\nu \left(k^2 + (\eta - kt)^2 + \ell^2 \right)$, this gives enhanced dissipation for $k \neq 0$. We start with the following inequality, which compares the sizes of these two factors: uniformly in $(k,\eta,\ell)$, if $k \neq 0$,
$$
\nu \left(k^2 + (\eta - kt)^2 + \ell^2 \right) >> \frac{|k(\eta - kt)|}{k^2 + (\eta - kt)^2+ \ell^2} \quad \mbox{if} \quad |t - \frac{\eta}{k} | >> \nu^{-1/3}.
$$
Indeed, $\frac{|k(\eta - kt)|}{\nu (k^2 + (\eta - kt)^2+ \ell^2)^2} \leq \frac{|t - \frac{\eta}{k}|}{\nu(1 + |t-\frac{\eta}{k}|^2)^2}$, and it is easy to check that $\frac{x}{\nu(1+x^2)^2} << 1$ for $|x| >> \nu^{-1/3}$. 

To summarize, stretching overcomes dissipation if $0 < t-\frac{\eta}{k} \lesssim \nu^{-1/3}$. To deal with this range of $t$, 
we introduce the multiplier $m$. Define $m(t,k,\eta,\ell)$ by $m(t=0,k,\eta,\ell) = 1$ and the following ODE:
\begin{equation*}
\frac{\dot{m}}{m} =
\left\{
\begin{array}{ll}
 0 & \mbox{if $t \notin \left[ \frac{\eta}{k}, \frac{\eta}{k} + 1000\nu^{-1/3} \right]$} \\
\frac{2k(\eta-kt)}{k^2 + (\eta - kt)^2 + \ell^2} & \mbox{if $t \in \left[ \frac{\eta}{k}, \frac{\eta}{k} + 1000\nu^{-1/3} \right]$}
\end{array}
\right.
\end{equation*}
This multiplier is such that: if $f$ solves the above equation and $0 < t-\frac{\eta}{k} < 1000 \nu^{-1/3}$, then $m f$ solves
$$
\partial_t (mf) -  \nu \Delta_L (mf) = 0,
$$
and this equation is perfectly well behaved! That is, the growth that $f$ undergoes is balanced by the decay of the multiplier $m$ -- this is especially useful since the growth is highly anisotropic in frequency.  
Conveniently, it turns out that $m$ is given by a closed formula:
\begin{enumerate}
\item {If $k=0$}: $m(t,0,\eta,\ell) = 1$.
\item {If $k \neq 0$, $\frac{\eta}{k} < -1000 \nu^{-1/3}$}: $m(t,k,\eta,\ell) = 1$.
\item {If $k \neq 0$, $-1000 \nu^{-1/3} < \frac{\eta}{k} < 0$}:
\begin{itemize}
\item $m(t,k,\eta,\ell) = \frac{k^2 + \eta^2 + \ell^2}{k^2 + (\eta-kt)^2 + \ell^2}$ if $0<t<\frac{\eta}{k} + 1000 \nu^{-1/3}$.
\item $m(t,k,\eta,\ell) = \frac{k^2 + \eta^2 + \ell^2}{k^2 + (1000 k \nu^{-1/3})^2 + \ell^2}$ if $t>\frac{\eta}{k} + 1000 \nu^{-1/3}$.
\end{itemize}
\item {If $k \neq 0$, $\frac{\eta}{k} > 0$}:
\begin{itemize}
\item $m(t,k,\eta,\ell) = 1$ if $t < \frac{\eta}k$.
\item $m(t,k,\eta,\ell) = \frac{k^2 + \ell^2}{k^2 + (\eta-kt)^2 + \ell^2}$ if $ \frac{\eta}k < t < \frac{\eta}{k} + 1000 \nu^{-1/3}$.
\item $m(t,k,\eta,\ell) = \frac{k^2 + \ell^2}{k^2 + (1000 k \nu^{-1/3})^2 + \ell^2}$ if $t > \frac{\eta}k + 1000 \nu^{-1/3}$.
\end{itemize}
\end{enumerate}
Notice in particular that
\begin{align}
\nu^{2/3} \lesssim m(t,k,\eta,\ell) \leq 1.\label{mlowbound}
\end{align} 
Further, we point out the following key inequality, which shows that the growth is exactly balanced by $\Delta_L$: 
\begin{align}
m(t,k,\eta,l) \gtrsim \frac{k^2+l^2}{k^2 + l^2 + \abs{\eta-kt}^2}. \label{ineq:mDelTrick}
\end{align}

\bigskip
\noindent
\underline{Additional multiplier bounded from below by a positive constant.}
We will use several additional multipliers, which unlike $m$, are bounded above and below uniformly in $\nu$ and frequency. 
Multipliers $M^0$ and $M^1$ are used to balance the growth due to the linear pressure terms as well as some of the leading order nonlinear terms. 
The multiplier $M^2$ plays an especially crucial role by compensating for the transient slow-down of the enhanced dissipation near the critical times and hence this multiplier will be ultimately how we quantify accelerated dissipation \emph{without} regularity loss -- of crucial importance to our methods and not possible with the techniques employed in the infinite regularity works \cite{BMV14,BGM15I}. 

Define $M^i$, $i = 0,1,2$ as follows: $M^i(t=0,k,\eta,\ell) = 1$ and
\begin{itemize}
\item if $k=0$, $M^i(t,k,\eta,\ell) = 1$ for all $t$;
\item if $k \neq 0$, $\displaystyle \frac{\dot{M^0}}{M^0} = \frac{-k^2}{k^2 + \ell^2 + (\eta-kt)^2}$;
\item if $k \neq 0$, $\displaystyle \frac{\dot{M^1}}{M^1} = \frac{-2\langle k\ell \rangle}{k^2 + \ell^2 + (\eta-kt)^2}$;
\item if $k \neq 0$, $\displaystyle \frac{\dot{M^2}}{M^2} = - \frac{\nu^{1/3}}{\left[ \nu^{1/3} |t-\frac{\eta}{k} | \right]^{1 + \kappa} + 1}$;
\end{itemize}
where $\kappa \in (0,1/2)$ is a small, fixed constant. 
It is easy to check that these multipliers satisfy
$$
0 < c < M^i(t,k,\eta,l) \leq 1
$$
for a universal constant $c$. Define then
$$
M = M^0 M^1 M^2.
$$
To see the usefulness of $M$, consider the weighted energy estimate, 
\begin{align*} 
\frac{1}{2} \frac{d}{dt} \norm{M Q^3_{\neq}}_{H^N}^2 = -\norm{\sqrt{-\dot M M} Q^3_{\neq}}_{H^N}^2 + \langle M U^3 , M \partial_t Q^3_{\neq} \rangle_{H^N}.
\end{align*}
In order to bound the latter term, we may firstly use some of the negative term coming from $\dot{M}$, and secondly, if we can control the term by something like
\begin{align*}
\langle M Q^3 , M \partial_t Q^3_{\neq} \rangle_{H^N} \leq \frac{1}{2}\norm{\sqrt{-\dot M M} Q^3_{\neq}}_{H^N}^2 - \frac{\nu}{2}\norm{\grad_L M Q^3}_{H^N}^2 + \epsilon^3 \mathcal{E}(t),
\end{align*}
where $\mathcal{E}$ is uniformly bounded in $L^1_t$, then we have a bound on \emph{both} $\sqrt{\nu}\norm{\grad_L M Q^3}_{L^2 H^N}$ and $\norm{\sqrt{-\dot M M} Q^3_{\neq}}_{L^2 H^N}$!  
The usefulness of this estimate is emphasized by the following very important lemma (the proof of which is immediate from the definition of $M^2$) which shows how to deduce $L^2$ in time enhanced dissipation without losing any regularity.  
\begin{lemma} \label{lem:postmultED}
There holds for $k \neq 0$
\begin{align*}
1 \lesssim \nu^{-1/6} \sqrt{-\dot{M}^2 M^2}(k,\eta,l) + \nu^{1/3}\abs{k,\eta-kt,l}.  
\end{align*}
As a corollary, the following holds for any $f$ and $\alpha \geq 0$,  
\begin{align*}
\norm{f_{\neq}}_{L^2 H^\alpha} & \lesssim \nu^{-1/6}\left(\norm{\sqrt{-\dot{M}^2 M^2}f_{\neq}}_{L^2H^\alpha} + \nu^{1/2}\norm{\grad_L f_{\neq}}_{L^2 H^\alpha} \right).  
\end{align*}
\end{lemma}

The use of norms with time-decaying norms is quite classical when working in infinite regularity (see e.g. the Cauchy-Kowalevskaya theorems of \cite{Nirenberg72,Nishida77}). 
The use of dissipation-like terms that appear in $L^2$-based infinite regularity estimates goes back to \cite{FoiasTemam89}; see also related ideas in e.g. \cite{LevermoreOliver97,KukavicaVicol09,CGP11,MouhotVillani11} and the references therein.  
The ghost energy of Alinhac \cite{Alinhac01} for quasilinear wave equations uses $O(1)$ time-dependent weights in the norms, and the $O(1)$ multiplier $M$ is a Fourier-side analogue -- this general idea has been used several times \cite{Zillinger2014,BGM15I,BGM15II}. 
Combining ideas like the ghost energy with the Cauchy-Kowalevskaya-type ideas are multipliers such as $m(t,\grad)$, which are \emph{not} $O(1)$ ($m^{-1}$ is bounded only by $O(\nu^{-2/3})$); this is significantly more complicated, as will be clear from the proof. 
In the context of nonlinear mixing, this general idea was introduced for infinite regularity in \cite{BM13} and extended further in \cite{BMV14,BGM15I,BGM15II}.   
However, $m$ is very different from ideas appearing in these infinite regularity works as we must use very differently  the interplay between the dissipation and destabilizing effects. 

\subsection{Bootstrap} \label{sec:boot}

In the following, it will be convenient (for bookkeeping purposes) to introduce 
$$
N = \sigma - 2 > \frac{5}{2}.
$$
First, we have the following standard lemma; we omit the proof for brevity. 
\begin{lemma}[Local existence, continuation, and propagation of analyticity] \label{lem:Cont}
Let $u_{in}$ be divergence free and satisfy \eqref{ineq:smallness}. Then there exists a $T^\star > 0$ independent of $\nu$ such that there is a unique strong solution to \eqref{NSC} $u(t) \in C([0,T^\star];H^{N+2})$ which satisfies the initial data and is real analytic for $t \in (0,T^\star]$. 
Moreover, there exists a maximal time of existence $T_0$ with $T^\star < T_0 \leq \infty$ such that the solution $u(t)$ remains unique and real analytic on $(0,T_0)$, and, if for some $\tau \leq T_0$  we have $\limsup_{t \nearrow \tau} \norm{u(t)}_{H^{N}} < \infty$, then $\tau < T_0$.  
\end{lemma} 

By similar considerations (see \S\ref{sec:coordstuff}), for $\epsilon$ sufficiently small, there are no issues getting estimates on $q^i$, $u^i$, and $\psi$ until $t = 2$. 
\begin{lemma} \label{lem:BootStart}
For $\epsilon \nu^{-3/2}$ sufficiently small and constants $C_0,C_1$ sufficiently large (chosen below), the following estimates hold for $t \in [0,2]$:  
\begin{subequations} \label{ineq:initlayer}
\begin{align}
\norm{q^1(t)}_{H^{N}} + \norm{q^3(t)}_{H^{N}}  & \leq 2C_0\epsilon \\ 
\norm{q^2(t)}_{H^{N}}  & \leq 2\epsilon \\ 
\norm{u^1(t)}_{H^{N+2}} + \norm{u^3(t)}_{H^{N+2}} & \leq 2C_0\epsilon \\ 
\norm{u^2(t)}_{H^{N+2}} & \leq 2\epsilon \\ 
\norm{t\psi(t)}_{H^N} & \leq 2C_1\epsilon.  
\end{align}
\end{subequations} 
\end{lemma} 

Lemma \ref{lem:BootStart} shows that we only need to worry about times $t > 1$, for which we now move to the coordinate system defined in \S\ref{sec:coordchnge}; for details on converting the estimates to and from these coordinates, see \S\ref{sec:coordstuff} below.  
From now on, all time norms are taken over the interval $[1,T]$ unless otherwise stated; that is all norms are defined via 
\begin{align*}
\norm{f}_{L^p H^s} := \norm{ \norm{f(t)}_{H^s} }_{L^p([1,T])}. 
\end{align*}
Fix $C_0$ and $C_1$ large constants determined by the proof below and let $T$ be the largest time $T \geq 1$ such that the following estimates hold on $[1,T]$ (Lemma \ref{lem:BootStart} implies $T > 2$; see \S\ref{sec:coordstuff}): \\
the bounds on $Q$:
\begin{subequations}\label{boundsQ}
\begin{align}
\hspace{-1cm}\norm{\jap{t}^{-1} Q_0^1(t)}_{L^\infty H^N} & \leq  8C_0 \epsilon \label{boundQ10short} \\
\hspace{-1cm} \norm{Q_0^1}_{L^\infty H^N} + \nu^{1/2} \norm{ \grad Q_0^1}_{L^2 H^N} & \leq 8C_0 \epsilon \nu^{-1} \label{boundQ10} \\ 
\hspace{-1cm} \norm{m M Q_{\neq}^1}_{L^\infty H^N} + \nu^{1/2} \norm{m M \grad_L Q_{\neq}^1}_{L^2H^N} + \norm{\sqrt{-\dot{M} M} m Q^1_{\neq}}_{L^2H^N} & \leq 8C_0 \epsilon \nu^{-1/3} \label{boundQ1d} \\ 
\hspace{-1cm} \norm{m^{1/2} M Q^2}_{L^\infty H^N} + \nu^{1/2} \norm{m^{1/2} M \grad_L Q^2}_{L^2H^N} + \norm{\sqrt{-\dot{M} M} m^{1/2} Q^2}_{L^2H^N} & \leq 8\epsilon \label{boundQ2} \\ 
\hspace{-1cm} \norm{m M Q^3}_{L^\infty H^N} + \nu^{1/2} \norm{m M \grad_L Q^3}_{L^2H^N} + \norm{\sqrt{-\dot{M} M} m  Q^3}_{L^2H^N} & \leq  8 C_0\epsilon \label{boundQ3} \\
\hspace{-1cm} \norm{M Q^2_{\neq}}_{L^\infty H^{N-1}} + \nu^{1/2} \norm{M \grad_L Q^2_{\neq}}_{L^2H^{N-1}} + \norm{\sqrt{-\dot{M} M} Q^2_{\neq}}_{L^2 H^{N-1}} & \leq 8\epsilon, \label{boundQ2low}
\end{align}
\end{subequations} 
the bounds on $U$, 
\begin{subequations} \label{boundsU}
\begin{align}
\label{boundU10short} \norm{\jap{t}^{-1} U^1_0}_{L^\infty H^{N-1}} & \leq 8 C_0 \epsilon \\
 \label{boundU10} \norm{U^1_0}_{L^\infty H^{N-1}} + \nu^{1/2} \norm{\nabla U^1_0}_{L^2 H^{N-1}}  & \leq 8 C_0  \epsilon \nu^{-1} \\
\label{boundU2} \|U^2_0\|_{L^\infty H^{N-1}} + \nu^{1/2} \|U^2_0\|_{L^2 H^{N-1}} +  \nu^{1/2} \|\nabla U^2_0 \|_{L^2 H^{N-1}} & \leq 8\epsilon \\
 \label{boundU30}\| U^3_0\|_{L^\infty H^{N-1}} + \nu^{1/2} \|\nabla U^3_0 \|_{L^2 H^{N-1}} &\leq 8 C_0 \epsilon \\
\label{boundU1neq} \norm{M U_{\neq}^1}_{L^\infty H^{N-1}} + \nu^{1/2} \norm{\grad_L M U_{\neq}^1}_{L^2H^{N-1}} + \norm{\sqrt{-\dot{M} M} U^1_{\neq}}_{L^2H^{N-1}} & \leq 8C_0\epsilon \\  
 \label{boundU02decay} \norm{U_0^2}_{L^1L^2} & \leq 8\epsilon \nu^{-1}, 
\end{align}
\end{subequations} 
and the bounds on the coordinate system
\begin{subequations}\label{boundsC}
\begin{align}
\norm{g}_{L^\infty H^{N+2}} + \nu^{1/2} \norm{\grad g}_{L^2H^{N+2}} & \leq 8C_0\epsilon \label{boundg1}\\ 
\norm{t^2 g}_{L^\infty H^{N-1}} + \nu^{1/2}\norm{t^2 \nabla g}_{L^2 H^{N-1}} & \leq 8C_0\epsilon \label{boundg}\\ 
\norm{C}_{L^\infty H^{N+2}} +  \nu^{1/2}\norm{\grad C}_{L^2 H^{N+2}} & \leq 8C_1\epsilon \nu^{-1}.  \label{boundC}
\end{align}
\end{subequations}

The goal is then to prove that $T = +\infty$, which follows immediately from the following (and that all of these norms are continuous in time).  

\begin{prop} \label{propbootstrap}
Assume that $\| u_{in} \|_{H^{N+2}} \leq \epsilon \leq \delta \nu^{3/2}$, $\nu \in (0,1)$, and that, for some $T > 1$, the estimates~\eqref{boundsU}, \eqref{boundsQ}, \eqref{boundsC} hold on $[1,T]$. Then for $\delta$ sufficiently small, these same estimates hold with all the occurrences of $8$ on the right-hand side replaced by $4$.
\end{prop}

That Proposition \ref{propbootstrap} implies Theorem \ref{mainthm} is proved in Lemma \ref{lem:intermed} below. 
The proof of Proposition \ref{propbootstrap} comprises the remainder of the paper. 
Let us briefly comment on the structure of the scheme laid out in the coupled estimates \eqref{boundsQ}, \eqref{boundsU}, and \eqref{boundsC}. 
One of the main subtleties are the two estimates on $Q^2$ in \eqref{boundQ2} and \eqref{boundQ2low}. 
The nonlinear effect of high frequencies can be quite dramatic near the critical times, and one of the leading order nonlinear terms in the $Q^2$ equation, specifically $NLP(1,3,0,\neq)$ in \S\ref{sec:NLPQ2}  below, cannot be bounded uniformly in $H^N$ if we only assume $\epsilon \lesssim \nu^{3/2}$. This term is a very 3D nonlinear interaction involving the Orr mechanism, the stretching of $Q^3$, and the lift-up effect of $U^1_0$ all at once.   
By allowing $Q^2$ to grow near the critical time until the dissipation can balance the growth, quantified by the inclusion of the decaying $m^{1/2}$ in the norm, one can complete the estimate -- hence \eqref{boundQ2}. 
However, the loss in \eqref{boundQ2} due to $m^{1/2}$ is too dramatic to close the scheme, indeed, the $m^{1/2}$ permits a growth on $Q^2$ which in turn limits the amount of inviscid damping on $U^2$, however, this inviscid damping provides a kind of null structure which damps some nonlinear terms that would otherwise be uncontrollable. 
The solution is to pay regularity and get a better uniform estimate at lower frequencies, as expressed in \eqref{boundQ2low} -- the gap of one derivative is roughly analogous to the fact that paying one derivative will give one power of $t^{-1}$ decay in an estimate such as \eqref{ineq:LinID}.  

\subsection{Choice of constants} 
Three constants have not been specified yet: $\epsilon \nu^{-3/2} = \delta$, which appears in the statement of Theorem \ref{mainthm}, and $C_0$ and $C_1$, which appear in the above bootstrap estimates. 
In the course of the proof, we choose them small such that
\begin{align*}
\frac{1}{C_0} + \frac{C_0}{C_1} + C_1 \delta + C_0 \delta < \frac{1}{A},
\end{align*} 
for a universal constant $A = A(\sigma)$ which depends only on $\sigma$. 
Specifically, this means that one first fixes $C_0$, then $C_1$ dependent on $C_0$, and then finally $\delta$ small relative to both.  

\subsection{Estimates following immediately from the bootstrap hypotheses}
This section is to outline some of the consequences of the bootstrap hypotheses. 

The first lemma is a simple result of the Sobolev product law, the geometric series representation \eqref{eq:coefGeom}, and the bootstrap hypotheses (also recall the shorthand \eqref{def:G}). 
\begin{lemma} \label{lem:CoefCtrl} 
Under the bootstrap hypotheses, for $\epsilon \nu^{-1}$ sufficiently small and $1 < s \leq N+2$,  
\begin{align*}
\norm{\psi_y}_{H^{s}}  + \norm{\psi_z}_{H^{s}} + \norm{G}_{H^{s}} \lesssim \norm{\grad C}_{H^{s}}.  
\end{align*}
As a consequence, there holds for all $i,j \in \set{X,Y,Z}$,  
\begin{subequations} \label{ineq:gradtL} 
\begin{align}
\norm{\partial_Y^t f}_{H^s} & \lesssim \norm{\partial_Y^L f}_{H^s} \quad\quad\quad\quad s \leq N+1, \\ 
\norm{\partial_Z^t f}_{H^s} \lesssim \norm{\partial_Z f}_{H^s} + \epsilon \nu^{-1} \norm{\partial_Y^L f}_{H^s} & \lesssim \norm{\grad_L f}_{H^s} \quad\quad\quad\quad s \leq N+1, \\  
\norm{\Delta_t f_{\neq}}_{H^s} + \norm{\partial_i^t \partial_j^t f_{\neq}}_{H^s} & \lesssim \norm{\Delta_L f_{\neq}}_{H^s} \quad\quad\quad\quad s \leq N \\ 
\norm{\Delta_t f_0}_{H^s} + \norm{\partial_i^t \partial_j^t f_0}_{H^s} & \lesssim \norm{\Delta f_0}_{H^s} + \epsilon \nu^{-1} \norm{\grad f_0}_{H^s} \quad\quad s \leq N. 
\end{align}
\end{subequations}
Similarly, by using also Lemma \ref{lem:mcomm}, we have for all $\alpha \in [0,1]$ and $3/2< s \leq N$,   
\begin{align}
\norm{m^{\alpha} \partial_j^t f_{\neq}}_{H^{s}} & \lesssim \left(1 + \norm{\grad C}_{H^{s + 2\alpha}}\right)\norm{\grad_L m^{\alpha} f}_{H^s}. \label{ineq:mdjtf}
\end{align}
\end{lemma}
\begin{remark} 
Note that for $s + 2\alpha \leq N+1$ the leading factor in \eqref{ineq:mdjtf} can be ignored by the $L^\infty H^{N+2}$ control on $C$. 
\end{remark}
 
An important consequence of \eqref{ineq:gradtL} is that in many places, the difference between $\partial_i^t$ and $\partial_i^L$ is irrelevant, however, the difference cannot be neglected everywhere. 
For example, for $s > 3/2$ there holds  
\begin{align}
\jap{f,\partial_Y^t g}_{H^s} \lesssim \norm{\grad_L f}_{H^s} \norm{g}_{H^s} + \norm{\grad^2 C}_{H^s} \norm{f}_{H^s} \norm{g}_{H^s}; \label{ineq:0freqIBP}
\end{align}
indeed this is proved by integrating by parts and Cauchy-Schwarz. 
If $s \leq N$ and the frequency of $f$ is non-zero, then the second term in \eqref{ineq:0freqIBP} is controlled by the first:  
\begin{align}
\jap{f_{\neq},\partial_Y^t g}_{H^s} \lesssim \norm{\grad_L f_{\neq}}_{H^s} \norm{g}_{H^s}. \label{ineq:neqfreqIBP}
\end{align}
However, at the zero frequency, we need both terms in \eqref{ineq:0freqIBP}. 
Similar inequalities hold also for $\partial_Z^t$. 
One also has the following variant for $\alpha \in [0,1]$, which is useful in many places, 
\begin{align}
\jap{m^{\alpha} f_{\neq},m^{\alpha}\left(\partial_Y^t g\right)}_{H^s} \lesssim \norm{\grad_L m^\alpha f_{\neq}}_{H^s} \norm{m^{\min(\alpha,1/2)} g}_{H^s}. \label{ineq:neqfreqIBPm}
\end{align}

The next proposition consists of those estimates which follow directly from the estimates on $Q^i$ and the elliptic lemmas detailed in \S\ref{sec:PEL}.  
These elliptic lemmas provide the technical tools for understanding $\Delta_t^{-1}$, important for recovering $U^i$ from $Q^i = \Delta_t U^i$. 

\begin{prop}[Basic a priori estimates on the velocity in $H^N$] \label{lem:BasicApriori} 
Under the bootstrap hypotheses, for $\epsilon \nu^{-3/2}$ sufficiently small, the following additional estimates hold: 
\begin{subequations}
\begin{align}
\norm{\jap{t}^{-1} U_0^1(t)}_{L^\infty H^{N+2}} & \lesssim \epsilon \\ 
\norm{U_0^1}_{L^\infty H^{N+2}} + \nu^{1/2} \norm{\grad U_0^1}_{L^2 H^{N+2}} &  \lesssim \epsilon \nu^{-1} \\
\norm{U_0^2}_{L^\infty H^{N+2}} + \nu^{1/2} \norm{\grad U_0^2}_{L^2 H^{N+2}} &  \lesssim \epsilon \\ 
\norm{U_0^3}_{L^\infty H^{N+2}} + \nu^{1/2} \norm{\grad U_0^3}_{L^2 H^{N+2}} &  \lesssim \epsilon \\
\norm{U_{\neq}^1}_{L^\infty H^N} + \nu^{1/2} \norm{\grad_L U_{\neq}^1}_{L^2 H^N} + \norm{\sqrt{-\dot{M} M} U_{\neq}^1}_{L^2 H^N}  & \lesssim \epsilon \nu^{-1/3} \label{ineq:U1Boot} \\ 
\norm{U_{\neq}^2}_{L^\infty H^N} + \nu^{1/2} \norm{\grad_L U_{\neq}^2}_{L^2 H^N} + \norm{\sqrt{-\dot{M} M} U_{\neq}^2}_{L^2 H^N} & \lesssim \epsilon \label{ineq:U2Boot} \\ 
\norm{U_{\neq}^3}_{L^\infty H^N} + \nu^{1/2} \norm{\grad_L U_{\neq}^3}_{L^2 H^N} + \norm{\sqrt{-\dot{M} M} U_{\neq}^3}_{L^2 H^N} & \lesssim \epsilon \label{ineq:U3Boot} \\ 
\norm{m \Delta_L U_{\neq}^1}_{L^\infty H^N} + \nu^{1/2} \norm{m \grad_L \Delta_L  U_{\neq}^1}_{L^2 H^N} + \norm{\sqrt{-\dot{M} M} m \Delta_L  U_{\neq}^1}_{L^2 H^N}  & \lesssim \epsilon \nu^{-1/3} \label{ineq:mU1Boot} \\ 
\norm{m^{1/2}\Delta_L U_{\neq}^2}_{L^\infty H^N} + \nu^{1/2} \norm{m^{1/2}\grad_L \Delta_L  U_{\neq}^2}_{L^2 H^N} + \norm{\sqrt{-\dot{M} M} m^{1/2}\Delta_L  U_{\neq}^2}_{L^2 H^N} & \lesssim \epsilon \label{ineq:mU2Boot} \\ 
\norm{m\Delta_L U_{\neq}^3}_{L^\infty H^N} + \nu^{1/2} \norm{m\grad_L \Delta_L  U_{\neq}^3}_{L^2 H^N} + \norm{\sqrt{-\dot{M} M}m \Delta_L  U_{\neq}^3}_{L^2 H^N} & \lesssim \epsilon. \label{ineq:mU3Boot} 
\end{align}
\end{subequations} 
\end{prop} 
\begin{proof} 
The estimates on the zero frequencies follow rapidly from Lemma \ref{lem:ZeroModePEL} and the bootstrap hypotheses. 

By \eqref{ineq:mDelTrick}, the estimates \eqref{ineq:mU1Boot}, \eqref{ineq:mU2Boot}, \eqref{ineq:mU3Boot} imply \eqref{ineq:U1Boot}, \eqref{ineq:U2Boot}, \eqref{ineq:U3Boot}. 
The estimates \eqref{ineq:mU1Boot}, \eqref{ineq:mU2Boot}, \eqref{ineq:mU3Boot} follow from applying Lemmas \ref{lem:BasicPEL}, \ref{lem:EDPEL}, and \ref{lem:dotMPEL} and using the bootstrap hypotheses on $Q$ and $C$.   
\end{proof}

The next proposition details the inviscid damping of $U^2$ and the enhanced dissipation.  

\begin{prop} \label{prop:EDID}
Under the bootstrap hypotheses, the following additional estimates hold: 
\begin{itemize} 
\item the enhanced dissipation of $Q^i$: 
\begin{subequations} \label{ineq:QiDecay}
\begin{align}
\norm{m Q^1_{\neq}}_{L^2 H^N} & \lesssim \epsilon \nu^{-1/2} \\ 
\norm{m^{1/2} Q^2_{\neq}}_{L^2 H^N} + \norm{Q^2_{\neq}}_{L^2 H^{N-1}} & \lesssim \epsilon \nu^{-1/6}\\ 
\norm{m Q^3_{\neq}}_{L^2 H^N} & \lesssim \epsilon \nu^{-1/6}; 
\end{align}
\end{subequations} 
\item the enhanced dissipation and inviscid damping of $U^i$:  
\begin{subequations} 
\begin{align}
\norm{\grad_L U^2_{\neq}}_{L^2 H^N} & \lesssim \epsilon \nu^{-1/6}  \label{ineq:gradU2HN} \\ 
\norm{\grad_L U^2_{\neq}}_{L^2 H^{N-1}} & \lesssim \epsilon \label{ineq:gradU2} \\
\norm{\Delta_{X,Z} U^3_{\neq}}_{L^2 H^N} & \lesssim \epsilon \nu^{-1/6}  \label{ineq:U3Decay} \\ 
\norm{\Delta_{X,Z} U^1_{\neq}}_{L^2 H^{N}} & \lesssim \epsilon \nu^{-1/2} \label{ineq:U1Decay}  \\ 
\norm{U^1_{\neq}}_{L^2 H^{N-1}} & \lesssim \epsilon \nu^{-1/6}; \label{ineq:U1Decay2} 
\end{align}
\end{subequations}
\item the enhanced dissipation of $t U^i$: 
\begin{subequations} \label{ineq:tdecay}
\begin{align}
\norm{t \partial_X U^{1}_{\neq}}_{L^2 H^{N-1}} & \lesssim \epsilon \nu^{-5/6} \\
\norm{t \partial_XU^{1}_{\neq}}_{L^2 H^{N-2}} & \lesssim \epsilon \nu^{-1/2} \\
\norm{t \partial_X U^{3}_{\neq}}_{L^2 H^{N-1}} & \lesssim \epsilon \nu^{-1/2} \\
\norm{t \partial_X U^{2}_{\neq}}_{L^2 H^{N-2}} & \lesssim \epsilon \\
\norm{t \partial_X U^{2}_{\neq}}_{L^2 H^{N-1}} & \lesssim \epsilon \nu^{-1/6}. 
\end{align}
\end{subequations}
\end{itemize} 
\end{prop} 
\begin{proof} The enhanced dissipation estimates on $Q^i$ follow from Lemma~\ref{lem:postmultED} and the bootstrap estimates on $Q$.
Turning to~\eqref{ineq:gradU2HN},  by \eqref{mlowbound} and Lemma \ref{lem:postmultED}, 
\begin{align*}
\norm{\grad_L U^2_{\neq}}_{H^N} & \lesssim \norm{m^{1/2} \Delta_L U^2}_{H^N}  \\ 
& \lesssim \nu^{-1/6}\left(\norm{m^{1/2} \sqrt{-\dot{M} M} \Delta_L U^2}_{H^N} + \nu^{1/2}\norm{m^{1/2} \grad_L \Delta_L U^2}_{H^N}\right), 
\end{align*}
from which the estimate follows from Proposition \ref{lem:BasicApriori}. 
For~\eqref{ineq:gradU2}, we use the definition of $M$: 
\begin{align*}
\norm{\grad_L U^2_{\neq}}_{H^{N-1}} & \lesssim \norm{\sqrt{-\dot{M} M} \Delta_L U^2}_{H^{N-1}}, 
\end{align*}
from which the estimate follows from Lemma \ref{lem:dotMPEL} and the a priori estimate \eqref{boundQ2low}.  

To deduce \eqref{ineq:U3Decay}, we use \eqref{ineq:mDelTrick} followed by Lemma \ref{lem:postmultED} to derive 
\begin{align*}
\norm{\Delta_{X,Z} U^3_{\neq}}_{H^N} & \lesssim \norm{m \Delta_L U^3_{\neq}}_{H^N}  \\ 
& \lesssim \nu^{-1/6}\left(\norm{m \sqrt{- \dot{M} M} \Delta_L U^3_{\neq}}_{H^N} + \nu^{1/2}\norm{m \grad_L \Delta_L U^3_{\neq}}_{H^N}\right);   
\end{align*} 
after which the estimates follow from Proposition \ref{lem:BasicApriori}. 
The estimates on $U^1$ in \eqref{ineq:U1Decay} and \eqref{ineq:U1Decay2} follow similarly. 

Turn next to the enhanced dissipation estimates involving powers of $t$ in \eqref{ineq:tdecay}.  
For example, we have by \eqref{ineq:mDelTrick} and $\abs{kt} \lesssim \jap{\eta-kt}\jap{\eta}$: 
\begin{align*}
\norm{t\partial_X  U^1_{\neq}}_{H^{N-1}} & \lesssim \norm{\grad_L m \Delta_L U^1_{\neq}}_{H^N};  
\end{align*}
and similarly for $t\partial_X U^3$. 
For $t \partial_X U^2$ we again use $\abs{kt} \lesssim \jap{\eta-kt}\jap{\eta}$: 
\begin{align*}
\norm{t\partial_X U^2_{\neq}}_{H^{N-2}} \lesssim \norm{\grad_L U^2_{\neq}}_{H^{N-1}}, 
\end{align*}
which is then controlled by \eqref{ineq:gradU2}; similarly for the analogous inequality in $H^{N-1}$. 
\end{proof} 

In what follows we will use the shorthand 
\begin{subequations} \label{ABshort} 
\begin{align}
A & = m^{1/2} M \jap{D}^N \\
B & = m M \jap{D}^N. 
\end{align}
\end{subequations}

\subsection{Equivalence of coordinate systems} \label{sec:coordstuff}
Coordinate systems of the general type \eqref{def:cords} have been used in \cite{BM13,BMV14,BGM15I,BGM15II} and here we may follow a similar scheme for how to transfer information from one coordinate system to the other; we will give a sketch for completeness.  
In Sobolev regularity the technical details are significantly simpler as compositions behave well in finite regularity classes. 
In particular, we have the following composition lemma; if $s' \in \mathbb N$ this is immediate from the Fa\'a di Bruno formula and Sobolev embedding, for fractional $s$, see e.g. \cite{Inci2013} for a proof.  
\begin{lemma}[Sobolev composition] \label{lem:SobComp} 
Let $s > 5/2$, $s \geq s' \geq 0$, $f \in H^{s'}$, and $g \in H^s$ be such that $\norm{\grad g}_{\infty} < 1$. Then, there holds 
\begin{align*}
\norm{f \circ (Id + g)}_{H^{s'}} \lesssim_{\norm{g}_{H^s}, 1 - \norm{\grad g}_{\infty},s,s'} \norm{f}_{H^{s'}}. 
\end{align*}  
\end{lemma} 

We also need a Sobolev inverse function theorem, which follows by straightforward arguments using Lemma \ref{lem:SobComp}.  

\begin{lemma}[Sobolev inverse function theorem] \label{lem:SobInv} 
Let $s > 5/2$. Then, there exists an $\epsilon_0 = \epsilon_0(s)$ such that if $\norm{\alpha}_{H^s} \leq \epsilon_0$, then there exists a unique solution $\beta$ to 
\begin{align*}
\beta(y) = \alpha(y + \beta(y)),
\end{align*}
which satisfies $\norm{\beta}_{H^s} \lesssim \epsilon_0$. 
\end{lemma} 

The next step is to prove Lemma \ref{lem:BootStart} and also deduce that we may take $T > 1$, the $T$ such that the bootstrap hypotheses \eqref{boundsQ}, \eqref{boundsU}, \eqref{boundsC} hold.  
Hence, we do not need to worry about the coordinate singularity at $t = 0$.    

\begin{lemma} \label{lem:BootStart2}
For $\epsilon \nu^{-3/2}$ sufficiently small, Lemma \ref{lem:BootStart} holds, we may take $2 \leq T$ (defined in \S\ref{sec:boot} above), and for $t \leq 2$, the inequalities \eqref{boundsQ}, \eqref{boundsU}, \eqref{boundsC} all hold with constant $2$ instead of $8$.   
\end{lemma} 
\begin{proof} 
As in analogous lemmas in \cite{BM13,BGM15I}, the proof is done by using the linearized coordinate transform. Indeed, define 
\begin{subequations} \label{def:barxqh} 
\begin{align} 
\bar{x} & = x-ty \\ 
h^i(t,\bar{x},y,z) & = q^i(t,\bar{x} + ty,y,z) \\ 
v^i(t,\bar{x},y,z) & = u^i(t,\bar{x} + ty,y,z);
\end{align}
\end{subequations} 
note that $v^i = \Delta_L^{-1}h^i$. These satisfy natural analogues of \eqref{def:MainSys} and \eqref{eq:momentum}. 
 Using standard (inviscid) energy methods, it is easy to propagate $H^N$ regularity on these unknowns to $t = 2$ (or any other fixed, finite time) by choosing $\epsilon$ sufficiently small.   
Next, we need to solve for $(\bar{x},y,z)$ in terms of $(X,Y,Z)$ and then apply Lemma \ref{lem:SobComp}. 
From \eqref{def:psi} it is straightforward via classical energy methods to derive $\norm{t \psi}_{H^{N+2}} \lesssim \epsilon$ for $t \in [0,2]$.
 For $t \in [1/2,2]$ this yields good estimates on $\psi(t,y,z) = Y(t,y,z) - y$ and $X(t,x,y,z) = \bar{x}(t,x,y) - t\psi(t,y,z)$. We then write 
 \begin{align*} 
\bar{x}(t,X,Y,Z) & = X + t\psi(t,y(t,Y,Z),z(t,Y,Z)) \\ 
y(t,Y,Z) & = Y + \psi(t,y(t,Y,Z),z(t,Y,Z)) \\ 
z(t,Y,Z) & = Z,  
\end{align*}
apply Lemma \ref{lem:SobInv} to solve for $(\bar{x},y,z)$ in terms of $X,Y,Z$ and then Lemma \ref{lem:SobComp} and \eqref{def:barxqh} complete the lemma for $\epsilon \nu^{-1}$ sufficiently small.   
\end{proof}

In order to move information back to the original variables, as in \cite{BM13,BMV14,BGM15I}, we first move to the coordinate system $(X,y,z)$. Hence, write $\bar{q}^i(t,X,y,z)  = Q^i(t,X,Y(t,y,z),Z)$ and $\bar{u}^i(t,X,y,z) = U^i(t,X,Y(t,y,z),Z)$ (recall that $Z = z$). 
This lemma also proves that Proposition \ref{propbootstrap} implies Theorem \ref{mainthm}. 

\begin{lemma} \label{lem:intermed} 
For $\epsilon < \delta\nu^{3/2}$ with $\delta$ sufficiently small, the bootstrap hypotheses imply that all the estimates in Propositions \ref{lem:BasicApriori} and \ref{prop:EDID} hold also for $\bar{q}^i$ and $\bar{u}^i$ (with different implicit constants). 

In particular, for $\epsilon \nu^{-3/2}$ sufficiently small, Proposition \ref{propbootstrap} implies Theorem \ref{mainthm}. 
\end{lemma} 
\begin{proof} 
Notice that $Z(y,z) = z$ and $Y(y,z) - y = \psi(y,z)$, and hence we need estimates on $\psi$, however, from \eqref{boundsC}, we only have estimates on $C$, $\psi_y$ and $\psi_z$ in $(Y,Z)$ coordinates. 
Hence, we need to solve for $y = y(t,Y,z)$. To this end, write $Y - y = C(t,Y,Z) = C(t,y + (Y-y))$ and then apply Lemma \ref{lem:SobInv} to solve for $Y - y(t,Y,z) = \beta(t,Y,z)$. 
Lemma \ref{lem:SobInv} moreover provides the uniform estimate $\norm{y(t,Y,z) - Y}_{H^{N+2}} \lesssim \epsilon \nu^{-1}$. 
With the bootstrap hypotheses, \eqref{def:cords}, and Lemma \ref{lem:SobComp}, this completes the lemma.   
Indeed, by the definition of $X$ in \eqref{def:cords}, Theorem \ref{mainthm} follows immediately. 
\end{proof}

\section{Energy estimates on $Q^2$} \label{EEQ2}
In this section, we prove that, under the assumptions of Proposition~\ref{propbootstrap} (in particular, the bootstrap assumptions~\eqref{boundsU}, \eqref{boundsQ}, \eqref{boundsC}), the inequalities~\eqref{boundQ2} and \eqref{boundQ2low} hold, with $8$ replaced by $4$ on the right-hand side.

\subsection{$H^N$ estimate on $Q^2$} \label{sec:HNQ2}
An energy estimate gives (recall the shorthand \eqref{ABshort})  
\begin{align*}
& \frac{1}{2} \| M m^{1/2} Q^2(T) \|_{H^N}^2 +\nu  \| \nabla_L M m^{1/2} Q^2 \|_{L^2  H^N}^2 + \| \sqrt{-\dot{M} M} m^{1/2} Q^2 \|_{L^2  H^N}^2 \\
& \leq  \frac{1}{2} \| M m^{1/2} Q^2(1) \|_{H^N}^2 + \int_1^T \int A Q^2 A \left[-  (\widetilde{U}_0 \cdot \nabla + U_{\neq} \cdot \nabla_t) Q^2  - Q^j \partial^t_j U^2 \right. \\
& \left. \qquad \qquad \qquad \qquad \qquad \qquad \qquad  - \partial^t_{i} U^j \partial^t_{ij}U^2 + \partial_Y^t (\partial^t_i U^j \partial^t_j U^i) + \nu (\widetilde{\Delta}_t - \Delta_L) Q^2 \right]\,dV\,dt \\
& = \frac{1}{2} \| M m^{1/2} Q^2(1) \|_{H^N}^2 + \mathcal{T} + NLS1 + NLS2 + NLP + DE.
\end{align*}

\subsubsection{Transport nonlinearity} \label{toucan} 
Decompose the transport nonlinearity by frequency:
\begin{align*}
\mathcal{T} & = \int_1^T \int  A Q^2 A\left( \tilde U_0 \cdot \grad Q^2_0  + \tilde U_0 \cdot \grad Q^2_{\neq} \right) dV dt + \int  A Q^2 A\left( U_{\neq} \cdot \grad_t Q^2_0 + U_{\neq} \cdot \nabla_t Q^2_{\neq} \right) dV dt \\ 
& = \mathcal{T}_{00} + \mathcal{T}_{0\neq} + \mathcal{T}_{\neq0} + \mathcal{T}_{\neq \neq}. 
\end{align*}
Further decompose $\mathcal{T}_{00}$ into
\begin{align*} 
\mathcal{T}_{00} = \int_1^T \int \langle D \rangle^N Q_0^2  \langle D \rangle^N \left( g \partial_Y Q^2_0\right) \,dV\,dt + \int \langle D \rangle^N Q_0^2  \langle D \rangle^N \left( U^3_0 \partial_Z Q^2_0\right) \,dV\,dt = \mathcal{T}_{00}^2 + \mathcal{T}_{00}^3.
\end{align*} 
To bound $ \mathcal{T}_{00}^2$, split $g$ into low and high frequencies: 
\begin{align*}
\mathcal{T}_{00}^2 & = \int_1^T \int  \langle D \rangle^N Q_0^2  \langle D \rangle^N \left( P_{\leq 1} g \partial_Y Q^2_0 \right) \,dV\,dt + \int_1^T \int  \langle D \rangle^N Q_0^2  \langle D \rangle^N \left( P_{>1} g \partial_Y Q^2_0 \right) \,dV\,dt \\
& \lesssim \| Q_0^2 \|_{L^\infty H^N} \| g \|_{L^2 L^2} \| \nabla Q^2_0 \|_{L^2 H^N} +  \| Q_0^2 \|_{L^\infty H^N} \| \nabla g \|_{L^2 H^N} \| \nabla Q^2_0 \|_{L^2 H^N} \\
& \lesssim \epsilon^3 \nu^{-1/2-1/2} = \epsilon^3 \nu^{-1}, 
\end{align*}
where the last line followed from the bootstrap hypotheses (note \eqref{boundg} is used to deduce $\norm{g}_{L^2 L^2} \lesssim \epsilon$). 
To bound $ \mathcal{T}_{00}^3$, observe that either the first $Q_0^3$ factor, or the $U_0^3$ factor, must have nonzero $Z$ frequency - or the contribution is zero. Therefore (using also Proposition \ref{lem:BasicApriori}), 
\begin{align*} 
\mathcal{T}_{00}^3 \lesssim \| Q^2_0 \|_{L^\infty H^N} \| \nabla U^3_0 \|_{L^2 H^N} \| \nabla Q^2_0 \|_{L^2 H^N} +  \| U^3_0 \|_{L^\infty H^N} \| \nabla Q^2_0 \|_{L^2 H^N}^2 \lesssim \epsilon^3 \nu^{-1/2-1/2} = \epsilon^3 \nu^{-1}.
\end{align*} 
For the $\mathcal{T}_{0\neq}$ term, we apply the paraproduct decomposition defined above in \S\ref{sec:LPP}:  
\begin{align*}
\mathcal{T}_{0\neq} & = \int_1^T \int  A Q^2_{\neq} A\left( \left(\tilde{U}_0\right)_{Hi} \cdot \left(\grad Q^2_{\neq}\right)_{Lo} \right) dV dt + \int_1^T \int  A Q^2_{\neq} A\left( \left(\tilde{U}_0\right)_{Lo} \cdot \left(\grad Q^2_{\neq}\right)_{Hi} \right) dV dt \\ 
& = \mathcal{T}_{0\neq;HL} + \mathcal{T}_{0\neq;LH}. 
\end{align*} 
Consider the $LH$ term first, which we write out as follows: 
\begin{align*} 
\mathcal{T}_{0\neq;LH} &  = \int_1^T \int  A Q^2 A\left( g_{Lo} (\partial_{Y}-t\partial_X) (Q^2_{\neq})_{Hi} + (U^3_0)_{Lo} (\partial_{Z} Q^2_{\neq})_{Hi} \right) dV dt \\ & \quad + \int_1^T \int  A Q^2 A\left( g_{Lo} t \partial_{X} (Q^2_{\neq})_{Hi} \right) dV dt. 
\end{align*}
By \eqref{LPprod}, \eqref{mlowbound}, and the bootstrap hypotheses,  
\begin{align*} 
\mathcal{T}_{0\neq;LH} & \lesssim \nu^{-1/3}\norm{A Q^2_{\neq}}_{L^2 L^2} \left(\norm{\jap{t} g}_{L^\infty H^{3/2+}} + \norm{U_0^3}_{L^\infty H^{3/2+}}\right)\norm{\grad_L A Q^2}_{L^2 L^2} \lesssim \epsilon^3 \nu^{-1},  
\end{align*}
where note that we applied \eqref{ineq:QiDecay}. 
Similarly, for the $HL$ term we have by \eqref{mlowbound} and Proposition \ref{prop:EDID} (using $N > 5/2$), 
\begin{align*} 
\mathcal{T}_{0\neq;HL} & \lesssim \norm{A Q^2_{\neq}}_{L^2 L^2} \left(\norm{g}_{L^\infty H^N} + \norm{U_0^3}_{L^\infty H^N}\right)\norm{\grad Q^2_{\neq}}_{L^2 H^{3/2+}} \lesssim \epsilon^3\nu^{-2/3}. 
\end{align*}
Consider next $\mathcal{T}_{\neq 0}$. By the product rule and the bootstrap hypotheses, 
\begin{align*}
\mathcal{T}_{\neq 0} & \lesssim \norm{A Q^2_{\neq}}_{L^\infty L^2} \norm{U^{2,3}_{\neq}}_{L^2 H^N}\norm{\grad Q^2_0}_{L^2 H^N} \lesssim \epsilon^3 \nu^{-2/3}.  
\end{align*} 
Consider finally $\mathcal{T}_{\neq\neq}$, by \eqref{mlowbound}, 
\begin{align*}
\mathcal{T}_{\neq\neq} & \lesssim \nu^{-1/3}\norm{A Q^2}_{L^\infty L^2} \norm{U_{\neq}}_{L^2 H^N} \norm{\grad_L A Q^2}_{L^2 L^2} \lesssim \epsilon^3\nu^{-4/3}. 
\end{align*}

\subsubsection{Nonlinear pressure terms} \label{sec:NLPQ2} 
Recall the shorthands defined in \S\ref{sec:short}.
Consider first the $NLP(0,0)$ terms, which are straightforward.
We first apply \eqref{ineq:0freqIBP}, which results in an error term when the derivative lands on the coefficients, and then we apply Lemma \ref{lem:CoefCtrl}: 
\begin{align*}
NLP(i,j,0,0) & \lesssim \norm{\grad Q^2_0}_{L^2 H^N} \| \grad U^{2,3}_0 \|_{L^2 H^N} \| \grad U_0^{2,3} \|_{L^\infty H^{3/2+}} \\ & \quad + \norm{\Delta C}_{L^2 H^N} \norm{Q^2_0}_{L^\infty H^N} \| \grad U^{2,3}_0 \|_{L^2 H^N} \| \grad U_0^{2,3} \|_{L^\infty H^{3/2+}} \\
& \lesssim \epsilon^3 \nu^{-1}.
\end{align*} 

Next turn to the $NLP(0,\neq,i,j)$ terms, which include one of the leading order nonlinear terms, $NLP(1,3,0,\neq)$. 
Consider this problematic term first and expand with a paraproduct as described in \S\ref{sec:LPP}, 
\begin{align*}
NLP(1,3,0,\neq) & = \int_1^T \int A Q^2 A \partial_Y^t\left( \left(\partial^t_Z U_0^1\right)_{Hi} \left(\partial_X U_{\neq}^3\right)_{Lo} \right) dV dt \\ & \quad + \int_1^T \int A Q^2 A \partial_Y^t\left( \left(\partial^t_Z U_0^1\right)_{Lo} \left(\partial_X U_{\neq}^3\right)_{Hi} \right) dV dt \\ 
& =  P_{HL} + P_{LH}. 
\end{align*} 
For the $LH$ term we have, using \eqref{ineq:neqfreqIBPm}, Lemma~\ref{lem:mcomm}, and the inequality $|m^{1/2} \partial_X| \lesssim m \sqrt{-\dot M M} (-\Delta_L)$ which follows from \eqref{ineq:mDelTrick},
\begin{align*}
P_{LH} & \lesssim \| \nabla_L A Q^2 \|_{L^2 L^2} \| \partial^t_Z U^1_0 \|_{L^\infty H^{5/2+}} \| m^{1/2} \partial_X U^3_{\neq} \|_{L^2 H^N} \\
& \lesssim \| \nabla_L A Q^2 \|_{L^2 L^2} \| \partial^t_Z U^1_0 \|_{L^\infty H^{5/2+}} \norm{ m \sqrt{-\dot{M} M} \Delta_L U^3_{\neq} }_{L^2 H^N} \\
& \lesssim \epsilon^3 \nu^{-1/2-1} =  \epsilon^3\nu^{-3/2}, 
\end{align*}
which suffices for $\epsilon \nu^{-3/2} \ll 1$; hence this term uses sharply the smallness requirement. 
For the $HL$ term we can apply \eqref{ineq:neqfreqIBP} and deduce using \eqref{ineq:tdecay}, 
\begin{align*}
P_{HL} & \lesssim \norm{\grad_L m^{1/2} M Q^2}_{L^2 H^N} \norm{\jap{t}^{-1} \grad U_0^1}_{L^\infty H^N} \norm{\jap{t}\partial_X U^3_{\neq}}_{L^2 H^{3/2+}}  \lesssim \nu^{-1}\epsilon^3. 
\end{align*}
This completes $NLP(1,3,0,\neq)$ term; $NLP(1,2,0,\neq)$ is similar. 

Consider next $NLP(i,j,0,\neq)$ with $i,j \neq 1$. 
For these terms we do not need a sophisticated argument; using only \eqref{ineq:mDelTrick} and \eqref{ineq:neqfreqIBP}, 
\begin{align*}
NLP(i,j,0,\neq)\mathbf{1}_{i,j,\neq 1}  & \lesssim \norm{\grad_L A Q^2}_{L^2 L^2} \norm{\grad U_0^j}_{L^\infty H^N} \norm{\grad_L U^i_{\neq}}_{L^2 H^N} \mathbf{1}_{i,j,\neq 1} \\ 
& \lesssim \norm{\grad_L A Q^2}_{L^2 L^2} \norm{\grad U_0^j}_{L^\infty H^N} \norm{\grad_L \Delta_L m U^i_{\neq}}_{L^2 H^N} \mathbf{1}_{i,j,\neq 1} \\ 
& \lesssim  \epsilon^3\nu^{-1/2-1/2} = \epsilon^3\nu^{-1}. 
\end{align*}

Turn next to $NLP(i,j,\neq,\neq)$. 
We expand with a paraproduct and by symmetry, we only have to consider the case when $i$ is in ``high frequency''. 
By \eqref{ineq:0freqIBP} (note that the leading factor could have zero $X$ frequency), 
\begin{align*}
NLP(i,j,\neq,\neq) & = \int_1^T \int A Q^2 A \partial_Y^t\left( \left(\partial^t_j U_{\neq}^i\right)_{Hi} \left(\partial^t_i U_{\neq}^j\right)_{Lo} \right) dV dt \\ 
& \lesssim \left(\norm{\grad_L A Q^2}_{L^2 L^2} +  \norm{A Q^2}_{L^\infty L^2}\norm{\grad C}_{L^2 H^{N+1}}\right) 
 \norm{\grad_L U^i_{\neq}}_{L^\infty H^N} \norm{ \partial^t_i U_{\neq}^j}_{L^2 H^{3/2+}} \\ 
& \lesssim \nu^{-1/3}\left(\norm{\grad_L A Q^2}_{L^2 L^2} +  \norm{A Q^2}_{L^\infty L^2}\norm{\grad C}_{L^2 H^{N+1}}\right) \norm{\Delta_L m U^i_{\neq}}_{L^\infty H^N} \norm{ \partial^t_i U_{\neq}^j}_{L^2 H^{N-1}}
\end{align*}
At this point, we distinguish two cases: if on the one hand $i =1$, by Proposition \ref{prop:EDID}, 
\begin{align*}
NLP(1,j,\neq,\neq) &  \lesssim  \nu^{-1/3} \left(\norm{\grad_L A Q^2}_{L^2 L^2} +  \norm{A Q^2}_{L^\infty L^2}\norm{\grad C}_{L^2 H^{N+1}}\right) \| \Delta_L m U^1_{\neq} \|_{L^\infty H^N} \| \partial_X U^j \|_{L^2 H^{N-1}} \\
&  \lesssim  \epsilon^3\nu^{-1/3-1/2-1/3-1/6} = \epsilon^3 \nu^{-4/3},
\end{align*}
while on the other hand if $i \neq 1$, by Proposition \ref{prop:EDID},
\begin{align*}
NLP(i,j,\neq,\neq) & \lesssim \nu^{-1/3} \left(\norm{\grad_L A Q^2}_{L^2 L^2} +  \norm{A Q^2}_{L^\infty L^2}\norm{\grad C}_{L^2 H^{N+1}}\right)\| \Delta_L m U^{2,3}_{\neq} \|_{L^\infty H^N} \| \nabla_L U^j \|_{L^2 H^{N-1}}\\
&  \lesssim  \epsilon^3\nu^{-1/3-1/2-1/2} = \epsilon^3 \nu^{-4/3}.
\end{align*}
This completes the nonlinear pressure terms. 

\subsubsection{Nonlinear stretching terms} \label{penguin} 
Consider $NLS1$, starting with $NLS1(0,0,j)$ (note that $j\neq 1$) and using \eqref{ineq:gradtL},  
\begin{align*}
NLS1(0,0,j) & \lesssim \norm{Q^2_0}_{L^\infty H^N} \norm{\Delta_t U^j_0}_{L^2 H^N} \norm{\grad U^2_0}_{L^2 H^N} \\ 
& \lesssim \norm{Q^2_0}_{L^\infty H^N} \norm{\grad U^j_0}_{L^2 H^{N+1}} \norm{\grad U^2_0}_{L^2 H^N} \\ 
& \lesssim \epsilon^3 \nu^{-1}. 
\end{align*}
By Propositions \ref{lem:BasicApriori} and \ref{prop:EDID} (note that $j = 1$ is permitted),   
\begin{align*}
NLS1(0,\neq,j) & \lesssim \norm{A Q^2_{\neq}}_{L^2 L^2} \norm{Q^j_{0}}_{L^\infty H^N} \norm{\grad_L U^2_{\neq}}_{L^2 H^N} \lesssim \epsilon^3 \nu^{-4/3},  
\end{align*}
and, using \eqref{mlowbound} and Proposition \ref{prop:EDID} (note here $j \neq 1$),  
\begin{align*}
NLS1(\neq,0,j) & \lesssim \norm{A Q^2_{\neq}}_{L^2 L^2} \norm{Q^j_{\neq}}_{L^2 H^N} \norm{\grad U^2_{0}}_{L^\infty H^N} \lesssim \epsilon^3 \nu^{-1}.   
\end{align*}
Similarly (note $j = 1$ is permitted),
\begin{align*}
NLS1(\neq,\neq,j) & \lesssim \norm{A Q^2}_{L^\infty L^2} \norm{Q^j_{\neq}}_{L^2 H^N} \norm{\grad_L U^2_{\neq}}_{L^2 H^N} \lesssim \epsilon^3 \nu^{-4/3}.   
\end{align*}

Recall the second stretching term, $NLS2$, is written $\partial_i^t U^j \partial^t_{ij}U^2$. 
The contributions from the $NLS2(0,0)$ terms are treated in the same manner as $NLS1(0,0)$ above and are hence omitted for brevity. 
Turning to the non-zero frequencies, we have by \eqref{ineq:mDelTrick}, \eqref{mlowbound}, and Proposition \ref{lem:BasicApriori}  
\begin{align*}
NLS2(0,\neq,j) & \lesssim \norm{A Q^2_{\neq}}_{L^2 L^2} \norm{\partial_i^t U_0^j}_{L^\infty H^N} \norm{\partial_{ij}^t U^2_{\neq}}_{L^2 H^N}\mathbf{1}_{j \neq 1} \\ & \quad + \norm{A Q^2_{\neq}}_{L^2 L^2} \norm{\partial_i^t U_0^1}_{L^\infty H^N} \norm{\partial_X \partial_{i}^t U^2_{\neq}}_{L^2 H^N} \\   
& \lesssim \nu^{-1/3}\norm{A Q^2_{\neq}}_{L^2 L^2} \norm{\grad U_0^j}_{L^\infty H^N} \norm{m^{1/2} \Delta_L U^2_{\neq}}_{L^2 H^N}\mathbf{1}_{j \neq 1}  \\ & \quad + \norm{A Q^2_{\neq}}_{L^2 L^2} \norm{\grad U_0^1}_{L^\infty H^N} \norm{m^{1/2}\Delta_L U^2_{\neq}}_{L^2 H^N} \\   
& \lesssim \epsilon^3 \nu^{-1/3-1/6-1/6} + \epsilon^3 \nu^{-1/6-1-1/6} \lesssim \epsilon^3 \nu^{-4/3}.   
\end{align*}
Similarly, we have (note in this case $j \neq 1$), 
\begin{align*}
NLS2(\neq,0,j) & \lesssim \norm{A Q^2_{\neq}}_{L^2 L^2} \norm{\partial_i^t U_{\neq}^j}_{L^2 H^N} \norm{\grad U^2_{0}}_{L^\infty H^{N+1}} \lesssim \epsilon^3 \nu^{-1/6-1/2} = \epsilon^3 \nu^{-2/3},  
\end{align*}
and
\begin{align*}
NLS2(\neq,\neq,j) & \lesssim \norm{A Q^2}_{L^\infty L^2} \norm{\partial_i^t U_{\neq}^j}_{L^2 H^N} \norm{\Delta_L U^2_{\neq}}_{L^2 H^N} \lesssim \epsilon^3 \nu^{-1/2-1/3-1/3-1/6} = \epsilon^3\nu^{-4/3}.
\end{align*}

\subsubsection{Dissipation error terms} \label{sec:DEQ2}
These terms are given by (recall the short-hand \eqref{def:G}): 
\begin{align*}
DE & = \nu\int A Q^2 A \left(G \partial_{YY}^L Q^2 + 2\psi_z \partial_{YZ}^L Q^2\right) dV = \mathcal{E}_1 + \mathcal{E}_2. 
\end{align*} 
Both $\mathcal{E}_1$ and $\mathcal{E}_2$ are treated similarly, hence only  consider $\mathcal{E}_1$. 
By \eqref{ineq:0freqIBP} and \eqref{mlowbound}, we have 
\begin{align*}
\mathcal{E}_1 & \lesssim \nu^{2/3} \norm{G}_{L^\infty H^{N}}\norm{\grad_L A Q^2}_{L^2 L^2}^2 + \nu^{2/3} \norm{\grad G}_{L^2 H^{N}} \norm{A Q^2}_{L^\infty L^2} \norm{\grad_L A Q^2}_{L^2L^2} \\ 
& \lesssim \epsilon^3\nu^{2/3-1-1/2-1/2} = \epsilon^3 \nu^{-4/3},   
\end{align*}
which suffices.

\subsection{$H^{N-1}$ estimate on $Q^2$}
Recall that a crucial strategy of the current approach is to confirm that the extra $m^{1/2}$ on $Q^2$ can be removed in $H^{N-1}$.
As in \S\ref{sec:HNQ2}, an energy estimate gives 
\begin{align*}
& \frac{1}{2} \| M Q^2(T) \|_{H^{N-1}}^2 +\nu  \| \nabla_L M Q^2 \|_{L^2  H^{N-1}}^2 + \| \sqrt{-\dot{M} M} Q^2 \|_{L^2  H^{N-1}}^2 \\
& \leq  \frac{1}{2} \| M Q^2(1) \|_{H^{N-1}}^2 + \int_1^T \int \langle D \rangle^{N-1} M Q^2 \langle D \rangle^{N-1} M\left[-  (\widetilde{U}_0 \cdot \nabla + U_{\neq} \cdot \nabla_t) Q^2  - Q^j \partial^t_j U^2 \right. \\
& \left. \qquad \qquad \qquad \qquad \qquad \qquad \qquad  - \partial^t_{i} U^j \partial^t_{ij}U^2 + \partial_Y^t (\partial^t_i U^j \partial^t_j U^i) + \nu (\widetilde{\Delta}_t - \Delta_L) Q^2 \right]\,dV\,dt \\
& = \frac{1}{2} \| M Q^2(1) \|_{H^{N-1}}^2 + \mathcal{T} + NLS1 + NLS2 + NLP + DE. 
\end{align*}

Nearly ever step in this estimate is similar to those done in \S\ref{sec:HNQ2}, indeed, the presence of $m^{1/2}$ in \S\ref{sec:HNQ2} is used only to control the $NLP(1,3,0,\neq)$ term in \S\ref{sec:NLPQ2}.  
The $\mathcal{T}$ is bounded as in \S\ref{toucan} and is hence omitted (however, notice that this requires the $H^N$ estimate \eqref{boundQ2} here; this detail is due to our only assuming ${N-1} >3/2$, where normally $H^{5/2+}$ is natural for closing energy estimates on a system such as \eqref{def:MainSys}). 
Similarly, the dissipation error terms $DE$ are controlled as in \S\ref{sec:DEQ2}.

The $NLP(0,0)$ terms are treated as in \S\ref{sec:NLPQ2}
Now let us see how  the reduction of one derivative allows to eliminate the use of $m^{1/2}$ in the treatment of $NLP(1,3,0,\neq)$. By \eqref{ineq:neqfreqIBP}, \eqref{ineq:tdecay}, and  $N-1 > 3/2$, 
\begin{align*}
NLP(1,3,0,\neq)  & \lesssim \norm{\grad_L M Q^2}_{L^2 H^{N-1}} \norm{\jap{t}^{-1} \grad U_0^1}_{L^\infty H^{N-1}} \norm{\jap{t} \partial_X U^3_{\neq}}_{L^2 H^{N-1}} & \lesssim\epsilon^3 \nu^{-1/2-1/2} = \epsilon^3\nu^{-1}. 
\end{align*}
This suffices for the $NLP(1,3,0,\neq)$ term; the $NLP(1,2,0,\neq)$ is similar. 
The other $NLP$ terms can be treated as in \S\ref{sec:NLPQ2} and the $NLS1$ and $NLS2$ terms can be treated as in \S\ref{penguin} above, and hence these contributions are also omitted.
This completes the $H^{N-1}$ estimate \eqref{boundQ2low}. 

\section{Energy estimate on $Q^3$} 

\label{EEQ3}

In this section, we prove that, under the assumptions of Proposition~\ref{propbootstrap} (in particular, the bootstrap assumptions~\eqref{boundsU}, \eqref{boundsQ}, \eqref{boundsC}), the inequality~\eqref{boundQ3} holds, with $8$ replaced by $4$ on the right-hand side.

An energy estimate gives (recall the shorthand \eqref{ABshort}) 
\begin{align}
& \frac{1}{2} \| B Q^3(T) \|_{L^2}^2 + \nu \|  \nabla_L B Q^3 \|_{L^2  L^2}^2 + \| m \sqrt{-\dot{M} M} Q^3 \|_{L^2  H^N}^2 + \| M \sqrt{-\dot{m} m} Q^3 \|_{L^2  H^N}^2 \nonumber \\
& = \frac{1}{2} \| B Q^3(1) \|_{L^2}^2 + \int_1^T \int B Q^3 B \left[- 2\partial^t_{XY} U^3 + 2\partial^t_{XZ} U^2     \right.\nonumber \\
& \qquad  \qquad  \qquad \qquad \qquad -  (\widetilde{U}_0 \cdot \nabla + U_{\neq} \cdot \nabla_t) Q^3 -  Q^j \partial^t_j U^3   - 2\partial^t_{i} U^j \partial^t_{ij}U^3   \nonumber \\
& \qquad \qquad \qquad \qquad  \qquad  \qquad \left. + \partial^t_Z \left(\partial^t_i U^j \partial^t_j U^i\right) + \nu (\widetilde{\Delta}_t - \Delta_L) Q^3 \right] \,dV\,dt \nonumber \\
& =  \frac{1}{2} \| B Q^3(1) \|_{L^2}^2 + LS + LP + \mathcal{T} +  NLS1 +  NLS2 + NLP + DE. \label{eq:Q3Energy}
\end{align}

\subsection{The linear stretching term $LS$}
The linear stretching term can be split into (recall the shorthands \eqref{def:G} and \eqref{ABshort}), 
\begin{align*}
LS & = \int_1^T \int B Q^3 B \left[- 2\partial^L_{XY} \Delta_L^{-1}[ Q^3 - G \partial_{YY}^L U^3 - 2\psi_z \partial_{YZ}^L U^3 - \Delta_t C \partial_Y^L U^3] \right] \,dV dt \\
& \quad + \int_1^T \int B Q^3 B \left(\psi_y \partial_{XY}^t U^3 \right) \, dV dt \\ 
& = LS_1 + LS_2 + LS_3 + LS_4 + LS_5.
\end{align*}
The leading order term, $LS_1$, is absorbed by the left-hand side of \eqref{eq:Q3Energy}: indeed, by construction of $m$ in \S\ref{sec:FM} we have 
$$
LS_1 \leq \| M \sqrt{-\dot{m} m} Q^3 \|_{L^2  H^N}^2 + \frac{\nu}{2} \| \nabla_L B Q^3 \|_{L^2  L^2}^2.
$$
We turn next to the error terms. 
For $LS_2$ we apply Lemma \ref{lem:mcomm} and use that $\abs{k} \lesssim \abs{k,\eta-kt,l}\sqrt{-\dot{M}^0 M^0}$ to deduce: 
\begin{align*}
LS_2 & = -2 \int_1^T \int B \partial^L_{XY}\Delta_L^{-1} Q^3 B\left( G \partial_{YY}^L U^3\right) \,dV\,dt \\
& \lesssim \norm{\sqrt{- \dot{M^0} M^0} m Q^3}_{L^2 H^N} \norm{G}_{L^\infty H^{N+1}} \norm{m^{1/2}\Delta_L U^3_{\neq}}_{L^2 H^N} \\ 
& \lesssim \epsilon^3 \nu^{-1-1/3-1/6} = \epsilon^3 \nu^{-3/2}, 
\end{align*} 
which suffices for $\epsilon \nu^{-3/2} \ll 1$; notice that this is a sharp use of the smallness conditions. 
The term $LS_3$ can be estimated in the same way as $LS_2$. 
The $LS_4$ term is estimated with a slight variation, using \ref{ineq:mDelTrick} and $\abs{k} \lesssim \abs{k,\eta-kt,l}\sqrt{-\dot{M}^0 M^0}$: 
\begin{align*}
LS_4 \lesssim \norm{\sqrt{-\dot{M}^0M^0} m Q^3_{\neq}}_{L^2 H^N} \norm{C}_{L^\infty H^{N+2}} \norm{\grad_L U^3_{\neq}}_{L^2 H^N} \lesssim \epsilon^3\nu^{-3/2}. 
\end{align*}
Turn to $LS_5$. Here we use Lemma \ref{lem:mcomm}, $\abs{k} \lesssim \abs{k,\eta-kt,l}\sqrt{-\dot{M}^0 M^0}$, and \eqref{mlowbound}, to deduce 
\begin{align*}
LS_5 & \lesssim \norm{B Q^3_{\neq}}_{L^2 L^2} \norm{\grad C}_{L^\infty H^{N+1}} \norm{\sqrt{-\dot{M}^0 M^0} m^{1/2} \Delta_L U^3_{\neq}}_{L^2 H^N} \lesssim \epsilon^3 \nu^{-1/6-1-1/3} = \epsilon^3 \nu^{-3/2}, 
\end{align*}
which suffices for $\epsilon \nu^{-3/2} \ll 1$. 

\subsection{The linear pressure term $LP$}
The linear pressure term is split into two contributions: 
\begin{align*} 
LP & = \int_1^T \int B Q^3 B \partial_{XZ}^LU^2_{\neq} dV dt + \int_1^T \int B Q^3 B \left(\psi_z \partial_{XY}^LU^2_{\neq}\right) dV dt \\ 
& = LP_1 + LP_2. 
\end{align*}
By definition of $M^1$, 
\begin{align*}
LP_1 & \lesssim \norm{\sqrt{-\dot{M^1}M^1}mQ^3}_{L^2 H^N}\norm{\sqrt{-\dot{M^1}M^1} m^{1/2} \Delta_L U^2_{\neq}}_{L^2 H^N} \lesssim  C_0^{-1} \left(C_0 \epsilon\right)^2,  
\end{align*}
which suffices by choosing $C_0$ sufficiently large. 
For $LP_1$, we have, similar to $LS_5$ above, by the definition of $M$ and Lemma \ref{lem:dotMPEL} along with Lemma \ref{lem:mcomm}, 
\begin{align*}
LP_2 & \lesssim \norm{\grad C}_{L^\infty H^{N+1}} \norm{M mQ^3_{\neq}}_{L^2 H^N}\norm{\sqrt{-\dot{M}M} m^{1/2} \Delta_L U^2_{\neq}}_{L^2 H^N} \lesssim \epsilon^3 \nu^{-7/6},  
\end{align*}
which is sufficient. 

\subsection{Transport nonlinearity} \label{puffin}
The interaction of non-zero frequencies will require more precision here than in \S\ref{toucan}. 
As in \S\ref{toucan}, we subdivide based on frequency: 
\begin{align*}
\mathcal{T} & = -\int_1^T \int  B Q^3 B\left( \tilde U_0 \cdot \grad Q^3_0  + \tilde U_0 \cdot \grad Q^3_{\neq} \right) dV dt - \int_1^T \int  B Q^3 B\left( U_{\neq} \cdot \grad_t Q^3_0 + U_{\neq} \cdot \nabla_t Q^3_{\neq} \right) dV dt \\ 
& = \mathcal{T}_{00} + \mathcal{T}_{0\neq} + \mathcal{T}_{\neq0} + \mathcal{T}_{\neq \neq}. 
\end{align*}
The $\mathcal{T}_{00}$ term is treated as in \S\ref{toucan} and is hence omitted for brevity. The $\mathcal{T}_{0\neq}$ term is treated analogously to the corresponding term in \S\ref{toucan} via a paraproduct decomposition, yielding: 
\begin{align*}
\mathcal{T}_{0\neq} & \lesssim \norm{B Q^3_{\neq}}_{L^2 L^2} \left(\norm{g}_{L^\infty H^N} + \norm{U_0^3}_{L^\infty H^N}\right)\norm{\grad Q^3_{\neq}}_{L^2 H^{3/2+}}  \\ & \quad + \norm{B Q^3_{\neq}}_{L^2 L^2}\left(\norm{\jap{t} g}_{L^\infty H^{3/2+}} +  \norm{U_0^3}_{L^\infty H^{3/2+}}\right)\norm{\grad_L M Q^3_{\neq}}_{L^2 H^{N}} \\
& \lesssim \nu^{-2/3}\norm{B Q^3_{\neq}}^2_{L^2 L^2} \left(\norm{g}_{L^\infty H^N} + \norm{U_0^3}_{L^\infty H^N}\right) \\ & \quad + \nu^{-2/3}\norm{B Q^3_{\neq}}_{L^2 L^2}\left(\norm{\jap{t} g}_{L^\infty H^{3/2+}} +  \norm{U_0^3}_{L^\infty H^{3/2+}}\right)\norm{\grad_L B Q^3_{\neq}}_{L^2 L^2} \\
& \lesssim \epsilon^3 \nu^{-1} + \epsilon^3 \nu^{-4/3}. 
\end{align*} 
Consider next $\mathcal{T}_{\neq 0}$, which is straightforward:  
\begin{align*}
\mathcal{T}_{\neq 0} & \lesssim \norm{B Q^3_{\neq}}_{L^\infty L^2} \norm{U^{2,3}_{\neq}}_{L^2 H^N} \norm{\grad Q^3_0}_{L^2 H^N} \lesssim \epsilon^3 \nu^{-2/3}.
\end{align*} 
Finally consider next $\mathcal{T}_{\neq \neq}$. First, divide up based on the presence of $U^1_{\neq}$:  
\begin{align*}
\mathcal{T}_{\neq \neq} & = \int B Q^3 B \left( U_{\neq}^j \partial_j^t Q^3_{\neq} \right) dV dt = \mathcal{T}_{\neq\neq}^1 + \mathcal{T}_{\neq\neq}^2 + \mathcal{T}_{\neq\neq}^3.
\end{align*}
The latter two terms can be treated in a straightforward manner using \eqref{mlowbound} and Propositions \ref{lem:BasicApriori} and \ref{prop:EDID}, 
\begin{align*}
\mathcal{T}_{\neq\neq}^{2,3} & \lesssim \nu^{-2/3}\norm{B Q^3}_{L^\infty L^2} \norm{U^{2,3}_{\neq}}_{L^2 H^N} \norm{\grad_L B Q^3}_{L^2 L^2} \lesssim \epsilon^3 \nu^{-4/3}. 
\end{align*}
Next, decompose $\mathcal{T}^1_{\neq \neq}$ via a paraproduct 
\begin{align*}
\mathcal{T}^1_{\neq \neq} & = \int  B Q^3 B \left( (U_{\neq}^1)_{Hi} (\partial_X Q^3_{\neq})_{Lo} \right) dV dt + \int  B Q^3 B \left( (U_{\neq}^1)_{Lo} (\partial_X Q^3_{\neq})_{Hi} \right) dV dt  \\ 
& = \mathcal{T}^1_{\neq \neq;HL} + \mathcal{T}^1_{\neq \neq;LH}. 
\end{align*}
For the $HL$ term we have the following by \eqref{mlowbound} and Propositions \ref{lem:BasicApriori} and \ref{prop:EDID}, 
\begin{align*}
\mathcal{T}_{\neq\neq;HL}^1 & \lesssim \norm{B Q^3}_{L^\infty L^2} \norm{U^1_{\neq}}_{L^2 H^N} \norm{Q^3_{\neq}}_{L^2 H^{5/2+}} \\
&  \lesssim \nu^{-2/3}\norm{B Q^3}_{L^\infty L^2} \norm{U^1_{\neq}}_{L^2 H^N} \norm{B Q^3_{\neq}}_{L^2 L^2} \lesssim \epsilon^3 \nu^{-4/3}. 
\end{align*}
For the $LH$ term we use the better estimate on $\norm{U^1_{\neq}}_{H^{N-1}}$ from Proposition \ref{prop:EDID} (and \eqref{mlowbound}): 
\begin{align*}
\mathcal{T}_{\neq\neq;LH}^1 & \lesssim \nu^{-2/3}\norm{B Q^3}_{L^\infty L^2}\norm{U^1_{\neq}}_{L^2H^{3/2+}}\norm{\partial_X B Q^3}_{L^2 L^2} \lesssim \epsilon^3 \nu^{-4/3}. 
\end{align*}
This completes the transport nonlinearity. 

\subsection{Nonlinear pressure and stretching terms}
\subsubsection{The stretching terms} 
Recall the shorthands defined in \S\ref{sec:short}. 
The $NLSi(0,0)$ terms are treated as in \S\ref{penguin} and are hence here omitted. 

For $NLS1$ we get from Proposition \ref{prop:EDID} (note $j \neq 1$), 
\begin{align*}
NLS1(\neq,0,j) 
& \lesssim \norm{B Q_{\neq}^3}_{L^2 L^2} \norm{Q^j_{\neq}}_{L^2 H^N} \norm{\grad U^3_0}_{L^\infty H^N}  \lesssim \epsilon^3 \nu^{-1/6 - 2/3 - 1/6} = \epsilon^3 \nu^{-1}. 
\end{align*} 
Similarly, by Propositions \ref{prop:EDID} and \ref{lem:BasicApriori},   
\begin{align*}
NLS1(0,\neq,j) 
& \lesssim \norm{B Q_{\neq}^3}_{L^2 L^2} \norm{Q^j_{0}}_{L^\infty H^N} \norm{\partial_j^t U^3_{\neq}}_{L^2 H^N} \\ 
& \lesssim \epsilon^3 \left( \nu^{-1/6-1-1/6}\mathbf{1}_{j = 1}  + \nu^{-1/6-1/2}\mathbf{1}_{j \neq 1} \right) \\ 
&  \lesssim \epsilon^3 \nu^{-4/3}.
\end{align*}
For the interaction of non-zero frequencies we use the slight variant, 
\begin{align*}
NLS1(\neq,\neq,j) & \lesssim \norm{B Q^3}_{L^\infty L^2} \norm{Q^j_{\neq}}_{L^2 H^N} \norm{\partial_j^t U^3_{\neq}}_{L^2 H^N} \\ 
  & \lesssim \epsilon^3 \left( \nu^{-2/3-1/6-1/2}\mathbf{1}_{j \neq 1}  + \nu^{-1-1/6-1/6}\mathbf{1}_{j = 1} \right) \\ 
& \lesssim \epsilon^{3} \nu^{-4/3}, 
\end{align*}
which is sufficient for $\epsilon \nu^{-4/3} \ll 1$. 

Turn next to $NLS2$; first by Proposition \ref{prop:EDID} (since $j \neq 1$),
\begin{align*}
NLS2(\neq,0,j) & \lesssim \norm{B Q^3_{\neq}}_{L^2 L^2} \norm{\grad_L U^j_{\neq}}_{L^2 H^N} \norm{\grad U^3_0}_{L^\infty H^{N+1}} \lesssim \epsilon^3 \nu^{-1/6-1/2} = \epsilon^3 \nu^{-2/3}, 
\end{align*}
which suffices. 
Consider next $NLS2(0,\neq,1)$, which requires a slightly more precise treatment. 
Via a paraproduct decomposition, \eqref{ineq:mdjtf}, \eqref{ineq:mDelTrick}, and $t\abs{k} \lesssim \jap{\eta} \jap{\eta-kt}$, we have 
\begin{align*}
NLS2(0,\neq,1) & \lesssim \norm{B Q^3_{\neq}}_{L^2 L^2} \norm{\grad U^1_{0}}_{L^\infty H^{5/2+}} \norm{m^{1/2} \partial_i^t \partial_X U^3_{\neq}}_{L^2 H^N} \\ 
& \quad + \norm{B Q^3_{\neq}}_{L^2 L^2} \norm{\jap{t}^{-1} \partial_i^t U^1_{0}}_{L^\infty H^{N+1}} \norm{m^{1/2} \jap{t} \partial_i^t \partial_X U^3_{\neq}}_{L^2 H^{3/2+}} \\ 
& \lesssim \norm{B Q^3_{\neq}}_{L^2 L^2} \norm{U^1_{0}}_{L^\infty H^{N+2}} \norm{m \Delta_L U^3_{\neq}}_{L^2 H^N} \\ 
& \quad + \norm{B Q^3_{\neq}}_{L^2 L^2} \norm{\jap{t}^{-1} U^1_{0}}_{L^\infty H^{N+2}} \norm{m \grad_L \Delta_L U^3_{\neq}}_{L^2 H^{5/2+}} \\ 
& \lesssim \epsilon^3 \nu^{-1/6-1-1/6} + \epsilon^3 \nu^{-1/6-1/2} = \epsilon^3 \nu^{-4/3}, 
\end{align*}
which is sufficient (note we also used Lemma \ref{lem:CoefCtrl}). 
For $j\neq 1$ contributions, we have 
\begin{align*}
NLS2(0,\neq,j \neq 1) & \lesssim \norm{B Q^3_{\neq}}_{L^2 L^2} \norm{\grad U^j_{0}}_{L^\infty H^{N+1}} \norm{m^{1/2} \Delta_L U^3_{\neq}}_{L^2 H^N} \lesssim \epsilon^3 \nu^{-2/3}.
\end{align*}
Turn finally to the interaction of non-zero frequencies; using \eqref{ineq:mDelTrick}, \eqref{mlowbound}, and Proposition \ref{lem:BasicApriori},  
\begin{align*}
NLS2(\neq,\neq,j) & \lesssim \norm{B Q^3}_{L^\infty L^2} \norm{\grad_L U^{2,3}_{\neq}}_{L^2 H^{N}} \norm{\Delta_L U^3_{\neq}}_{L^2 H^N} \\
& \qquad + \norm{B Q^3}_{L^\infty L^2} \norm{\grad_L U^1_{\neq}}_{L^2 H^{N}} \norm{\partial_X \grad_L U^3_{\neq}}_{L^2 H^N} \\ 
& \lesssim \nu^{-1}\norm{B Q^3}_{L^\infty L^2} \norm{m \Delta_L U^{2,3}_{\neq}}_{L^2 H^{N}} \norm{m\Delta_L U^3_{\neq}}_{L^2 H^N} \\
& \qquad + \nu^{-2/3}\norm{B Q^3}_{L^\infty L^2} \norm{m \Delta_L U^1_{\neq}}_{L^2 H^{N}} \norm{m \Delta_L U^3_{\neq}}_{L^2 H^N} \\
& \lesssim \epsilon^3 \nu^{-1-1/6-1/6} = \epsilon^3 \nu^{-4/3},   
\end{align*}
which completes the treatment of the stretching terms. 

\subsubsection{The pressure term $NLP$}
the treatment of the $NLP(i,j,0,0)$ term is the same as that in \S\ref{sec:NLPQ2} and is hence omitted here. For the leading order term involving $i = 1$, we begin by subdividing via a paraproduct decomposition: 
\begin{align*}
NLP(1,j,0,\neq) & =  \int_1^T \int  B  Q^3  B\partial_Z^t \left( (\partial_j^t U^1_0)_{Hi} (\partial_X U^j_{\neq} )_{Lo} \right) \,dV\,dt \\ & \quad + \int_1^T \int  B  Q^3  B\partial_Z^t \left( (\partial_j^t U^1_0)_{Lo} (\partial_X U^j_{\neq} )_{Hi} \right) \,dV\,dt \\
& = P_{HL} + P_{LH}. 
\end{align*} 
By \eqref{ineq:neqfreqIBPm} and \eqref{ineq:tdecay}, we have, since $j \neq 1$,
\begin{align*}
P_{HL} & \lesssim \norm{\grad_L B Q^3_{\neq}}_{L^2 L^2}\norm{\jap{t}^{-1} \partial_j^t U^1_0}_{L^\infty H^{N}}\norm{\jap{t}\partial_X U^j_{\neq}}_{L^2 H^{3/2+}} \lesssim \epsilon^{3}\nu^{-1}. 
\end{align*}
By \eqref{ineq:neqfreqIBPm}, \eqref{mlowbound}, Lemma \ref{lem:CoefCtrl}, and the inequality $1 \lesssim (\abs{k} + \abs{l} + \abs{\eta-kt}) \sqrt{-\dot{M^0} M^0}$ to deduce 
\begin{align*}
P_{LH} & \lesssim \norm{\grad_L B Q^3_{\neq}}_{L^2 L^2}\norm{\partial_j^t U^1_0}_{L^\infty H^{5/2+}}\norm{\sqrt{-\dot{M} M} m^{1/2} \partial_X \grad_L U^j_{\neq}}_{L^2 H^N} \\ 
& \lesssim \norm{\grad_L B Q^3_{\neq}}_{L^2 L^2}\norm{U^1_0}_{L^\infty H^{N+2}}\norm{\sqrt{-\dot{M} M} m \Delta_L U^j_{\neq}}_{L^2 H^N}  \\ 
& \lesssim \epsilon^3 \nu^{-3/2},  
\end{align*}
which completes the $NLP(1,j,0,\neq)$ terms for $\epsilon \nu^{-3/2} \ll 1$. 
For the $NLP(i\neq 1,j,0,\neq)$ terms a much simpler argument is possible, indeed, by \eqref{ineq:neqfreqIBP} and Proposition \ref{prop:EDID}, 
\begin{align*}
NLP(i\neq1,j,0,\neq) & \lesssim \norm{\grad_L B Q^3_{\neq}}_{L^2 L^2} \norm{\partial_j^t U^i_0}_{L^\infty H^{N}} \norm{\grad_L U^j_{\neq}}_{L^2 H^N} \lesssim \epsilon^{3}\nu^{-1}.  
\end{align*}
Turn next to the $NLP(1,j,\neq,\neq)$ terms. 
By the paraproduct decomposition we deduce, 
\begin{align*}
NLP(1,j,\neq,\neq) & \lesssim \norm{B Q^3}_{L^\infty L^2} \norm{\Delta_L U^1_{\neq}}_{L^2 H^{N}} \norm{\partial_X U^j_{\neq}}_{L^2 H^{3/2+}} \\ & \quad + \norm{B Q^3}_{L^\infty L^2} \norm{\grad_L U^1_{\neq}}_{L^2 H^{N}} \norm{\partial_X \grad_L U^j_{\neq}}_{L^2 H^{3/2+}}  \\ 
& \quad + \norm{B Q^3}_{L^\infty L^2} \norm{\grad_L U^1_{\neq}}_{L^2 H^{3/2+}} \norm{\partial_{X} \grad_L U^j_{\neq}}_{L^2 H^{N}} \\ & \quad + \norm{B Q^3}_{L^\infty L^2} \norm{\Delta_L U^1_{\neq}}_{L^2 H^{3/2+}} \norm{\partial_{X} U^j_{\neq}}_{L^2 H^{N}}.
\end{align*}
Then, by \eqref{ineq:mDelTrick}, \eqref{mlowbound}, and Lemma \ref{lem:BasicApriori}, we have  
\begin{align*}
\norm{\Delta_L U^1_{\neq}}_{L^2 H^N} & \lesssim \nu^{-2/3}\norm{m \Delta_L U^1_{\neq}}_{L^2 H^N} \lesssim \epsilon\nu^{-7/6} \\ 
\norm{\Delta_L U^1_{\neq}}_{L^2 H^{3/2+}} & \lesssim \nu^{-2/3}\norm{m \Delta_L U^1_{\neq}}_{L^2 H^{3/2+}} \lesssim \epsilon\nu^{-5/6} \\ 
\norm{\grad_L U^1_{\neq}}_{L^2 H^{3/2+}} & \lesssim \epsilon \nu^{-1/2} \\ 
\norm{\grad_L U^1_{\neq}}_{L^2 H^{N}} & \lesssim \epsilon \nu^{-2/3} \\
\norm{\partial_X U^j_{\neq}}_{L^2 H^{N}} & \lesssim \norm{m \Delta_L U^j_{\neq}}_{L^2 H^{N}} \lesssim \epsilon \nu^{-1/6}\mathbf{1}_{j\neq1} + \epsilon \nu^{-1/2}\mathbf{1}_{j=1} \\ 
\norm{\partial_X U^j_{\neq}}_{L^2 H^{3/2+}} & \lesssim \norm{m \Delta_L U^j_{\neq}}_{L^2 H^{3/2+}} \lesssim \epsilon \nu^{-1/6} \\ 
\norm{\partial_X \grad_L U^j_{\neq}}_{L^2 H^{N}} & \lesssim \norm{m^{1/2} \Delta_L U^j_{\neq}}_{L^2 H^N} \lesssim \nu^{-1/3}\norm{m \Delta_L U^j_{\neq}}_{L^2 H^N} \lesssim \epsilon \nu^{-1/2}\mathbf{1}_{j\neq 1} + \epsilon \nu^{-5/6}\mathbf{1}_{j = 1}  \\ 
\norm{\partial_X \grad_L U^j_{\neq}}_{L^2 H^{3/2+}} & \lesssim \norm{m^{1/2} \Delta_L U^j_{\neq}}_{L^2 H^{3/2+}} \lesssim \nu^{-1/3}\norm{m \Delta_L U^j_{\neq}}_{L^2 H^{3/2+}} \lesssim \epsilon \nu^{-1/2}. 
\end{align*}
Therefore, it follows that 
\begin{align*}
NLP(1,j,\neq,\neq) & \lesssim \epsilon^3 \nu^{-4/3}. 
\end{align*}
For the other $\neq,\neq$ terms we can use the simpler treatment via \eqref{ineq:0freqIBP}, \eqref{mlowbound}, and Proposition \ref{lem:BasicApriori},  
\begin{align*}
NLP(i,j,\neq,\neq)\mathbf{1}_{i,j\neq 1} & \lesssim \left(\norm{\grad_L B Q^3}_{L^2 L^2} + \norm{\grad C}_{L^2 H^{N+1}}\norm{BQ^3}_{L^\infty L^2} \right) \norm{\grad_L U^i_{\neq}}_{L^\infty H^{N}} \norm{\grad_L U^j_{\neq}}_{L^2 H^{N}} \mathbf{1}_{i,j\neq 1}\\
& \lesssim \nu^{-2/3} \left(\norm{\grad_L B Q^3}_{L^2 L^2} + \norm{\grad C}_{L^2 H^{N+1}}\norm{BQ^3}_{L^\infty L^2} \right) \\ & \quad\quad \times \norm{m\Delta_L U^i_{\neq}}_{L^\infty H^{N}} \norm{m \Delta_L U^j_{\neq}}_{L^2 H^{N}} \mathbf{1}_{i,j\neq 1}\\
& \lesssim \epsilon^3 \nu^{-4/3}. 
\end{align*}
This completes all of the nonlinear pressure terms. 

\subsection{Dissipation error terms}
Next turn to the dissipation error terms,
\begin{align*}
DE = \nu\int\int B Q^3 B \left( G \partial_{YY}^L Q^3 + 2\psi_z \partial_{ZY}^L Q^3 \right) dV dt = \mathcal{E}_1 +  \mathcal{E}_2. 
\end{align*}
We will need a slightly more refined treatment here than was used in \S\ref{sec:DEQ2}. 
As $\mathcal{E}_{1}$ is slightly harder we will just treat this term and omit the treatment of $\mathcal{E}_2$ for brevity.  
At the zero $X$ frequency we have, via integration by parts and the product rule, 
\begin{align*}
\mathcal{E}_{1;0} & = \nu\int\int B Q^3_0 B \left( G \partial_{YY} Q^3_0\right) dV dt \\ 
& \lesssim \nu\norm{Q^3_0}_{L^\infty H^N}\norm{\grad G}_{L^2 H^N} \norm{\grad Q^3_0}_{L^2 H^N} + \nu\norm{G}_{L^\infty H^N} \norm{\grad Q^3_0}^2_{L^2 H^N} \\ 
& \lesssim \epsilon^3 \nu^{-1}. 
\end{align*}
Turn next to the non-zero frequencies. Via integration by parts, Lemma \ref{lem:mcomm}, and \eqref{mlowbound}, 
\begin{align*}
\mathcal{E}_{1;\neq} & = \nu \int B Q^3_{\neq} B\left(G \partial_{YY}^L Q^3_{\neq} \right) dV dt \\ 
& \lesssim \nu\norm{\grad_L B Q^3_{\neq}}_{L^2 L^2} \norm{G}_{L^\infty H^{N+1}} \norm{m^{1/2} \grad_L Q^3_{\neq}}_{L^2 H^{N}} \\ 
 & \quad + \nu\norm{B Q^3_{\neq}}_{L^2 L^2} \norm{G}_{L^\infty H^{N+1}} \norm{\grad_L Q^3_{\neq}}_{L^2 H^{N}} \\
& \lesssim \epsilon^3 \nu^{-4/3}. 
\end{align*}
This completes the treatment of the dissipation error terms. 

\section{Energy estimates on $Q^1$}
The energy estimates on $Q^1$ are generally much simpler than estimates on $Q^3$ as the bounds \eqref{boundQ1d}, \eqref{boundQ10short}, and \eqref{boundQ10} are so much weaker than \eqref{boundQ3} (the lift-up effect growth is generally much larger than what the nonlinear terms could do in this regime).   
 
\subsection{Energy estimate on $Q^1_{\neq}$ in $H^N$} 
An energy estimate gives (recall the shorthand \eqref{ABshort}), 
\begin{align*}
& \frac{1}{2} \| B Q^1_{\neq}(T) \|_{L^2}^2 + \nu \|  \nabla_L B Q^1_{\neq} \|_{L^2  L^2}^2 + \| m \sqrt{-\dot{M} M} Q^1_{\neq} \|_{L^2  H^N}^2 + \| M \sqrt{-\dot{m} m} Q^1_{\neq} \|_{L^2  H^N}^2\\
& =  \frac{1}{2} \| B Q^1_{\neq}(1) \|_{L^2}^2 + \int_1^T \int B Q^1_{\neq} B \left[-Q^2_{\neq} - 2\partial^t_{XY} U^1_{\neq} + 2\partial_{XX} U^2_{\neq}  \right.\\
& \qquad \qquad  \qquad \qquad - ((\widetilde{U}_0 \cdot \nabla  + U_{\neq} \cdot \nabla_t) Q^1)_{\neq} - (Q^j \partial^t_{i}U^1)_{\neq}  - 2(\partial^t_i U^j \partial_i^t \partial^t_j U^1)_{\neq}\\
& \left. \qquad \qquad  \qquad \qquad \qquad \qquad  + \partial_X \left(\partial^t_i U^j \partial^t_j U^i\right) + \nu (\widetilde{\Delta}_t - \Delta_L) Q^1_{\neq} \right] \,dV\,dt\\
& =\frac{1}{2} \| m M Q^1_{\neq}(1) \|_{H^N}^2 + LU + LS + LP + \mathcal{T} +  NLS1 +  NLS2+ NLP + DE.
\end{align*}
Several terms above can be estimated exactly like the corresponding terms for $Q^3$, namely $LS$, $LP$, and $DE$. Therefore, we omit the estimates of these terms for brevity, and only treat the others. 

\subsubsection{The lift up term $LU$}
The lift-up effect term is treated via Proposition \ref{prop:EDID}, which implies 
\begin{align*}
LU & = -\int_1^T \int B Q^1_{\neq} B Q^2_{\neq} dV dt   \lesssim \norm{B Q^1_{\neq}}_{L^2 L^2} \norm{A Q^2_{\neq}}_{L^2 L^2}  \lesssim C_0 \epsilon^2 \nu^{-2/3},  
\end{align*}
which is consistent with the estimate as stated for $C_0$ chosen sufficiently large. 

\subsubsection{The stretching and pressure terms $NLS1$, $NLS2$, and $NLP$}
We will focus on $NLS1$ and $NLS2$; the $NLP$ terms can be treated analogously to the latter. 

Consider first the $NLS1(0,\neq,1)$ terms. Using a paraproduct decomposition as has been done several times above and applying Lemma \ref{lem:mcomm}, we get 
\begin{align*}
NLS1(0,\neq,1) & \lesssim \norm{B Q^1_{\neq}}_{L^2 L^2}\norm{Q^1_0}_{L^\infty H^{5/2+}} \norm{\partial_X m^{1/2} U^1_{\neq}}_{L^2 H^N} \\ & \quad + \norm{B Q^1_{\neq}}_{L^2 L^2}\norm{Q^1_0}_{L^\infty H^N} \norm{\partial_X U^1_{\neq}}_{L^2 H^{3/2+}} \\ 
& \lesssim \nu^{-2/3}\epsilon^2 \left( \epsilon \nu^{-4/3}\right). 
\end{align*}
For corresponding terms with $j\neq 1$ an easier treatment is available: 
\begin{align*}
NLS1(0,\neq,j\neq 1) & \lesssim \norm{BQ^1_{\neq}}_{L^2 L^2}\norm{Q^j_0}_{L^\infty H^N} \norm{\grad_L U^1_{\neq}}_{L^2 H^N} \lesssim \epsilon^{-1/3-1/6-1/2-1/3} = \left(\epsilon \nu^{-2/3}\right) \nu^{-2/3}\epsilon^{2}.  
\end{align*}
Similarly, (noting that $j \neq 1$ by the nonlinear structure): 
\begin{align*}
NLS1(j,\neq,0) & \lesssim \norm{BQ^1_{\neq}}_{L^2 L^2}\norm{m^{1/2} Q^j_{\neq}}_{L^2 H^N} \norm{U^1_{0}}_{L^\infty H^{N+2}} \lesssim \left(\epsilon \nu^{-4/3}\right) \epsilon^{2} \nu^{-2/3}. 
\end{align*}
Finally for the $NLS1(i,\neq,\neq)$ terms we may use another straightforward argument. By \eqref{ineq:mDelTrick},  
\begin{align*}
NLS1(j,\neq,\neq) & \lesssim \norm{B Q^1_{\neq}}_{L^2 L^2}\norm{Q^1_{\neq}}_{L^\infty H^N} \norm{\partial_X U^1_{\neq}}_{L^2 H^N} + \norm{B Q^1_{\neq}}_{L^2 L^2}\norm{Q^{2,3}_{\neq}}_{L^\infty H^N} \norm{\grad_L U^1_{\neq}}_{L^2 H^N} \\ 
& \lesssim \left(\epsilon \nu^{-4/3}\right)\nu^{-2/3}\epsilon^{2}. 
\end{align*}
This completes the $NLS1$ terms.

Turning to $NLS2$, we have first, since $j \neq 1$,
\begin{align*}
NLS2(i,j,\neq,0) \lesssim \| B Q^1 \|_{L^2 L^2} \| \nabla_L U^{2,3}_{\neq} \|_{L^2 H^N} \| \grad U_0^1 \|_{L^\infty H^{N+1}} \lesssim \left(\epsilon \nu^{-4/3}\right)\epsilon^2 \nu^{-2/3}. 
\end{align*}
Next, we rely on Lemma\ref{lem:mcomm} and \eqref{ineq:mdjtf} (for the $j=1$ case), \eqref{ineq:mDelTrick}, and \eqref{mlowbound},  
\begin{align*}
NLS2(i,j,0,\neq) & \lesssim \| BQ^1_{\neq} \|_{L^2 L^2} \| \grad_t U_0^1 \|_{L^\infty H^{N+1}} \| m^{1/2} \partial_X \grad_L U^1_{\neq} \|_{L^2 H^N}\mathbf{1}_{j=1} \\ 
 & \quad + \| BQ^1_{\neq} \|_{L^2 L^2} \| \grad_t U_0^{2,3} \|_{L^\infty H^{N}} \| \Delta_L U^1_{\neq} \|_{L^2 H^N}\mathbf{1}_{j\neq1} \\ 
& \lesssim \| BQ^1_{\neq} \|_{L^2 L^2} \| U_0^1 \|_{L^\infty H^{N+2}} \| m \Delta_L U^1_{\neq} \|_{L^2 H^N}\mathbf{1}_{j=1} \\ 
 & \quad + \nu^{-2/3}\| BQ^1_{\neq} \|_{L^2 L^2} \| U_0^{2,3} \|_{L^\infty H^{N+1}} \| m \Delta_L U^1_{\neq} \|_{L^2 H^N}\mathbf{1}_{j\neq1} \\ 
& \lesssim \left( \epsilon \nu^{-4/3} \right) \epsilon^2 \nu^{-2/3}
\end{align*}
Finally, \eqref{ineq:neqfreqIBP}, a paraproduct decomposition, \eqref{ineq:mDelTrick}, and Propositions \ref{lem:BasicApriori} and \ref{prop:EDID} imply, 
\begin{align*}
NLS2(i,j,\neq,\neq) & \lesssim \| B Q^1_{\neq} \|_{L^2 L^2}\left(\| \grad_L U^i_{\neq} \|_{L^\infty H^{3/2+}} \| \Delta_L U^1_{\neq} \|_{L^2 H^N} + \| \grad_L U^i_{\neq} \|_{L^\infty H^{N}} \| \Delta_L U^1_{\neq} \|_{L^2 H^{3/2+}}\right)\mathbf{1}_{i\neq 1} \\  
& \quad + \| B Q^1_{\neq} \|_{L^2 L^2}\left(\| \grad_L U^1_{\neq} \|_{L^\infty H^{3/2+}} \| \partial_X \grad_L U^1_{\neq} \|_{L^2 H^N} + \| \grad_L U^1_{\neq} \|_{L^\infty H^{N}} \| \partial_X \grad_L U^1_{\neq} \|_{L^2 H^{3/2+}}\right) \\ 
& \lesssim \mathbf{1}_{i \neq 1} \left(\epsilon \nu^{-4/3}\right) \epsilon^2 \nu^{-2/3} + \mathbf{1}_{i = 1} \left(\epsilon \nu^{-1}\right) \epsilon^2 \nu^{-2/3}, 
\end{align*}
which suffices for $\nu \epsilon^{-4/3} \ll 1$. 

\subsubsection{Transport nonlinearity} \label{albatross}
These terms can mostly be treated as in \S\ref{puffin}, however, one must check the contributions from $Q_0^1$. 
As in \S\ref{toucan} and \S\ref{puffin}, we subdivide based on frequency (note the slight difference since we are only focusing on non-zero frequencies here): 
\begin{align*}
\mathcal{T} & = \int_1^T \int  BQ^1_{\neq} B \left( \tilde U_0 \cdot \grad Q^1_{\neq} \right) dV dt + \int_1^T \int B Q^1_{\neq} B \left( U_{\neq} \cdot \grad_t Q^1_0 + U_{\neq} \cdot \nabla_t Q^1_{\neq} \right) dV dt \\ 
& =  \mathcal{T}_{0\neq} + \mathcal{T}_{\neq0} + \mathcal{T}_{\neq \neq}. 
\end{align*}
The terms $\mathcal{T}_{0\neq}$ and $\mathcal{T}_{\neq \neq}$ can be treated as in \S\ref{puffin} and are hence omitted for the sake of brevity.
Hence, turn to the remaining $\mathcal{T}_{\neq 0}$. Here we have (note the nonlinear structure that eliminates $U^1_{\neq}$), 
\begin{align*}
\mathcal{T}_{\neq 0} & \lesssim \norm{B Q^1_{\neq}}_{L^2 L^2} \norm{U^{2,3}_{\neq}}_{L^\infty H^N} \norm{\grad Q^1_0}_{L^2 H^N} \lesssim \epsilon^3 \nu^{-1/3-1/6-3/2} \lesssim \nu^{-2/3}\epsilon^2 \left(\epsilon \nu^{-4/3}\right). 
\end{align*}

\subsection{Long-time energy estimate on $Q_0^1$} \label{sec:LongQ01}
In this section we improve the estimate \eqref{boundQ10}. First, $Q_0^1$ solves the equation
\begin{align*} 
\partial_t Q_0^1 - \nu \widetilde{\Delta}_t Q_0^1+ Q^2_0 = - ((\widetilde{U}_0 \cdot \nabla + U_{\neq} \cdot \nabla_t)Q^1)_0 - (Q^j \partial_j^t U^1)_0 - 2(\partial_i^t U^j \partial_j^t \partial_i^t U^1)_0.
\end{align*} 
An energy estimate gives
\begin{align*}
& \frac{1}{2} \| Q^1_0(T) \|_{H^N}^2  + \nu  \|  \nabla_L  Q^1_0 \|_{L^2  H^N}^2 \\
& =  \frac{1}{2} \| (Q^1(1))_0 \|_{H^N}^2 + \int_1^T \int \langle D \rangle^N  Q^1_0 \langle D \rangle^N  \left[- Q^2_0 - ((\widetilde{U}_0 \cdot \nabla + U_{\neq} \cdot \nabla_t)Q^1)_0 -  (Q^j \partial_j^t U^1)_0 \right. \\
& \left. \qquad \qquad \qquad \qquad  \qquad \qquad \qquad \qquad \qquad - 2(\partial_i^t U^j \partial_j^t \partial_i^t U^1)_0 + \nu (\widetilde{\Delta}_t - \Delta_L) Q^1_0 \right]\,dV \,dt \\
& = \frac{1}{2} \| (Q^1(1))_0 \|_{H^N}^2 + LU + \mathcal{T} + NLS1 + NLS2 + DE.
\end{align*}

\subsubsection{The lift up term $LU$}
Using Lemma \ref{lem:CoefCtrl} and Proposition \ref{lem:BasicApriori}, 
\begin{align*}
LU & = -\int_1^T \int \jap{D}^N Q^1_0 \jap{D}^N \left[\Delta U_0^2 + G\partial_{YY} U_0^2 + 2\psi_{z} \partial_{YZ} U_0^1 + \Delta_t C \partial_Y U_0^2\right] dV dt \\  
& \lesssim \left(1 + \norm{\grad C}_{L^\infty H^N}\right)\norm{\grad Q_0^1}_{L^2 H^N}\norm{\grad U_0^2}_{L^2 H^N} + \norm{\grad C}_{L^2 H^{N+1}} \norm{Q_0^1}_{L^\infty H^N} \norm{\grad U_0^2}_{L^2 H^N} \\ 
& \lesssim C_0 \epsilon^2 \nu^{-2}, 
\end{align*} 
which is consistent with the bootstrap argument provided $C_0$ is chosen sufficiently large.  

\subsubsection{Transport nonlinearity} \label{Cormorant}
Similar to \S\ref{albatross}, we subdivide based on frequency:  
\begin{align*}
\mathcal{T} & = \int_1^T \int \jap{D}^N Q_0^1 \jap{D}^N \left( \tilde U_0 \cdot \grad Q_0^1 \right) dV dt + \int_1^T \int \jap{D}^N Q_0^1 \jap{D}^N \left( U_{\neq} \cdot \grad_t Q_{\neq}^1 \right) dV dt \\ & = \mathcal{T}_{0} + \mathcal{T}_{\neq}. 
\end{align*}
The zero frequencies $\mathcal{T}_0$ can be treated as in \S\ref{toucan} and are hence omitted for brevity. 
For the non-zero frequencies, first apply the divergence free condition: 
\begin{align*}
\mathcal{T}_{\neq} & = \int_1^T \int \jap{D}^N Q_0^1 \jap{D}^N \grad_t \cdot \left( U_{\neq} Q_{\neq}^1 \right)_0 dV dt. 
\end{align*}
Due to the $X$ average, the contribution from $U^1_{\neq}$ is crucially eliminated as well as the term involving $-t \partial_X$ in $\partial_Y^L$.  
Hence, by \eqref{ineq:0freqIBP}, \eqref{mlowbound} and Propositions \ref{lem:BasicApriori} and \ref{prop:EDID}:  
\begin{align*}
\mathcal{T}_{\neq}& \lesssim \norm{\grad Q_0^1}_{L^2 H^N} \norm{U^{2,3}_{\neq}}_{L^2H^N} \norm{Q^1_{\neq}}_{L^\infty H^N} + \norm{Q_0^1}_{L^\infty H^N} \norm{C}_{L^\infty H^{N+2}}\norm{U^{2,3}_{\neq}}_{L^2H^N} \norm{Q^1_{\neq}}_{L^2 H^N} \\ 
& \lesssim \nu^{-2}\epsilon^2\left(\epsilon \nu^{-2/3} + \epsilon^2\nu^{-4/3} \right), 
\end{align*}
which suffices. 

\subsubsection{Nonlinear stretching terms} 
Consider first $NLS1(0,0)$, which can be treated similar to $NLS(0,0)$ and $NLP(0,0)$ terms above: by Proposition \ref{lem:BasicApriori}, 
\begin{align*}
NLS1(0,0) & \lesssim \norm{Q_0^1}_{L^\infty H^N} \norm{\grad U^j_0}_{L^2 H^{N+1}} \norm{\partial_j^t U_0^1}_{L^2 H^N} \lesssim   \frac{\epsilon}{\nu}\left(\frac{\epsilon^2}{\nu^2}\right), 
\end{align*}
which is sufficient. $NLS2(0,0)$ is treated similarly and is hence omitted. 

Turn next to $NLS1(\neq,\neq)$:
\begin{align*}
& NLS1(\neq,\neq,j \neq 1)  \lesssim \norm{Q_0^1}_{L^\infty H^N} \norm{Q^j_{\neq}}_{L^2 H^N} \norm{\partial_j^t U_{\neq}^1}_{L^2 H^N} \lesssim \frac{\epsilon}{\nu^{2/3}}\left(\frac{\epsilon^2}{\nu^2}\right)\\
& NLS1(\neq,\neq,1)  \lesssim \norm{Q_0^1}_{L^\infty H^N} \norm{Q^1_{\neq}}_{L^2 H^N} \norm{\partial_X U_{\neq}^1}_{L^2 H^N}
 \lesssim \frac{\epsilon}{\nu^{2/3}} \left(\frac{\epsilon^2}{\nu^2}\right), 
\end{align*}
which is sufficient. 
The $NLS2(\neq,\neq)$ term is treated analogously and is hence omitted for brevity.

\subsubsection{The dissipation error terms $DE$} These are controlled as in \S\ref{sec:DEQ2} and are hence omitted for brevity. 

\subsection{Short-time energy estimate on $Q^1_0$ in $H^N$} \label{sec:ShortQ01}
Here we deduce \eqref{boundQ10short}, which refer to as a ``short-time'' estimate since it provides a superior estimate on $\norm{Q_0^1(t)}_{H^N}$ for $t \ll \nu^{-1}$ versus the ``long-time'' estimate $\norm{Q_0^1(t)}_{H^N} \lesssim \epsilon \nu^{-1}$.  

For this estimate (and the similar \eqref{boundU10short}), we use a slightly 
different method than we have applied for most estimates in the paper. 
Consider the differential equality: 
\begin{align}
\frac{1}{2}\frac{d}{dt} \left(\jap{t}^{-2} \norm{Q^1_0(t)}_{H^N}^2\right) & = -\frac{t}{\jap{t}^4} \norm{Q^1_0(t)}_{H^N}^2 - \jap{t}^{-2}\int \jap{D}^N Q_0^1 \jap{D}^N Q_0^2 dV \nonumber \\ & \quad - \nu \jap{t}^{-2} \norm{\grad Q_0^1}_{H^N}^2 + \mathcal{NL}, \label{eq:Q10diff}
\end{align}
where, using the shorthand from \S\ref{sec:short} analogous to that used in \S\ref{sec:LongQ01}, 
\begin{align*}
\mathcal{NL} = \mathcal{T} + NLS1 + NLS2 + DE
\end{align*}
denotes the contributions from all of the nonlinear terms. 
For the lift-up effect term, by \eqref{boundQ2},  
\begin{align*}
- \jap{t}^{-2}\int \jap{D}^N Q_0^1 \jap{D}^N Q_0^2 dV \leq \jap{t}^{-2}\norm{Q_0^1}_{H^N}\norm{Q_0^2}_{H^N} \leq \jap{t}^{-2} 8 \epsilon \| Q^1_0 \|_{H^N},
\end{align*}
and hence \eqref{eq:Q10diff} becomes
\begin{align}
\frac{1}{2}\frac{d}{dt} \left(\jap{t}^{-2} \norm{Q^1_0(t)}_{H^N}^2\right) + \nu\jap{t}^{-2} \norm{\grad Q_0^1}_{H^N}^2 \leq
\frac{1}{\jap{t}^2}\left(8\epsilon - \frac{t}{\jap{t}^2}\norm{Q_0^1(t)}_{H^N}\right)\norm{Q^1_0(t)}_{H^N} + \mathcal{NL}. \label{ineq:Q10diff2}
\end{align}
It follows from this differential inequality that if $\mathcal{NL} \leq \frac{1}{2}\nu \jap{t}^{-2} \| \nabla Q_0^1 \|_{H^N}^2 + f$, with $\| f \|_{L^1} \leq C_0\epsilon^2$, then \eqref{boundQ10short} holds for $C_0$ sufficiently large. 

\subsubsection{Transport nonlinearity} 
As in \S\ref{Cormorant}, we  divide the transport nonlinearity into two pieces 
\begin{align*}
\mathcal{T} & = \jap{t}^{-2} \int \jap{D}^N Q_0^1 \jap{D}^N\left(\tilde U_0 \cdot \grad Q_0^1 + \left(U_{\neq} \cdot \grad_t Q_{\neq}^1 \right)_0 \right) dV  = \mathcal{T}_0 + \mathcal{T}_{\neq}. 
\end{align*}
The first term is treated analogously to the treatment in \S\ref{toucan}: 
\begin{align*} 
\mathcal{T}_0 & \lesssim \norm{U_0^3}_{H^N} \norm{\jap{t}^{-1} \grad Q_0^1}_{H^N}^2 + \left(\norm{g}_{L^2} + \| \nabla g \|_{H^N} + \norm{\grad U_0^3}_{H^N} \right) \norm{\jap{t}^{-1} \grad Q_0^1}_{H^N}\norm{\jap{t}^{-1} Q_0^1}_{H^N} \\ 
& \lesssim \epsilon  \norm{\jap{t}^{-1} \grad Q_0^1}_{L^2 H^N}^2 + \epsilon \left(\norm{g}^2_{L^2} + \norm{\grad U_0^3}^2_{H^N} + \| \nabla g \|_{H^N}^2 \right);   
\end{align*}
the first term is absorbed by the dissipation in \eqref{ineq:Q10diff2} and the latter term integrates to $O(\epsilon^3 \nu^{-1})$. 

For $\mathcal{T}_{\neq}$, we first use the divergence free condition as in \S\ref{Cormorant}: 
\begin{align*}
\mathcal{T}_{\neq} = \jap{t}^{-2} \int \jap{D}^N Q_0^1 \jap{D}^N \grad_t \cdot \left(U_{\neq}  Q_{\neq}^1 \right)_0 dV, 
\end{align*}
which eliminates the contribution from $U^1$ and $-t\partial_X$. 
By \eqref{ineq:0freqIBP}, \eqref{mlowbound}, and Proposition \ref{lem:BasicApriori}, and for any constant $K$,
\begin{align*}
\mathcal{T}_{\neq} & \lesssim \jap{t}^{-2} \norm{\grad Q_0^1}_{H^N} \norm{U^{2,3}_{\neq}}_{H^N}\norm{Q_{\neq}^1}_{H^N} 
+  \jap{t}^{-2} \norm{Q_0^1}_{H^N} \norm{C}_{H^{N+2}} \norm{U^{2,3}_{\neq}}_{H^N}\norm{Q_{\neq}^1}_{H^N} \\ 
& \lesssim \frac{\nu}{K} \jap{t}^{-2} \norm{\grad Q_0^1}_{H^N}^2 +  \frac{K}{\nu} \jap{t}^{-2} \norm{U^{2,3}_{\neq}}_{H^N}^2 \norm{Q_{\neq}^1}_{H^N}^2  \\ 
& \qquad \qquad + \jap{t}^{-2} \norm{Q_0^1}_{H^N} \norm{C}_{H^{N+2}} \norm{U^{2,3}_{\neq}}_{H^N}\norm{Q_{\neq}^1}_{H^N}.
\end{align*}
The first term is absorbed by the dissipation in \eqref{ineq:Q10diff2} for $K$ sufficiently large; the remaining terms integrate to 
$\epsilon^4 \nu^{-3}$ and $\epsilon^4 \nu^{-2-1/6}$ using the $L^\infty$ controls on $U$, $Q$, and $C$, which is sufficient (note that $\epsilon^4 \nu^{-3} \ll \epsilon^2$ is borderline as it requires $\epsilon \nu^{-3/2} \ll 1$).  

\subsubsection{Nonlinear stretching}
Turn first to the interaction of zero frequencies. Consider $NLS1$ (noting that $j \neq 1$): 
\begin{align*}
NLS1(0,0) & = \jap{t}^{-2} \int \jap{D}^N Q_0^1 \jap{D}^N \left( \Delta_t U^j_0 \partial_j^t U_0^1\right) dV \lesssim \jap{t}^{-2}\norm{Q_0^1}_{H^N} \norm{\grad U^j_0}_{H^{N+1}} \norm{\grad U_0^1}_{H^N}.  
\end{align*}
Hence, by Proposition \ref{lem:BasicApriori} and Cauchy-Schwarz in time, 
\begin{align*}
\int_1^T NLS1(0,0) dt & \lesssim \norm{\jap{t}^{-1}Q_0^1}_{L^\infty H^N} \norm{U^j_0}_{L^\infty H^{N+2}} \norm{\grad U_0^1}_{L^2 H^N}  \lesssim \epsilon^3 \nu^{-3/2}, 
\end{align*}
which is sufficient for $\epsilon \nu^{-3/2}$ sufficiently small.  
The $NLS2(0,0)$ terms can be treated similarly and are hence omitted for the sake of brevity. 

Consider next $NLS1(\neq,\neq)$. 
Using that $\grad_t \cdot Q = 0$ due to the divergence free condition and \eqref{ineq:0freqIBP}, we have for any $K$, 
\begin{align*}
NLS1(\neq,\neq) & = \jap{t}^{-2} \int \jap{D}^N Q_0^1 \jap{D}^N \left( Q^j_{\neq} \partial_j^t U_{\neq}^1\right)_{0} dV \\
& = \jap{t}^{-2} \int \jap{D}^N Q_0^1 \jap{D}^N \partial_j^t \left( Q^j_{\neq} U_{\neq}^1\right)_{0} dV \\
& \lesssim \jap{t}^{-2} \left(\norm{\grad Q_0^1}_{H^N} + \norm{\grad C}_{H^{N+1}}\norm{Q_0^1}_{H^N}\right) \norm{Q^{2,3}_{\neq}}_{H^N}\norm{U_{\neq}^1}_{H^N} \\ 
& \lesssim \frac{\nu}{K}\jap{t}^{-2} \norm{\grad Q_0^1}_{H^N}^2 + \frac{K}{\nu}\jap{t}^{-2}\norm{Q^{2,3}_{\neq}}_{H^N}^2\norm{U_{\neq}^1}_{H^N}^2 \\ & \quad + \jap{t}^{-2} \norm{\grad C}_{H^{N+1}}\norm{Q_0^1}_{H^N} \norm{Q^{2,3}_{\neq}}_{H^N}\norm{U_{\neq}^1}_{H^N}.   
\end{align*}
For $K$ large, the first term is absorbed by the dissipation. By the $L^\infty$ controls from Proposition \ref{lem:BasicApriori}, the second factor integrates to $\epsilon^2\left(\epsilon^2 \nu^{-3}\right)$ and the third factor integrates to  $\epsilon^2\left( \epsilon^4 \nu^{-4}\right)$, both of which are sufficient. 
The $NLS2(\neq,\neq)$ term is treated similarly and is hence omitted. 

\subsubsection{Dissipation error estimates}
Write 
\begin{align*}
DE & = \jap{t}^{-2} \nu\int \jap{D}^N Q^1_0 \jap{D}^N \left(G \partial_{YY} Q^1_0 + 2\psi_z \partial_{YZ} Q^1_0\right) dV = \mathcal{E}_1 + \mathcal{E}_2. 
\end{align*}
We only bound $\mathcal{E}_1$; $\mathcal{E}_2$ is bounded in the same manner. Via integration by parts and the Sobolev product rule,  
\begin{align*}
\mathcal{E}_1 & \lesssim \jap{t}^{-2}\nu\norm{G}_{H^{N}}\norm{\grad Q_0^1}_{H^N}^2 + \jap{t}^{-2}\nu\norm{\grad G}_{H^{N}}\norm{Q_0^1}_{H^N}\norm{\grad Q_0^1}_{H^N} \\ 
& \lesssim \jap{t}^{-2} \epsilon\norm{\grad Q_0^1}_{H^N}^2 + \epsilon \nu^2 \norm{\grad C}_{H^{N+1}}^2. 
\end{align*}
The first term is absorbed by the leading order dissipation in \eqref{ineq:Q10diff2} and the other term integrates to $O(\epsilon^3 \nu^{-1})$, which suffices. 
This completes the short-time energy estimate on $Q_0^1$. 

\section{Energy estimate on $U^1_{\neq}$}
In this section we deduce the control \eqref{boundU1neq}; this is relatively easy due to the lower regularity, but it is slightly different to work in velocity form than the previous $Q^i$ estimates. 
The entire point of this estimate is that by working directly on the velocity, it is easier to take advantage of the inviscid damping from \eqref{ineq:gradU2} in the lift-up effect term, which is the reason for the large growth of $Q^1_{\neq}$. 

From the momentum equations, the non-zero frequencies of $U^1$ solve 
$$
\partial_t U^1_{\neq} - \nu \widetilde{\Delta}_t U^1_{\neq}  = -U^2_{\neq} + 2 \partial_{XX} \Delta_t^{-1} U^2_{\neq} - \left( \left[\widetilde{U}_0 \cdot \nabla + U_{\neq} \cdot \nabla_t \right] U^1\right)_{\neq} + \partial_X \Delta_t^{-1} (\partial_i^t U^j \partial_j^t U^i )_{\neq}.
$$
An energy estimate gives
\begin{align*}
& \frac{1}{2} \| M U^1_{\neq}(T) \|_{H^{N-1}}^2 + \nu \|M \nabla_L U^1_{\neq}\|_{L^2  H^{N-1}}^2 + \| \sqrt{-\dot{M} M} U^1_{\neq} \|_{L^2  H^{N-1}}^2 \\
& = \frac{1}{2} \| M U^1_{\neq}(1) \|_{H^{N-1}}^2 + \int_1^T \int \langle D \rangle^{N-1} M  U_{\neq}^1 \langle D \rangle^{N-1} M \left[ -U^2_{\neq} + 2 \partial_X^2 \Delta_t^{-1} U^2_{\neq} - (\left[\widetilde{U}_0 \cdot \nabla + U_{\neq} \cdot \nabla_t \right] U^1)_{\neq} \right. \\
& \qquad \qquad  \qquad \qquad \qquad  \qquad  \qquad  \qquad \quad \left. + (\partial_X \Delta_t^{-1} \partial^t_i U^j \partial_j^t U^i )_{\neq} + \nu(\widetilde{\Delta}_t - \Delta_L) U^1_{\neq} \right] \,dV\,dt \\
& = \frac{1}{2} \| M U^1_{\neq}(1) \|_{H^{N-1}}^2 + LU + LP + \mathcal{T} + NLP + DE.
\end{align*}

\subsection{Lift-up effect}
Start with the lift-up effect term, which can be bounded through the inviscid damping estimate we have on $U^2_{\neq}$ in $H^{N-1}$ in \eqref{ineq:gradU2}. 
In particular, since $1 \lesssim \sqrt{- \dot{M^0} M^0} |\nabla_L |$ (see \S\ref{sec:FM}), 
\begin{align*}
LU & \lesssim \norm{\sqrt{\dot{M} M} U^1_{\neq}}_{L^2 H^{N-1}} \norm{\grad_L U^2_{\neq}}_{L^2 H^{N-1}} \lesssim C_0\epsilon^2, 
\end{align*}  
which is sufficient for $C_0$ chosen sufficiently big. 
We remark that the simplicity and effectiveness of this estimate is the reason we are working with $U^1_{\neq}$.  

\subsection{Linear pressure}
We now turn to the linear pressure term, $LP$, which we bound by relying first on the inequality $1 \lesssim \sqrt{- \dot{M^0} M^0} |\nabla_L |$, and then on Lemmas \ref{lem:EDPEL} and Proposition \ref{prop:EDID}, 
\begin{align*}
LP & \lesssim \| \sqrt{\dot M M} U^1_{\neq} \|_{L^2 H^{N-1}} \| \nabla_L  \Delta_L \Delta_t^{-1} U^2_{\neq} \|_{L^2 H^{N-1}} \\
& \lesssim \| \sqrt{\dot M M} U^1_{\neq} \|_{L^2 H^{N-1}} \left[ \| \nabla_L U^2_{\neq} \|_{L^2 H^{N-1}} + \| \nabla C \|_{L^\infty H^{N+1}} \|  U^2_{\neq} \|_{L^2 H^{N-1}} \right] \\
& \lesssim C_0 \epsilon^2.
\end{align*}
which is sufficient for $C_0$ sufficiently large. Note that the inviscid damping of $U^2$ is also very important here.  

\subsection{Transport nonlinearity}
Turning to the transport term, we subdivide analogous to what has been applied in e.g. \S\ref{toucan}: 
\begin{align*}
\mathcal{T} & = \int_1^T \int \langle D \rangle^{N-1} M  U_{\neq}^1 \langle D \rangle^{N-1} M \left[ \widetilde{U}_0 \cdot \nabla U^1_{\neq} + U_{\neq} \cdot \nabla_t  U^1_{\neq} + U_{\neq} \cdot \grad_t U^1_0 \right] \, dV \,dt \\
& = \mathcal{T}_{0\neq} + \mathcal{T}_{\neq\neq} + \mathcal{T}_{\neq 0}.
\end{align*}
The term $\mathcal{T}_{0\neq}$ is treated like the analogous term in \S\ref{toucan} and is hence omitted. 
The term $\mathcal{T}_{\neq 0}$ is treated via the following, using Proposition \ref{lem:BasicApriori} and \eqref{ineq:tdecay} (also ${N-1} > 3/2+$): 
\begin{align*}
\mathcal{T}_{\neq 0} & \lesssim \norm{M U^1_{\neq}}_{L^2 H^{N-1}} \norm{\jap{t} U_{\neq}^{2,3}}_{L^2 H^{N-1}} \norm{\jap{t}^{-1} \grad U_0^1}_{L^\infty H^{N-1}} \lesssim \epsilon^3 \nu^{-1/6-1/2} = \epsilon^3 \nu^{-2/3},
\end{align*}
For $\mathcal{T}_{\neq \neq}$ we may also apply a straightforward argument: 
\begin{align*}
\mathcal{T}_{\neq \neq} & \lesssim \norm{M U^1_{\neq}}_{L^\infty H^{N-1}} \norm{U_{\neq}^{1,2,3}}_{L^2 H^{N-1}} \norm{\grad_L U_{\neq}^1}_{L^2 H^{N-1}} \lesssim \epsilon^3\nu^{-1/6-1/2} = \epsilon^3 \nu^{-2/3}. 
\end{align*}
This completes the transport terms. 

\subsection{Nonlinear pressure}
The nonlinear pressure term can be split into one piece for which both velocity fields have nonzero $X$ frequency, and its complement:
\begin{align*}
NLP & = \int_1^T \int \langle D \rangle^{N-1} M  U_{\neq}^1 \langle D \rangle^{N-1} M  \partial_X \Delta_t^{-1} (\partial_i^t U^j_{\neq} \partial_j^t U^i_{\neq} + 2 \partial_i^t U^j_{0} \partial_j^t U^i_{\neq}) \,dV \\
& = NLP_{\neq} + NLP_0.
\end{align*}
Treating $NLP_{\neq}$ is straightforward: using the divergence free condition, and Lemma \ref{lem:UNLP} we have
\begin{align*}
NLP_{\neq} & =  \int_1^T \int \langle D \rangle^{N-1} M  U_{\neq}^1 \langle D \rangle^{N-1} M  \Delta_t^{-1}\partial_X\partial_i^t ( U^j_{\neq} \partial_j^t U^i_{\neq}) \,dV \\
& \lesssim \| U_{\neq}^1 \|_{L^2  H^{N-1}} \| U_{\neq} \|_{L^\infty  H^{N-1}} \| \grad_L U_{\neq} \|_{L^2  H^{N-1}} 
\lesssim \epsilon^3 \nu^{-1/6-1/2} = \epsilon^3 \nu^{-2/3}. 
\end{align*}
The $NLP_0$ terms are bounded similarly, except for the ones involving $U^1_0$ -- to these we now turn: using the divergence free condition, 
\begin{align*}
 \int_1^T \int \langle D \rangle^{N-1} M  U_{\neq}^1 \langle D \rangle^{N-1} M \partial_X \Delta_t^{-1} ( 2 \partial_i^t U^1_{0} \partial_X U^i_{\neq}) \,dV & \\ & \hspace{-5cm}  =  \int_0^T \int \langle D \rangle^{N-1} M  U_{\neq}^1 \langle D \rangle^{N-1} M  \Delta_t^{-1} \partial_X \partial_i^t ( 2 U^1_{0} \partial_X U^i_{\neq}) \,dV \\
& \hspace{-5cm} \lesssim \norm{M U_{\neq}^1}_{L^2 H^{N-1}} \norm{U_0^1}_{L^\infty H^{N-1}} \left( \| \partial_X U^2 \|_{L^2  H^{N-1}} + \| \partial_X U^3 \|_{L^2  H^{N-1}} \right) \\ 
& \hspace{-5cm} \lesssim \epsilon^3 \nu^{-1/6-1-1/6} = \epsilon^{3}\nu^{-4/3}. 
\end{align*}
This completes the nonlinear pressure terms.

\subsection{Dissipation error}
Finally, the dissipation error is easily dealt with via the same method we have used several times previously: integrating by parts in the second equality,
\begin{align*}
DE & = \nu  \int_1^T \int \langle D \rangle^{N-1} M  U_{\neq}^1 \langle D \rangle^{N-1} M (G \partial_{YY}^L + 2 \psi_z \partial_{YZ}^L) U^1_{\neq})\,dV \,dt \\
& = - \nu  \int_1^T \int \langle D \rangle^{N-1} M \partial_Y^L U_{\neq}^1 \langle D \rangle^{N-1} M  (G \partial_Y^L + 2 \psi_z \partial_{Z}) U^1_{\neq}) \,dV\,dt\\
& \qquad \qquad \qquad \qquad \qquad  - \nu  \int_1^T \int \langle D \rangle^{N-1} M U_{\neq}^1 \langle D \rangle^{N-1} M (\partial_Y G  \partial_Y^L + 2 \partial_Y \psi_y \partial_Z)U^1_{\neq})\,dV\,dt \\
& \lesssim \nu \left[ \| C \|_{L^\infty  H^{N}} \|\nabla_L U^1_{\neq} \|_{L^2  H^{N-1}}^2 + \|\nabla C\|_{L^2  H^{N}} \|U^1_{\neq}\|_{L^\infty  H^{N-1}} \|\nabla_L U^1_{\neq} \|_{L^2  H^{N-1}} \right] \\ 
& \lesssim \epsilon^3\nu^{-1}. 
\end{align*}
This completes the estimate on $U^1_{\neq}$.

\section{Estimates on $C$ and $g$} 

\label{EEC}

\subsection{Energy estimate on $C$} \label{sec:LongC}
In this section, we prove that, under the assumptions of Proposition~\ref{propbootstrap} (in particular, the bootstrap assumptions~\eqref{boundsU}, \eqref{boundsQ}, \eqref{boundsC}), the inequality~\eqref{boundC} holds, with $8$ replaced by $4$ on the right-hand side.
Recall \eqref{eq:Cg}. 
An energy estimate gives 
\begin{align*}
& \frac{1}{2} \|   C(T) \|_{H^{N+2}}^2 + \nu \|  \nabla_L  C  \|_{L^2  H^{N+2}}^2 \\
& = \frac{1}{2} \|   C(1) \|_{H^{N+2}}^2 + \int_1^T \int \langle D \rangle^{N+2}   C \langle D \rangle^{N+2}   \left[- \tilde{U}_0 \cdot \grad C + g - U^2_0 + \nu (\widetilde{\Delta}_t - \Delta_L) C \right]\, dV \, dt \\
& = \frac{1}{2} \|   C(1) \|_{H^{N+2}}^2 + \mathcal{T} + L1 + L2 + DE.
\end{align*}
The transport nonlinearity $\mathcal{T}$ can be treated in the same manner as in the $Q^i_0$ energy estimates above and are hence omitted for brevity. 

\subsubsection{The linear term $L1$} \label{sec:LongL1}
Distinguish first between high and low frequencies:
\begin{align*}
L1 & = \int_1^T \int \langle  D \rangle^{N+2}   C \langle D \rangle^{N+2}   \left[ P_{\leq 1} g + P_{>1} g \right] \, dV\,dt  = L1_L + L1_H.
\end{align*}
Low frequencies are estimated by taking advantage of the decay of $g$:
\begin{align*}
L1_L \lesssim \| C \|_{L^\infty  L^2} \| g \|_{L^1 L^2} \lesssim \frac{C_1 \epsilon}{\nu} (C_0 \epsilon) = \frac{C_0 \nu}{C_1} \left( \frac{C_1 \epsilon}{\nu} \right)^2,
\end{align*}
while high frequencies are estimated with the help of the viscous dissipation:
\begin{align*}
L1_H \lesssim \| \nabla C \|_{L^2  H^N} \| \nabla g \|_{L^2  H^N} \lesssim \frac{C_1 \epsilon}{\nu^{3/2}} \frac{C_0 \epsilon}{\sqrt \nu} = \frac{C_0}{C_1} \left( \frac{C_1 \epsilon}{\nu} \right)^2 .
\end{align*}
Both are consistent with the Proposition \ref{propbootstrap} for $C_1 \gg C_0$. 

\subsubsection{The linear term $L2$}
The approach is analogous to the above term. 
Separating first high and low frequencies:
\begin{align*}
L2 & = \int_1^T \int \langle  D \rangle^{N+2}   C \langle D \rangle^{N+2}   \left[ P_{\leq 1} U^2_0 + P_{> 1} U^2_0 \right] \, dV\,dt  = L2_L + L2_H,
\end{align*}
we estimate low frequencies with the help of~\eqref{boundU02decay}
\begin{align*}
L2_L \lesssim \| C \|_{L^\infty  L^2} \| U^2_0 \|_{L^1 L^2} \lesssim \frac{C_1 \epsilon}{\nu} \frac{C_0 \epsilon}{\nu} = \frac{C_0}{C_1} \left( \frac{C_1 \epsilon}{\nu} \right)^2,
\end{align*}
and high frequencies using viscous dissipation:
\begin{align*}
L2_H \lesssim \| \nabla_L C \|_{L^2  H^N} \| \nabla_L U_0^2 \|_{L^2  H^N} \lesssim \frac{C_1 \epsilon}{\nu^{3/2}} \frac{C_0 \epsilon}{\sqrt \nu} = \frac{C_0}{C_1} \left( \frac{C_1 \epsilon}{\nu} \right)^2 .
\end{align*}
This completes the treatment of the linear terms. 

\subsubsection{Dissipation error terms}
Write 
\begin{align*}
DE = \nu\int_1^T \int \jap{D}^{N+2}C \jap{D}^{N+2}\left(G \partial_{YY}C + 2\psi_z \partial_{YZ} C\right) dV dt = \mathcal{E}_1 + \mathcal{E}_2. 
\end{align*}
The two error terms are treated exactly the same, so consider only $\mathcal{E}_1$. Using a paraproduct decomposition: 
\begin{align*}
\mathcal{E}_1 & =\nu\int_1^T \int \jap{D}^{N+2}C \jap{D}^{N+2}\left(G_{Hi} \partial_{YY}C_{Lo} + G_{Lo} \partial_{YY}C_{Hi} \right) dV dt \\ 
& \lesssim \nu\norm{C}_{L^\infty H^{N+2}}\norm{G}_{L^2 H^{N+2}} \norm{\grad C}_{L^2 H^{5/2+}} \\ & \quad +  \nu\left(\norm{\grad C}_{L^2 H^{N+2}}\norm{G}_{L^\infty H^{3/2+}} + \norm{C}_{L^\infty H^{N+2}}\norm{\grad G}_{L^2 H^{3/2+}}\right)\norm{\grad C}_{L^2 H^{N+2}}\\ 
& \lesssim \epsilon^3\nu^{-3}, 
\end{align*}
which is sufficient for $\epsilon \nu^{-1} \ll 1$.

\subsection{Estimates on $g$}
\label{EEg}

In this section, we prove that, under the assumptions of Proposition~\ref{propbootstrap} (in particular, the bootstrap assumptions~\eqref{boundsU}, \eqref{boundsQ}, \eqref{boundsC}), the inequalities~\eqref{boundg} and~\eqref{boundg1} hold, with $8$ replaced by $4$ on the right-hand side.

\subsubsection{Decay estimate on $g$ in $H^{N-1}$} \label{sec:decayg}
In this section we improve \eqref{boundg}. Recall \eqref{eq:Cg}.  
Therefore, an energy estimate gives
\begin{align*}
& \frac{1}{2} \| T^2 g(T) \|_{H^{N-1}}^2 + \nu \| t^2 \nabla_L g \|_{L^2  H^{N-1}} \\
& =  \frac{1}{2} \| g(1) \|_{H^{N-1}}^2 + \int_1^T \int t^4 \langle D \rangle^{N-1} g  \langle D \rangle^{N-1} \left[ -\tilde U_0 \cdot \grad g - \frac{1}{t} (U_{\neq} \cdot \nabla_t U^1_{\neq})_0 \right. \\
& \qquad \qquad \qquad \qquad\qquad \qquad\qquad \qquad \qquad \qquad \left.+ \nu ( \widetilde{\Delta}_t - \Delta_L) g \right] \,dV\,dt \\
& = \frac{1}{2} \| g(1) \|_{H^{N-1}}^2 + \mathcal{T}_0 + \mathcal{T}_{\neq} + DE; 
\end{align*}
notice the cancellation between the derivative of the time weight and the damping term. 
The estimates of $\mathcal{T}_0$ and $DE$ are obtained similar to the treatment in \S\ref{sec:LongC} and are hence omitted for brevity.
However, a new element appears in the estimate of $\mathcal{T}_{\neq}$. First, notice that 
\begin{align*}
\mathcal{T}_{\neq} & =- \int_1^T \int \langle D \rangle^{N-1} t^2 g  \langle D \rangle^{N-1} t \Big[ \underbrace{(U^1_{\neq} \partial_X U^1_{\neq})_0}_0 + (U^2_{\neq} \partial_Y^t U_{\neq}^1)_0 + (U^3_{\neq} \partial_Z^t U_{\neq}^1)_0 \Big]\,dV\,dt. 
\end{align*}
Therefore, by \eqref{ineq:tdecay} and Proposition \ref{lem:BasicApriori}, it follows that 
\begin{align*}
\mathcal{T}_{\neq} & \lesssim \| t^2 g \|_{L^\infty  H^{N-1}} \left( \| \jap{t} U^2_{\neq} \|_{L^2  H^{N-1}} + \| \jap{t} U^3_{\neq} \|_{L^2  H^{N-1}} \right) \| \nabla_L U^1_{\neq} \|_{L^2  H^{N-1}} \\
& \lesssim \epsilon^3 \nu^{-1}. 
\end{align*}
This completes the improvement of the estimate \eqref{boundg}.  

\subsubsection{Energy estimate on $g$ in $H^{N+2}$} 
From \eqref{eq:Cg}, an energy estimate on $g$ gives
\begin{align*}
& \frac{1}{2} \| g(T) \|_{H^{N+2}}^2 + \nu \|  \nabla_L  g  \|_{L^2  H^{N+2}}^2 \\
& = \frac{1}{2} \|   g(1) \|_{H^{N+2}}^2 + \int_1^T \int \langle D \rangle^{N+2} g \langle D \rangle^{N+2}   \left[-\tilde U_0 \cdot \grad g - \frac{2g}{t} \right.\\
& \qquad \qquad \qquad \qquad  \qquad \qquad \qquad \left.- \frac{1}{t}(U_{\neq} \cdot \nabla_t U^1_{\neq})_0 + \nu (\widetilde{\Delta}_t - \Delta_L) g \right]\, dV \, dt \\
&  = \frac{1}{2} \|   g(1) \|_{H^{N+2}}^2 + \mathcal{T}_0 + L + \mathcal{T}_{\neq} + DE.
\end{align*}
Observe that $L$ does not need to be estimated, since it has a favorable sign. All other terms appearing in the right-hand side can be estimated following the same pattern as in many other instances in this paper (hence these are omitted for the sake of brevity), except for $\mathcal{T}_{\neq}$, to which we now turn. Observe that
\begin{align*} 
\mathcal{T}_{\neq} \leq \| g \|_{L^\infty H^{N+2}} \left\| \frac{1}{t} (U_{\neq} \cdot \nabla_t U^1_{\neq})_0 \right\|_{L^1 H^{N+2}} \lesssim C_0 \epsilon \left\| \frac{1}{t} (U_{\neq} \cdot \nabla_t U^1_{\neq})_0 \right\|_{L^1 H^{N+2}}.
\end{align*} 
This last factor can, in turn, be estimated by
\begin{align*} 
\left\| \frac{1}{t} (U_{\neq} \cdot \nabla_t U^1_{\neq})_0 \right\|_{L^1 H^{N+2}} \lesssim \left\| \frac{1}{t} (U_{\neq} \cdot \nabla_t U^1_{\neq})_0 \right\|_{L^1 L^2} +  \left\| \Delta \frac{1}{t} (U_{\neq} \cdot \nabla_t U^1_{\neq})_0 \right\|_{L^1 H^N}.
\end{align*} 
The first term on the right-hand side is easily estimated (using that ${N-1} > 3/2$ for Sobolev embedding): 
\begin{align*}
\left\| \frac{1}{t} (U_{\neq} \cdot \nabla_t U^1_{\neq})_0 \right\|_{L^1 L^2} & \lesssim \norm{U_{\neq}}_{L^\infty H^{N-1}} \norm{\grad_L U^1_{\neq}}_{L^2 H^{N-1}} \lesssim \epsilon^2 \nu^{-1/2}. 
\end{align*}
 For the second term, we use that, for any function $f$, $\Delta f_0 = (\Delta_L f)_0$ as well as the identity $(U^1_{\neq} \partial_X  U^1_{\neq})_0=0$ (which was used in \S\ref{sec:decayg} above as well) to obtain that
\begin{align*}
& \left\| \Delta \frac{1}{t} (U_{\neq} \cdot \nabla_t U^1_{\neq})_0 \right\|_{L^1 H^N} = \left\|  \frac{1}{t} (\Delta_L [U_{\neq}^{2,3} \cdot \nabla_t U^1_{\neq}])_0 \right\|_{L^1 H^N} \\
& \lesssim \left\| \frac{1}{t} ( (\Delta_L U_{\neq}^{2,3}) \cdot \nabla_t U^1_{\neq})_0 \right\|_{L^1 H^N} + \left\| \frac{1}{t} ( U_{\neq}^{2,3} \cdot \Delta_L \nabla_t U^1_{\neq})_0 \right\|_{L^1 H^N} + \left\| \frac{1}{t} ( (\nabla_L U_{\neq}^{2,3}) \cdot \nabla_L \nabla_t U^1_{\neq})_0 \right\|_{L^1 H^N} \\
& \lesssim \| \Delta_L U_{\neq}^{2,3} \|_{L^\infty H^N} \|  \nabla_L U^1_{\neq} \|_{L^2 H^N} + \| U_{\neq}^{2,3} \|_{L^\infty H^N} \| \nabla_L \Delta_L U^1_{\neq} \|_{L^2 H^N} + \| \nabla_L U_{\neq}^{2,3} \|_{L^2 H^N} \| \Delta_L U^1_{\neq} \|_{L^\infty H^N} \\
& \lesssim \epsilon^2 \nu^{-1-1/2} = \epsilon^2 \nu^{-3/2},
\end{align*}
where in the last line we used \eqref{mlowbound} and Lemma \ref{lem:BasicApriori}. 
This completes the improvement of \eqref{boundg1} for $\epsilon \nu^{-3/2} \ll 1$ (note the sharp use of the hypotheses).   

\section{Zero frequency velocity estimates} 

The purpose of these estimates are to deduce low frequency controls on the velocity. 
First, observe that by the discussion in \S\ref{sec:coordstuff}, it suffices to prove these estimates on $u_0^i$, rather than $U_0^i$. 
Indeed, due to Lemma \ref{lem:SobComp} and the estimate $\norm{C}_{L^\infty H^{N+2}} \lesssim \epsilon \nu^{-1}$, for $\epsilon \nu^{-1} \ll 1$, we may move from one coordinate system to another, in particular
\begin{subequations} \label{ineq:Ui0}  
\begin{align}
\| U^i_0 \|_{H^s} & \approx \|u^i_0 \|_{H^s} \\  
\norm{U^i_{\neq}}_{H^s} & \approx \norm{\bar{u}^i_{\neq}}_{H^s};  
\end{align}
\end{subequations} 
recall the definition of $\bar{u}^i$ from \S\ref{sec:coordstuff}.

\subsection{Decay of $U_0^2$}
In this section, we improve the estimate~\eqref{boundU02decay}.
First, due to the divergence-free condition, $\widehat u^2_0 (k=0,\eta,l=0) = 0$, thus $Q u^2_0 = u_0^2$, where $Q$ projects on the Fourier modes for which $k$ or $l \neq 0$. 
Therefore, $u^2_0$ solves
\begin{align*}
\partial_t u_0^2 - \nu \Delta u_0^2 & = - Q (u \cdot \nabla u^2)_0 + Q \partial_y \Delta^{-1} (\partial_i u^j \partial_j u^i)_0 \\
& = - Q (u_0 \cdot \nabla u^2_0) + Q(\partial_y \Delta^{-1} (\partial_i u^j_0 \partial_j u^i_0)) - Q  (u_{\neq} \cdot \nabla u^2_{\neq})_0  + Q \partial_y \Delta^{-1} (\partial_i u^j_{\neq} \partial_j u^i_{\neq})_0 \\
& = Q T_0 + Q P_0 + Q T_{\neq} + QP_{\neq}.
\end{align*}
with data $(u^2_{in})_0$. 
Duhamel's formulation then reads
\begin{align*} 
u^2_0 = e^{\nu t \Delta} (u^2_{in})_0 + \int_0^t e^{\nu (t-s)\Delta} (Q T_0(s) + QP_0(s) + QT_{\neq}(s) + Q P_{\neq}(s))\,ds = I + II + III + IV + V.
\end{align*} 
Due to the spectral gap made possible via $Q$, there holds 
\begin{align} 
\| e^{\nu t \Delta} Q f \|_{L^2} \lesssim e^{-\nu t} \| f \|_{L^2} \quad \mbox{and} \quad \| e^{\nu t \Delta} \nabla Q f \|_{L^2} \lesssim \frac{1}{\sqrt {\nu t}} e^{-\nu t} \| f \|_{L^2}, \label{ineq:Qdeltdec}
\end{align} 
so that
\begin{align*}
& \left\| \int_0^t e^{\nu (t-s) \Delta} Q F(s) \,ds \right\|_{L^1 L^2} \lesssim \frac{1}{\nu} \| F \|_{L^1 L^2} \\
& \left\| \int_0^t e^{\nu (t-s) \Delta} \nabla Q F(s) \,ds \right\|_{L^1 L^2} \lesssim \frac{1}{\nu} \| F \|_{L^1 L^2}.
\end{align*}
Therefore, one obtains immediately
\begin{align*}
& \| I \|_{L^1  L^2} \lesssim \frac{1}{\nu} \| u_{in} \|_{L^2} \lesssim \frac{\epsilon}{\nu}.
\end{align*} 
Next, by the divergence-free condition on $u$ and Sobolev embedding,
\begin{align*}
\| II \|_{L^1  L^2} & = \left\| \int_0^t e^{\nu (t-s) \Delta} Q [ \partial_y (u^2_0)^2 + \partial_z (u^2_0 u^3_0) ]\,ds \right\|_{L^1  L^2}  \lesssim \frac{1}{\nu} \left[ \| (u_0^2)^2 \|_{L^1  L^2} + \| u_0^2 u_0^3 \|_{L^1  L^2} \right] \\
& \lesssim  \frac{1}{\nu} \| u_0^2 \|_{L^1  L^2} \left[ \| u_0^2 \|_{L^\infty H^{N-1}} + \| u_0^3 \|_{L^\infty  H^{N-1}} \right] \lesssim \frac{\epsilon}{\nu}  \| u_0^2 \|_{L^1  L^2}, 
\end{align*}
which is sufficient for $\epsilon \nu^{-1} \ll 1$. 
Similarly, we claim that the same bound holds for $III$:
$$
\| III \|_{L^1  L^2} \lesssim \frac{\epsilon}{\nu}  \| u_0^2 \|_{L^1  L^2}.
$$
Indeed, let us look at $QP_0$, which, since $u$ is divergence-free, can be written $Q \partial_{y i j} \Delta^{-1} (u_0^j u^i_0)$. If $i$ or $j$ is equal to 2, then the same proof as for $II$ applies. If both $i$ and $j$ are equal to 3, use the divergence free condition on $u$, namely $\partial_z u^3_0 = - \partial_y u^2_0$ to reduce matters to the previous case. 

Next turn to estimates $IV$ and $V$. Due to the zero mode projection and the divergence free constraint, first note 
\begin{align*}
(u_{\neq} \cdot \nabla u^2_{\neq})_0  = \left(\grad \cdot (u_{\neq} u^2_{\neq}) \right)_0 = \left(\partial_y(\bar{u}_{\neq}^2 \bar{u}^2_{\neq}) \right)_0 + \left(\partial_z(\bar{u}_{\neq}^3 \bar{u}^2_{\neq}) \right)_0. 
\end{align*}
Therefore, by \eqref{ineq:Qdeltdec} we have 
\begin{align*}
\norm{IV}_{L^1 L^2} & \lesssim \nu^{-1}\left(\norm{\bar{u}^2_{\neq}}^2_{L^1L^4} + \norm{\bar{u}^3\bar{u}^2_{\neq}}_{L^1L^2}\right) \\ 
& \lesssim \nu^{-1}\norm{\bar{u}^2_{\neq}}_{L^1L^2}\left(\norm{\bar{u}^2_{\neq}}_{L^\infty H^{N-1}} + \norm{\bar{u}^3_{\neq}}_{L^\infty H^{N-1}}\right) \\ 
& \lesssim \nu^{-1}\epsilon^2, 
\end{align*}
where the last line followed from \eqref{ineq:Ui0} and \eqref{ineq:tdecay} -- note the use of the inviscid damping on $\bar{u}^2_{\neq}$. 
We may apply a similar treatment for $V$, indeed, by the zero mode projection and the divergence free constraint, 
\begin{align*}
\partial_y \Delta^{-1} (\partial_i u^j_{\neq} \partial_j u^i_{\neq})_0 & = \partial_y\Delta^{-1} (\partial_i \partial_j \left(u^j_{\neq} u^i_{\neq}\right))_0 = \partial_y\Delta^{-1}\left(\partial_{yy}\left(\bar{u}^2_{\neq} \bar{u}^2_{\neq}\right) + 2\partial_{yz}\left(\bar{u}^2_{\neq} \bar{u}^3_{\neq}\right) + \partial_{zz}\left(\bar{u}^3_{\neq} \bar{u}^3_{\neq}\right)\right)_0. 
\end{align*}
By \eqref{ineq:Qdeltdec} we have 
\begin{align*}
\norm{V}_{L^1 L^2} & \lesssim \nu^{-1}\norm{\bar{u}^{i}_{\neq} \bar{u}^j_{\neq}}_{L^1 H^1} \mathbf{1}_{i\neq 1}\mathbf{1}_{j\neq 1} \lesssim  \nu^{-1}\norm{\bar{u}^{i}_{\neq}}_{L^2 H^{N-1}} \norm{\bar{u}^j_{\neq}}_{L^2 H^{N-1}} \mathbf{1}_{i\neq 1}\mathbf{1}_{j\neq 1} \lesssim \nu^{-4/3} \epsilon^2,  
\end{align*}
which is sufficient for $\epsilon \nu^{-4/3} \ll 1$. 

Gathering all the above estimates, we obtain that, for a constant $K$,
$$
\| u^2_0 \|_{L^1  L^2} \leq K  \frac{\epsilon^2}{\nu^{4/3}} + K \frac{\epsilon}{\nu}  \| u_0^2 \|_{L^1  L^2},
$$
which, by \eqref{ineq:Ui0}, improves \eqref{boundU02decay} for $\epsilon \nu^{-3/2} = \delta$ sufficiently small.  

\subsection{Uniform bound on $U_0^1$} \label{sec:U01low}
As discussed above, it suffices to consider the velocity in the original coordinates, $u_0^1$, which solves
$$
\partial_t u^1_0 - \nu \Delta u^1_0  = -u^2_0 - (u\cdot \nabla u^1)_0.
$$
An energy estimate gives
\begin{align*}
\frac{1}{2} \|  u^1_0(t)\|_{H^{N-1}}^2 + \nu \|\nabla u^1_0\|_{L^2  H^{N-1}}^2 & = \frac{1}{2} \|  (u^1_{in})_0\|_{H^{N-1}}^2 - \int_0^T \int \langle D \rangle^{N-1} u^1_0 \jap{D}^{N-1} \left[ u^2_0 - (u\cdot \nabla_t u^1)_0 \right] \,dV\,dt \\
& = \frac{1}{2} \|  (u^1_{in})_0\|_{H^{N-1}}^2 + LU + \mathcal{T}.
\end{align*}
To estimate the lift up term, use that $u^2_0$ always has a nonzero $z$ frequency by incompressibility together with the algebra property of $H^{N-1}$ to obtain
\begin{align*}
LU & \leq \| \nabla u_0^1 \|_{L^2  H^{N-1}} \| \nabla u_0^2 \|_{L^2  H^{N-1}} \lesssim \frac{C_0 \epsilon}{\nu^{3/2}} \frac{\epsilon}{\sqrt{\nu}} =  \frac{1}{C_0} \left( \frac{C_0 \epsilon}{\nu} \right)^2, 
\end{align*}
which suffices for $C_0$ sufficiently large. 
Split the transport term into the contribution of zero and non-zero frequencies (in $X$):
\begin{align*}
\mathcal{T} & =  \int_0^T \int  \langle D \rangle^{N-1} u^1_0 \langle D \rangle^{N-1} \left[ u_0^2 \partial_y u^1_0 + u_0^3 \partial_z u^1_0 + (u_{\neq} \cdot \nabla u^1_{\neq})_0 \right]\,dV\,ds \\
& = \mathcal{T}_0 + \mathcal{T}_{\neq}.
\end{align*}
To estimate $\mathcal{T}_0$, consider first the term involving (roughly speaking) $u_0^1 u_0^2 \partial_y u^1_0 $; to bound it, we will use again that the $z$ frequency of $u_0^2$ cannot be zero. Consider next the term involving $u^1_0 u_0^3 \partial_z u^1_0$; to bound it, we will use that at least two of the factors $u^1_0$, $u^3_0$, and $\partial_z u^1_0$ must have nonzero $z$ frequency. This leads to the estimate
\begin{align*}
\mathcal{T}_0 &\lesssim \| u^1_0 \|_{L^\infty  H^{N-1}} \| \nabla u_0^2 \|_{L^2  H^{N-1}} \| \nabla u^1_0 \|_{L^2  H^{N-1}} + \| u^1_0 \|_{L^\infty  H^{N-1}} \| \nabla  u_0^3 \|_{L^2 H^{N-1}} \| \nabla u^1_0 \|_{L^2  H^{N-1}} \\
& \qquad \qquad \qquad +  \| \nabla  u^1_0 \|_{L^\infty  H^{N-1}} \| u_0^3 \|_{L^2  H^{N-1}} \| \nabla u^1_0 \|_{L^2  H^{N-1}} \\
& \lesssim \frac{C_0 \epsilon}{\nu} \frac{C_0 \epsilon}{\sqrt{\nu}} \frac{C_0 \epsilon}{{\nu^{3/2}}} \leq \frac{C_0 \epsilon}{\nu} \left( \frac{C_0 \epsilon}{\nu} \right)^2,
\end{align*}
which suffices for $\epsilon \nu^{-1}$ sufficiently small. 

To estimate $\mathcal{T}_{\neq}$ we use the projection to zero frequency to note 
\begin{align*}
(u_{\neq} \cdot \nabla u^1_{\neq})_0 = (\bar{u}_{\neq}^2 \cdot (\partial_y - t\partial_y u_0^1\partial_x)\bar{u}^1_{\neq})_0 + (\bar{u}_{\neq}^3 \cdot \partial_z\bar{u}^1_{\neq})_0,
\end{align*}
(note that the $\bar{u}^1 \partial_X \bar{u}^1$ is eliminated), which implies (using also $N-1 > 3/2$), 
\begin{align*} 
\mathcal{T}_{\neq} \lesssim \| u_0^1 \|_{L^\infty  H^{N-1}} \| \bar{u}^{2,3}_{\neq} \|_{L^2  H^{N-1}} \|  \nabla_L \bar{u}^1_{\neq} \|_{L^2  H^{N-1}} \lesssim \frac{C_0 \epsilon}{\nu}  \frac{C_0 \epsilon}{\nu^{1/6}}  \frac{C_0 \epsilon}{\sqrt \nu} \leq C_0 \epsilon \nu^{1/3} \left( \frac{C_0 \epsilon}{\nu} \right)^2, 
\end{align*} 
which suffices for $\epsilon$ sufficiently small. 
This completes the energy estimate on $u_0^1$. 

\subsection{Short time estimate on $U_0^1$}
We also need to deduce \eqref{boundU10short}. 
For this, we combine the techniques of \S\ref{sec:ShortQ01} combined with the methods applied in \S\ref{sec:U01low}.
We omit the treatment for brevity as the details follow analogously (note the main change from \S\ref{sec:U01low} is the way the lift-up effect is treated). 

\subsection{Uniform bound on $U_0^2$}
In this section we improve the bound \eqref{boundU2}. 
As discussed above, we may perform estimates on $u_0^2$ rather than $U^2_0$.
In the original coordinates, $u_0^2$ solves the equation 
$$
\partial_t u_0^2 - \nu \Delta u_0^2 = - (u \cdot \nabla u^2)_0 + \partial_y \Delta^{-1} (\partial_i u^j \partial_j u^i)_0.
$$
An energy estimate gives
\begin{align*}
& \frac{1}{2} \| u_0^2(T) \|_{H^{N-1}}^2 + \nu \| \nabla_L u_0^2 \|_{L^2  H^{N-1}}^2 \\
& \qquad \qquad = \frac{1}{2} \| (u_{in}^2)_0 \|_{H^{N-1}}^2 + \int_0^T \int \langle D \rangle^{N-1} u_0^2 \langle D \rangle^{N-1} \left[ - (u \cdot \nabla u^2)_0 + \partial_y  \Delta^{-1} (\partial_i u^j \partial_j u^i)_0 \right]\,dV\,dt\\
& \qquad \qquad =  \frac{1}{2} \| (u_{in}^2)_0 \|_{H^{N-1}}^2 + \mathcal{T} + NLP.
\end{align*}
The transport term $\mathcal{T}$ can be treated as for $u_0^1$ in \S\ref{sec:U01low}; we omit the details. Turning to the nonlinear pressure term, it can be written, using that $u$ is divergence free, as
\begin{align*}
NLP & =  \int_0^T \int  \langle D \rangle^{N-1} u_0^2 \langle D \rangle^{N-1}  \partial_y  \Delta^{-1} \left[ \partial_i (u^j_0 \partial_j u^i_0) + \partial_i( u^j_{\neq} \partial_j u^i_{\neq} )_0 \right] \\
& = NLP_0 + NLP_{\neq}.
\end{align*}
In order to bound $NLP_{\neq}$, we use once again the remark that, due to the $X$ average,
\begin{align*}
(\partial_i u^j_{\neq} \partial_j u^i_{\neq} )_0 = \partial_{ij}(\bar{u}^i \bar{u}^j)_0 \mathbf{1}_{i \neq 1} \mathbf{1}_{j \neq 1}.  
\end{align*} 
Therefore, 
\begin{align*} 
NLP_{\neq} \lesssim \| u^2_0 \|_{L^\infty  H^{N-1}} \| \bar{u}_{\neq}^{2,3} \|_{L^2  H^{N-1}} \| \nabla_L \bar{u}^{2,3}_{\neq} \|_{L^2  H^{N-1}} \lesssim \epsilon^3 \nu^{-2/3}. 
\end{align*}
Since $i$ and $j$ can only be equal to $2$ or $3$, $NLP_0$ can be estimated by
\begin{align*}
NLP_0 & \lesssim \| u^2_0 \|_{L^2  H^{N-1}} \left( \| u^2_0 \|_{L^\infty  H^{N-1}} +  \| u^3_0 \|_{L^\infty  H^{N-1}} \right) \left( \|\nabla u^2_0 \|_{L^2  H^{N-1}} +  \|\nabla u^3_0 \|_{L^2  H^{N-1}} \right) \\
& \lesssim \epsilon^3 \nu^{-1}.  
\end{align*}
This gives the desired bound on $\| u_0^2 \|_{L^\infty  H^{N-1}}^2 + \nu \| \nabla_L u_0^2 \|_{L^2  H^{N-1}}^2$ for $\epsilon \nu^{-1}$ sufficiently small. 

\subsection{Uniform bound on $U^3_0$} 
As already explained above, we perform estimates on $u^3_0$, which solves
$$
\partial_t u_0^3 - \nu \Delta u_0^3 = - (u \cdot \nabla u^3)_0 + \partial_y  \Delta^{-1} (\partial_i u^j \partial_j u^i)_0
$$
An energy estimate gives
\begin{align*}
& \frac{1}{2} \| u_0^3(T) \|_{H^{N-1}}^2 + \nu \| \nabla u_0^3 \|_{L^2  H^{N-1}}^2 \\
& \qquad \qquad = \frac{1}{2} \| (u_{in}^3)_0 \|_{H^{N-1}}^2 + \int_0^T \int \langle D \rangle^{N-1} u_0^2 \langle D \rangle^{N-1} \left[ - (u \cdot \nabla u^3)_0 + \partial_z  \Delta^{-1}(\partial_i u^j \partial_j u^i)_0 \right]\,dV\,dt\\
& \qquad \qquad = \frac{1}{2} \| (u_{in}^3)_0 \|_{H^{N-1}}^2 + \mathcal{T} + NLP.
\end{align*}
The estimate on $\mathcal{T}$ is similar to that done on $u_0^1$ and hence is omitted for brevity. The estimate on $NLP$ requires a slight variant of what is done for $u_0^2$. 
First, 
\begin{align*}
NLP & =  \int_0^T \int  \langle D \rangle^{N-1} u_0^3 \langle D \rangle^{N-1}  \partial_z \partial_i \Delta^{-1} \left[ (u^j_0 \partial_j u^i_0) + ( u^j_{\neq} \partial_j u^i_{\neq} )_0 \right] \\
& = NLP_0 + NLP_{\neq}.
\end{align*}
The treatment of $NLP_{\neq}$ is the same as for $u_0^2$ and is hence omitted. 
Turn next to $NLP_0$. If $i = j = 3$, then at least two of the three factors must have a non-zero $z$ derivative and hence we have 
\begin{align*}
NLP_0 & \lesssim \norm{u_0^3}_{L^\infty H^{N-1}} \norm{\grad u_0^3}_{L^2 H^{N-1}}^2 + \norm{u_0^3}_{L^\infty H^{N-1}}\norm{u_0^2}_{L^2 L^{N-1}}\norm{\grad u_0^3}_{L^2 L^{N-1}} \\ & \quad + \norm{u_0^3}_{L^\infty H^{N-1}}\norm{u_0^2}_{L^2 H^{N-1}}\norm{\grad u_0^2}_{L^2 L^{N-1}} \\ 
& \lesssim \frac{\epsilon}{\nu} \epsilon^{2}, 
\end{align*} 
which suffices for $\epsilon \nu^{-1}$ sufficiently small. Notice that we used $\norm{u_0^2}_{L^2 H^{N-1}} \lesssim \nu^{-1/2}$; one way to deduce this is via incompressibility, $\norm{u_0^2}_{L^2 H^{N-1}} \leq \norm{\partial_z u_0^2}_{L^2 H^{N-1}}$. 
This completes all of the zero frequency velocity estimates. 

\section*{Acknowledgments}
The authors would like to thank Peter Constantin, Nick Trefethen, and Vlad Vicol for helpful discussions. The authors would like to especially thank Tej Ghoul for encouraging us to focus our attention on finite Reynolds number questions. The work of JB was in part supported by a Sloan fellowship and NSF grant DMS-1413177, the work of PG was in part supported by a Sloan fellowship and the NSF grant DMS-1101269, while the work of NM was in part supported by the NSF grant DMS-1211806. 

\appendix

\section{Commutation and elliptic estimates}

\label{sectionappendix}
\subsection{Commutator-like estimates}
In this section we outline some technical pointwise estimates on the Fourier multipliers we are employing; these essentially become product-rule type estimates in practice. 

\begin{lemma}[Commutator-type estimate on $m$] \label{lem:mcomm}  
For all $t,k,l,\eta,\xi$ there holds 
\begin{align*}
m(t,k,\eta,l) & \lesssim \jap{\eta-\xi,l-l^\prime}^2 m(t,k,\xi,l^\prime). 
\end{align*}
\end{lemma} 
\begin{proof} 
Clearly it suffices to show that
$$
\frac{m(t,k,\eta,l)}{m(t,k,\eta',l')} \lesssim \langle l-l' \rangle^2 +  \langle \eta-\eta' \rangle^2.
$$
Due to the definition of $m$, this estimate is proved by distinguishing several cases, depending on how $t$, $\frac{\eta}{k}$, and $\frac{\eta'}{k}$ compare. Since all these cases are fairly similar, we will only consider three of them for brevity:
\begin{itemize}
\item  If $\frac{\eta}{k} >0$, $\frac{\eta'}{k} >0$, $t > \frac{\eta}{k} + 1000 \nu^{-1/3}$ and $t > \frac{\eta'}{k} + 1000 \nu^{-1/3}$, 
\begin{align*}
\frac{m(t,k,\eta,l)}{m(t,k,\eta',l')} = \frac{k^2 + l^2}{k^2 + (l')^2} \frac{k^2 + (1000 \nu^{-1/3} k)^2 + (l')^2}{k^2 + (1000 \nu^{-1/3} k)^2 + l^2} \lesssim \frac{1 + L^2}{1+ (L')^2} \frac{\nu^{-2/3} + (L')^2}{\nu^{-2/3} + L^2},
\end{align*}
where we set $L = \frac{l}{k}$ and $L' = \frac{l'}{k}$. Since
$$
\frac{1 + L^2}{1+ (L')^2} \frac{\nu^{-2/3} + (L')^2}{\nu^{-2/3} + L^2} - 1 \lesssim \frac{\nu^{-2/3}(L^2 - (L')^2)}{\langle L' \rangle^2  (\nu^{-2/3} + L^2)},
$$
we deduce the desired bound.
\item If $\frac{\eta}{k} >0$, $\frac{\eta'}{k} >0$, and $t < \frac{\eta}{k}$ and $t > \frac{\eta'}{k} + 1000 \nu^{-1/3}$,
$$
\frac{m(t,k,\eta,l)}{m(t,k,\eta',l')} \lesssim \frac{\nu^{-2/3}+(L')^2}{1 + (L')^2} \lesssim 1 + \nu^{-2/3} \lesssim \langle \eta - \eta' \rangle^2.
$$
\item  If $0 < \frac{\eta}{k} < t < \frac{\eta}{k} + 1000 \nu^{-1/3}$ and $0 < \frac{\eta'}{k} < t < \frac{\eta'}{k} + 1000 \nu^{-1/3}$,
$$
\frac{m(t,k,\eta,l)}{m(t,k,\eta',l')} = \frac{(1+L^2)(1 + (t-H')^2 + (L')^2)}{(1 + (t-H)^2 + L^2)(1 + (L')^2)},
$$
where we set $L = \frac{l}{k}$, $L' = \frac{l'}{k}$, $H = \frac{\eta}{k}$, and $H' = \frac{\eta'}{k}$. Since
\begin{align*}
&\frac{(1+L^2)(1 + (t-H')^2 + (L')^2)}{(1 + (t-H)^2 + L^2)(1 + (L')^2)} - 1  = \frac{(t-H)^2( L^2 - (L')^2) + L^2 (2t - H - H')(H-H')}{(1 + (t-H)^2 + L^2)(1 + (L')^2)} \\
& \qquad \qquad \lesssim \frac{|(L')^2 - L^2|}{1 + (L')^2} + \frac{L^2 |t-H||H-H'|}{(1 + (t-H)^2 + L^2)(1 + (L')^2)} + \frac{L^2 |H-H'|^2}{(1 + (t-H)^2 + L^2)(1 + (L')^2)} \\
& \qquad \qquad \lesssim \langle L - L' \rangle^2 + \langle H - H' \rangle^2,
\end{align*}
the desired bound follows.
\end{itemize}
\end{proof}

\begin{lemma}[Commutator-type estimate on $\sqrt{-\dot{M} M}$] \label{lem:dotMcomm}  
For all $t,k,l,l',\eta,\eta'$ there holds the following estimates
\begin{subequations} 
\begin{align}
\sqrt{-\dot{M}^0 M^0}(t,k,\eta,l) & \lesssim \jap{\eta-\eta',l-l'} \sqrt{-\dot{M}^0 M^0}(t,k,\eta',l^\prime) \\ 
\sqrt{-\dot{M}^1 M^1}(t,k,\eta,l) & \lesssim \jap{\eta-\eta',l-l'}^{3/2} \sqrt{-\dot{M}^1 M^1}(t,k,\eta',l^\prime) \\ 
\sqrt{-\dot{M}^2 M^2}(t,k,\eta,l) & \lesssim \jap{\nu^{1/3}\abs{\eta-\eta'}}^{(1+\kappa)/2} \sqrt{-\dot{M}^2 M^2}(t,k,\eta',l^\prime). 
\end{align}
\end{subequations} 
\end{lemma} 
\begin{proof} 
All of these estimates follow immediately from the definition of $M^i$ in \S\ref{sec:FM}. 
\end{proof} 

\subsection{Elliptic lemmas} \label{sec:PEL}
This section concerns estimates on $\Delta_t^{-1}$ involving the Fourier multipliers $m$, $\dot{M}$, and $\grad_L$.  
All of these estimates are based on comparing $\Delta_t^{-1}$ to $\Delta_L^{-1}$. 
The estimates here differ from the analogous estimates employed previously in \cite{BM13,BMV14,BGM15I,BGM15II} due to the much lower regularity and the fact that the coefficients are a little smaller here (relative to the primary unknowns).  

The first estimate concerns inverting $\Delta_t$ at zero $X$ frequencies. 
\begin{lemma}[Zero mode elliptic regularity] \label{lem:ZeroModePEL} 
Under the bootstrap hypotheses, for $\epsilon \nu^{-1}$ sufficiently small, there holds for any $1 < s \leq N$, 
\begin{subequations}
\begin{align}
\norm{\Delta_t^{-1} \phi_0}_{H^{s+2}} & \lesssim \norm{\phi_0}_{H^{s}} + \norm{\Delta_t^{-1} \phi_0}_{L^2} \label{ineq:ExLow0}\\ 
\norm{\grad \Delta_t^{-1} \phi_0}_{H^{s+1}} & \lesssim \norm{\grad \phi_0}_{H^{s-1}} + \norm{\grad \Delta_t^{-1}\phi_0}_{L^2}  \label{ineq:Grad0}\\ 
\norm{\Delta \Delta_t^{-1} \phi_0}_{H^s} & \lesssim \norm{\phi_0}_{H^s} + \epsilon\nu^{-1}\norm{\grad \Delta_t^{-1} \phi_0}_{L^2}. \label{ineq:0BasicPEL}
\end{align}
\end{subequations}
\end{lemma} 
\begin{proof} 
Consider \eqref{ineq:ExLow0}. First, 
\begin{align*}
\norm{\Delta_t^{-1} \phi_0}_{H^{s+2}} \leq \norm{\Delta \Delta_t^{-1} \phi_0}_{H^{s}} + \norm{\Delta_t^{-1} \phi_0}_{L^2}.  
\end{align*}
From the definition of $\Delta_t$ we have  
\begin{align}
\norm{\Delta \Delta_t^{-1} \phi_0}_{H^{s}} & \leq \norm{\phi_0}_{H^{s}} + \norm{ \left( G \partial_{YY} \Delta_t^{-1} \phi_0\right)}_{H^{s}} \nonumber \\ & \quad  + 2\norm{ \left(\psi_z\partial_{YZ} \Delta_t^{-1}\phi_0\right)}_{H^{s}} + \norm{ \left(\Delta_t C\partial_{Y} \Delta_t^{-1}\phi_0\right)}_{H^{s}}. \label{ineq:phi0est}
\end{align}
For the first two error terms we simply have, 
\begin{align*}
\norm{ \left( G \partial_{YY} \Delta_t^{-1} \phi_0\right)}_{H^{s}} + 2\norm{ \left(\psi_z\partial_{YZ} \Delta_t^{-1}\phi_0\right)}_{H^{s}} & \lesssim \norm{\grad C}_{H^s} \norm{\Delta \Delta_t^{-1}\phi_0}_{H^s} \\ 
& \lesssim \epsilon \nu^{-1}\norm{\Delta \Delta_t^{-1}\phi_0}_{H^s}, 
\end{align*}
which is then absorbed on the left-hand side of \eqref{ineq:phi0est} for $\epsilon \nu^{-1} \ll 1$. 
For the last error term we use the product rule and a frequency decomposition: 
\begin{align*}
\norm{ \left(\Delta_t C\partial_{Y} \Delta_t^{-1}\phi_0\right)}_{H^{s}} & \leq \norm{ \left(\Delta_t C\partial_{Y}P_{\leq 1}\Delta_t^{-1}\phi_0\right)}_{H^{s}} + \norm{ \left(\Delta_t C\partial_{Y} P_{> 1} \Delta_t^{-1} \phi_0\right)}_{H^{s}} \\ 
& \lesssim \norm{C}_{H^{s+2}}\norm{\Delta_t^{-1}\phi_0}_{L^2} + \norm{C}_{H^{s+2}}\norm{\Delta \Delta_t^{-1}\phi_0}_{H^s}.  
\end{align*}
The latter term is again absorbed on the left-hand side of \eqref{ineq:phi0est} for $\epsilon \nu^{-1} \ll 1$ (since $s \leq N$), the former is consistent with the right-hand side of \eqref{ineq:ExLow0}. 
Estimate \eqref{ineq:Grad0} follows by similar considerations. 

Estimate \eqref{ineq:0BasicPEL} follows from 
\begin{align*}
\norm{\Delta \Delta_t^{-1} \phi_0}_{H^{s}} & \leq \norm{\phi_0}_{H^{s}} + \norm{G \partial_{YY} \Delta_t^{-1} \phi_0}_{H^{s}} + 2\norm{\psi_z\partial_{YZ} \Delta_t^{-1}\phi_0}_{H^{s}} + \norm{\Delta_t C\partial_{Y}\Delta_t^{-1}\phi_0}_{H^{s}}.  \\ 
& \lesssim  \norm{\phi_0}_{H^{s}} + \epsilon \nu^{-1}\norm{\Delta \Delta_t^{-1} \phi_0}_{H^{s}} + \epsilon \nu^{-1}\norm{\grad \Delta_t^{-1}\phi_0}_{H^{s}} \\ 
& \lesssim  \norm{\phi_0}_{H^{s}} + \epsilon \nu^{-1}\norm{\Delta \Delta_t^{-1} \phi_0}_{H^{s}} + \epsilon \nu^{-1}\norm{P_{< 1} \left(\grad \Delta_t^{-1}\phi_0\right)}_{H^{s}} \\  
& \lesssim  \norm{\phi_0}_{H^{s}} + \epsilon \nu^{-1}\norm{\Delta \Delta_t^{-1} \phi_0}_{H^{s}} + \epsilon \nu^{-1}\norm{\grad \Delta_t^{-1}\phi_0}_{L^2}. 
\end{align*}
The second term is then absorbed on the left hand side. 
\end{proof} 

\begin{lemma} \label{lem:BasicPEL} 
Under the bootstrap hypotheses, for $\epsilon \nu^{-4/3}$ sufficiently small, there holds for any $\alpha \in [0,1]$, $3/2 < s \leq N$,   
\begin{align*}
\norm{m^\alpha \Delta_L \Delta_t^{-1} \phi_{\neq}}_{H^s} \lesssim \norm{m^\alpha \phi_{\neq}}_{H^s}. 
\end{align*} 
\end{lemma} 
\begin{proof} 
Writing $P = \Delta_t^{-1}\phi$ gives 
\begin{align*}
\Delta_L P = \phi - G\partial_{Y}^L\partial_{Y}^LP - 2\psi_z \partial_Y^L \partial_Z P - \Delta_t C \partial_Y^L P. 
\end{align*}
Applying $\jap{D}^s m^\alpha$ to both sides gives
\begin{align}
\norm{m^\alpha \Delta_L \Delta_t^{-1} \phi_{\neq}}_{H^s} \lesssim \norm{m^\alpha \phi_{\neq}}_{H^s} + \sum_{j = 1}^3 \mathcal{E}_j, \label{ineq:basicPELstp} 
\end{align}
where
$$
\mathcal{E}_1 = \| m^\alpha G\partial_{Y}^L\partial_{Y}^LP \|_{H^s}, \quad \mathcal{E}_2 = \| m^\alpha \psi_z \partial_{Y}^L\partial_{Z}P \|_{H^s}, \quad \mathcal{E}_3 = \| m^\alpha  \Delta_t C \partial_Y^L P \|_{H^s}.
$$
By Lemma \ref{lem:mcomm} we can deduce 
\begin{align*}
\mathcal{E}_{1} + \mathcal{E}_{2}  & \lesssim \norm{\grad C}_{H^{s+\min(2\alpha,1)}}\norm{m^{\min(1/2,\alpha)} \Delta_L P}_{H^{s}} \\
& \lesssim \nu^{-\max(0,2\alpha-1)/3 } \norm{\grad C}_{H^{s+\min(2\alpha,1)}}\norm{m^{\alpha} \Delta_L P}_{H^{s}}. 
\end{align*}
However, since $s \leq N$, by the bootstrap hypotheses,  
\begin{align*}
\nu^{-\max(0,2\alpha-1)/3 } \norm{\grad C}_{H^{s+\min(2\alpha,1)}} \lesssim \epsilon \nu^{-4/3} \ll 1,  
\end{align*}
and this error can be absorbed by the LHS of the estimate in \eqref{ineq:basicPELstp}. 
For $\mathcal{E}_3$ we apply \eqref{ineq:mDelTrick}: 
\begin{align*}
\mathcal{E}_3 & \lesssim \norm{\grad^2 C}_{H^{s}}\norm{\grad_L P}_{H^{s}} \lesssim \norm{\grad C}_{H^{s+1}}\norm{m^{\min(1/2,\alpha)}\Delta_L P}_{H^{s}}, 
\end{align*}
and from here we may proceed as in $\mathcal{E}_{1,2}$ above. 
\end{proof}

\begin{lemma} \label{lem:UNLP} 
Under the bootstrap hypotheses, for $\epsilon \nu^{-4/3}$ sufficiently small, there holds for any $3/2 < s \leq N$,   
\begin{align*}
\norm{ \Delta_t^{-1} \partial_i^t \partial_j^t \phi_{\neq}}_{H^s} \lesssim \norm{\phi_{\neq}}_{H^s}. 
\end{align*} 
\end{lemma} 
\begin{proof} 
The first observation is that $\Delta_t^{-1}$ and $\partial_i^t \partial_j^t$ \emph{commute} -- indeed one need only un-do the coordinate transform, commute them as Fourier multipliers, and then re-do the coordinate transform. Therefore the estimate is the same as 
\begin{align*}
\norm{\partial_i^t \partial_j^t \Delta_t^{-1}  \phi_{\neq}}_{H^s} \lesssim \norm{\phi_{\neq}}_{H^s}. 
\end{align*}
By the $L^\infty H^{N+2}$ control on $C$ and the projection to non-zero frequencies, we have 
\begin{align*}
\norm{\partial_i^t \partial_j^t \Delta_t^{-1}  \phi_{\neq}}_{H^s} \lesssim \norm{\Delta_L \Delta_t^{-1}  \phi_{\neq}}_{H^s}. 
\end{align*}
Hence, the desired result now follows from Lemma \ref{lem:BasicPEL}. 
\end{proof} 

\begin{lemma}\label{lem:EDPEL} 
Under the bootstrap hypotheses, for $\epsilon \nu^{-4/3}$ sufficiently small, there holds for any $\alpha \in [0,1]$, $3/2 < s \leq N$,  
\begin{align}
\norm{\grad_L m^\alpha \Delta_L \Delta_t^{-1} \phi_{\neq}}_{H^s} \lesssim \norm{\grad_L m^\alpha \phi_{\neq}}_{H^s} + \norm{\grad C}_{H^{s+2}} \nu^{-\max(0,2\alpha-1)/3} \norm{m^\alpha \phi_{\neq}}_{H^{3/2+}}. \label{ineq:EDPEL}
\end{align}
\end{lemma}
\begin{remark} \label{rmk:EDPEL}
For $s \leq N-1$, the second term in \eqref{ineq:EDPEL} can be removed.  
\end{remark} 
\begin{proof} 
Define $P = \Delta_t^{-1} \phi_{\neq}$. 
As in the proof of Lemma \ref{lem:BasicPEL} above, 
\begin{align}
\norm{\grad_L m^\alpha \Delta_L \Delta_t^{-1} \phi_{\neq}}_{H^s} \lesssim \norm{\grad_L m^\alpha \phi_{\neq}}_{H^s} + \sum_{j = 1}^3 \mathcal{E}_j, \label{EDPELintermed}
\end{align}
where
\begin{align*} 
\mathcal{E}_1 = \| \grad_L m^\alpha G\partial_{Y}^L\partial_{Y}^LP \|_{H^s}, \quad \mathcal{E}_2 = \| \grad_L  m^\alpha \psi_z \partial_{Y}^L\partial_{Z}P \|_{H^s}, \quad \mathcal{E}_3 = \| \grad_L m^\alpha  \Delta_t C \partial_Y^L P \|_{H^s}.
\end{align*} 
By Leibniz's rule, a paraproduct decomposition, and Lemma \ref{lem:mcomm}, 
\begin{align*}
\mathcal{E}_{1,2} & \lesssim \norm{\grad C}_{H^{1+2\alpha+}}\norm{\grad_L m^\alpha \Delta_L P}_{H^s}  + \norm{\grad C}_{H^{2+\min(2\alpha,1)}}\norm{m^{\min(1/2,\alpha)} \Delta_L P}_{H^s} \\  
& \quad + \norm{\grad C}_{H^{s+\min(2\alpha,1)}}\norm{\grad_L m^{\min(1/2,\alpha)} \Delta_L P}_{H^{3/2+}} + \norm{\grad C}_{H^{s+1+\min(2\alpha,1)}}\norm{m^{\min(1/2,\alpha)}\Delta_L P}_{H^{3/2+}}. 
\end{align*}
By \eqref{mlowbound}, \eqref{ineq:mDelTrick}, and $N > 2$, we have 
\begin{align*}
\mathcal{E}_{1,2} & \lesssim \left(\norm{C}_{H^{N+2}} + \nu^{-\max(0,2\alpha-1)/3} \norm{\grad C}_{H^{s+\min(2\alpha,1)}} \right)\norm{\grad_L m^\alpha \Delta_L P}_{H^s} \\ & \quad + \nu^{-\max(0,2\alpha-1)/3} \norm{\grad C}_{H^{3+}}\norm{m^\alpha \Delta_L P}_{H^s}  + \nu^{-\max(0,2\alpha-1)/3} \norm{\grad C}_{H^{s+2}}\norm{ m^{\alpha} \Delta_L P}_{H^{3/2+}} \\ 
& \lesssim \epsilon \nu^{-4/3} \norm{\grad_L m^\alpha \Delta_L P}_{H^s} + \nu^{-\max(0,2\alpha-1)/3} \norm{\grad C}_{H^{s+2}}\norm{m^\alpha \Delta_L P}_{H^s}.  
\end{align*}
By $\epsilon \nu^{-4/3}$ sufficiently small, the first term can be absorbed on the right hand side of \eqref{EDPELintermed} and the second term is consistent with the stated \eqref{ineq:EDPEL} by Lemma \ref{lem:BasicPEL}. 

For $\mathcal{E}_3$ we apply again Leibniz's rule 
and \eqref{ineq:mDelTrick}: 
\begin{align*}
\mathcal{E}_3 & \lesssim \norm{\Delta C}_{H^{1+2\alpha+}}\norm{m^\alpha  \Delta_L P}_{H^s} + \norm{\Delta C}_{H^{2+}}\norm{\grad_L P}_{H^s} \\ 
& \quad + \norm{\Delta C}_{H^{s}}\norm{\Delta_L P}_{H^{3/2+}} + \norm{\Delta C}_{H^{s+1}}\norm{\grad_L P}_{H^{3/2+}} \\ 
& \lesssim \norm{C}_{H^{N+2}}\norm{m^\alpha \grad_L \Delta_L P}_{H^s}  + \norm{\Delta C}_{H^{s}}\norm{m^{\min(1/2,\alpha)}\grad_L \Delta_L P}_{H^{3/2+}} \\ & \quad + \norm{\Delta C}_{H^{s+1}}\norm{m^{\min(1/2,\alpha)}\Delta_L P}_{H^{3/2+}} \\  
& \lesssim \nu^{-\max(0,2\alpha-1)/3}\norm{C}_{H^{N+2}}\norm{m^\alpha \grad_L \Delta_L P}_{H^s} + \norm{\Delta C}_{H^{s+1}}\norm{m^{\min(1/2,\alpha)}\Delta_L P}_{H^{3/2+}} \\  
& \lesssim \epsilon \nu^{-4/3} \norm{m^{\alpha}\grad_L \Delta_L P}_{H^{s}} + \nu^{-\max(0,2\alpha-1)/3} \norm{\grad C}_{H^{s+2}}\norm{m^{\alpha}\Delta_L P}_{H^{3/2+}}. 
\end{align*}
As above, for $\epsilon \nu^{-4/3}$ sufficiently small, the first term can be absorbed on the right hand side of \eqref{EDPELintermed} and the second term is consistent with the stated \eqref{ineq:EDPEL} by Lemma \ref{lem:BasicPEL}. 
Also note that if $s+3 \leq N+2$, then the latter term can be absorbed on the right hand side of \eqref{EDPELintermed} for $\epsilon \nu^{-4/3} \ll 1$, as claimed in Remark \ref{rmk:EDPEL}. 
\end{proof} 

\begin{lemma}\label{lem:dotMPEL} 
Suppose $i \in \set{0,1,2}$. Under the bootstrap hypotheses, for $\epsilon \nu^{-3/2}$ sufficiently small, there holds for any $\alpha \in [0,1]$, $0 \leq s \leq N$,  
\begin{align}
\norm{\sqrt{-\dot{M}^i M^i} m^\alpha \Delta_L \Delta_t^{-1} \phi_{\neq}}_{H^s} \lesssim \norm{\sqrt{-\dot{M} M} m^\alpha \phi_{\neq}}_{H^s} + (\epsilon \nu^{-3/2}) \nu^{1/2} \norm{\grad_L m^{\alpha} \Delta_L \Delta_t^{-1}\phi_{\neq}}_{H^s}. \label{ineq:dotMPEL}
\end{align}
\end{lemma}
\begin{proof} 
Writing $P = \Delta_t^{-1}\phi$, applying the multiplier $\sqrt{-\dot{M}^i M^i}\jap{D}^s m^\alpha$ to both sides of the equation, and taking $L^2$ norms, gives 
\begin{align}
\norm{\sqrt{-\dot{M}^i M^i} m^\alpha \Delta_L \Delta_t^{-1} \phi_{\neq}}_{H^s} & \lesssim \norm{\sqrt{-\dot{M} M} m^\alpha \phi_{\neq}}_{H^s} + \sum_{j = 1}^3 \mathcal{E}_j, \label{ineq:MdotPELSetup}
\end{align}
where
\begin{align*} 
\mathcal{E}_1 = \| \sqrt{-\dot{M} M} m^\alpha G\partial_{Y}^L\partial_{Y}^LP \|_{H^s}, \quad \mathcal{E}_2 = \| \sqrt{-\dot{M} M}  m^\alpha \psi_z \partial_{Y}^L\partial_{Z}P \|_{H^s}, \quad \mathcal{E}_3 = \| \sqrt{-\dot{M} M} m^\alpha  \Delta_t C \partial_Y^L P \|_{H^s}.
\end{align*} 
Similar to the arguments employed in the other elliptic lemmas, via a paraproduct decomposition, Lemma \ref{lem:mcomm}, and Lemma \ref{lem:dotMcomm}, we get
\begin{align*}
\mathcal{E}_{1} + \mathcal{E}_{2} & \lesssim \nu^{-\max(0,2\alpha-1)/3}\norm{\grad C}_{H^{5/2 + \min(1,2\alpha)+}}\norm{\sqrt{-\dot{M} M} m^\alpha \Delta_L P}_{H^s} \\ & \quad + \nu^{-\max(0,2\alpha-1)/3}\norm{\grad C}_{H^{s+\min(1,2\alpha)}}\norm{m^{\alpha} \Delta_L P}_{H^{3/2+}}.  
\end{align*}
However, by Lemma \ref{lem:postmultED} we have 
\begin{align*}
\norm{m^{\alpha} \Delta_L P}_{H^{3/2+}} & \lesssim \nu^{-1/6}\left(\norm{\sqrt{-\dot{M} M}m^{\alpha} \Delta_L P}_{H^s} + \nu^{1/2}\norm{\grad_L m^{\alpha} \Delta_L P}_{H^s}\right), 
\end{align*}
which implies (along with $N \geq \max(s,5/2+)$), 
\begin{align*}
\mathcal{E}_{1} + \mathcal{E}_{2} \lesssim (\epsilon \nu^{-3/2}) \norm{\sqrt{-\dot{M} M} m^\alpha \Delta_L P}_{H^s} + (\epsilon \nu^{-3/2})\nu^{1/2}\norm{\grad_L m^{\alpha} \Delta_L P}_{H^s}. 
\end{align*}
For $\epsilon \nu^{-3/2}$ sufficiently small, the first term is absorbed in the LHS of \eqref{ineq:MdotPELSetup} whereas the latter term is consistent with \eqref{ineq:dotMPEL}.  

Consider next the error term $\mathcal{E}_3$, which by a paraproduct decomposition, Lemma \ref{lem:postmultED}, \eqref{ineq:mDelTrick}, and the lower bound on $m$,  
\begin{align*}
\mathcal{E}_3 & \lesssim \norm{\Delta C}_{H^{5/2 +}} \norm{\sqrt{-\dot{M} M} \grad_L P}_{H^s} + \norm{\grad^2 C}_{H^{s}}\norm{\grad_L P}_{3/2+} \\ 
& \lesssim \nu^{-\max(0,2\alpha-1)/3}\norm{C}_{H^{N+2}} \norm{\sqrt{-\dot{M} M} m^{\alpha} \Delta_L P}_{H^s} + \nu^{-\max(0,2\alpha-1)/3} \norm{\grad C}_{H^{s+1}}\norm{m^\alpha \Delta_L P}_{3/2+},
\end{align*} 
from which the result follows in the same way as for $\mathcal{E}_{1,2}$. 
\end{proof} 

\bibliographystyle{abbrv} 
\bibliography{eulereqns_vlad}

\end{document}